\documentclass{amsart}

\usepackage{amssymb, amscd, amsthm}

\newtheorem{MainThm}{Theorem}

\newtheorem{Thm}{Theorem}[subsection]
\newtheorem{Prop}[Thm]{Proposition}
\newtheorem{Cor}[Thm]{Corollary}
\newtheorem{Lem}[Thm]{Lemma}
\newtheorem{Sublem}[Thm]{Sublemma}
\newtheorem{Def}[Thm]{Definition}
\theoremstyle{remark}
\newtheorem{Rmk}[Thm]{Remark}
\newtheorem*{Notation}{Notation}
\newtheorem*{Ack}{Acknowledgments}

\newcommand{\Order}{\mathcal{O}}
\newcommand{\into}{\hookrightarrow}
\newcommand{\onto}{\twoheadrightarrow}
\newcommand{\isomto}{\overset{\sim}{\to}}

\newcommand{\compose}{\circ}
\newcommand{\tensor}{\otimes}
\newcommand{\dtensor}{\tensor^{L}}
\newcommand{\closure}[1]{\overline{#1}}
\newcommand{\N}{\mathbb{N}}
\newcommand{\Z}{\mathbb{Z}}
\newcommand{\Q}{\mathbb{Q}}
\newcommand{\F}{\mathbb{F}}
\newcommand{\Affine}{\mathbb{A}}
\newcommand{\ab}{\mathrm{ab}}

\newcommand{\ur}{\mathrm{ur}}
\newcommand{\et}{\mathrm{et}}
\newcommand{\zar}{\mathrm{zar}}

\newcommand{\fppf}{\mathrm{fppf}}
\newcommand{\rat}{\mathrm{rat}}
\newcommand{\gen}{\mathrm{gen}}
\newcommand{\perf}{\mathrm{perf}}
\newcommand{\sch}{\mathrm{sch}}
\newcommand{\id}{\mathrm{id}}
\newcommand{\incl}{\mathrm{incl}}
\newcommand{\proj}{\mathrm{proj}}
\newcommand{\set}{\mathrm{set}}
\newcommand{\pre}{\mathrm{pre}}
\newcommand{\fl}{\mathrm{fl}}
\newcommand{\fin}{\mathrm{fin}}
\newcommand{\op}{\mathrm{op}}
\newcommand{\Alg}{\mathrm{Alg}}
\newcommand{\Pro}{\mathrm{P}}
\newcommand{\Loc}{\mathrm{L}}
\newcommand{\PDual}{\mathrm{PD}}
\newcommand{\CDual}{\mathrm{CD}}
\newcommand{\Gm}{\mathbf{G}_{m}}
\newcommand{\Ga}{\mathbf{G}_{a}}
\newcommand{\boundary}{\partial}

\newcommand{\ind}{\mathrm{ind}}
\newcommand{\pro}{\mathrm{pro}}
\newcommand{\dlog}{\mathrm{dlog}}
\newcommand{\dirlim}{\varinjlim}
\newcommand{\invlim}{\varprojlim}
\newcommand{\ideal}[1]{\mathfrak{#1}}
\newcommand{\DM}{\mathrm{DM}}
\newcommand{\alg}[1]{\mathbf{#1}}
\mathchardef\mhyphen="2D

\DeclareMathOperator{\Gal}{Gal}
\DeclareMathOperator{\Hom}{Hom}

\DeclareMathOperator{\Ker}{Ker}

\DeclareMathOperator{\Ext}{Ext}

\DeclareMathOperator{\Spec}{Spec}
\DeclareMathOperator{\Ab}{Ab}
\DeclareMathOperator{\Set}{Set}
\DeclareMathOperator{\Res}{Res}
\DeclareMathOperator{\Map}{Map}
\DeclareMathOperator{\Frac}{Frac}
\DeclareMathOperator{\Tor}{Tor}
\DeclareMathOperator{\GrAb}{GrAb}
\DeclareMathOperator{\DGAb}{DGAb}
\let\Im\relax
\DeclareMathOperator{\Im}{Im}
\DeclareMathOperator{\sh}{\mathsf{sh}}
\DeclareMathOperator{\sheafhom}{\alg{Hom}}
\DeclareMathOperator{\sheafext}{\alg{Ext}}
\DeclareMathOperator{\Pic}{Pic}

\newcommand{\BetweenThmAndList}{\leavevmode}

\title[Duality for local fields]
{Duality for local fields and sheaves on the category of fields}
\author{Takashi Suzuki}
\address{
	Department of Mathematics, Chuo University,
	1-13-27 Kasuga, Bunkyo-ku, Tokyo 112-8551, Japan
}
\email{tsuzuki@gug.math.chuo-u.ac.jp}
\date{January 27, 2021}

\begin{document}

\begin{abstract}
	Duality for complete discrete valuation fields with perfect residue field
	with coefficients in (possibly $p$-torsion) finite flat group schemes
	was obtained by B\'egueri, Bester and Kato.
	In this paper, we give another formulation and proof of this result.
	We use the category of fields and a Grothendieck topology on it.
	This simplifies the formulation and proof
	and reduces the duality to classical results on Galois cohomology.
	A key point is that the resulting site correctly captures
	extension groups between algebraic groups.
\end{abstract}

\maketitle

\tableofcontents


\section{Introduction}

\subsection{Aim of the paper}

Let $K$ be a complete discrete valuation field
with algebraically closed (more generally, perfect) residue field $k$ of characteristic $p > 0$.
Duality for $K$ with coefficients in finite flat group schemes
(with torsion not necessarily prime to $p$)
was obtained by B\'egueri (\cite{Beg81}), Bester (\cite{Bes78})
and Kato (unpublished, but announcements
in a much more general setting can be found in \cite{Kat86}, \cite[\S 3.3]{Kat91}).
This is a generalization of Serre's local class field theory (\cite{Ser61})
in the style of local Tate duality.
A summary of the work of B\'egueri and Bester can be found in Milne's book \cite[III, \S 4, \S 10]{Mil80}.
The hard part of this duality is how to give a nice geometric structure for cohomology groups.

In this paper, we give another formulation and proof of this result.
Our method is simple and straightforward,
requiring only classical results on Galois cohomology of such discrete valuation fields
written for example in \cite{Ser79} or \cite{Ser02}
together with Serre's local class field theory,
once we define a certain Grothendieck site as a set-up for the duality
and establish its basic site-theoretic properties.
The underlying category of the site we use is
the \emph{category of fields} (possibly transcendental) over $k$,
in contrast to other usual sites whose underlying categories are categories of rings or schemes.
An easy fact is that
the cohomology theory of our site is simple:
it is essentially Galois cohomology.
A difficult fact is that
extensions between sheaves on our site are rich:
they correctly capture extensions between algebraic groups.

In \cite{Suz20},
we apply our formulation and results to
Grothendieck's conjecture on the special fibers of abelian varieties over $K$ that he made in SGA7.


\subsection{Main results}

Now we formulate here a relatively simple part of the duality.
Let $k$ be a perfect field of characteristic $p > 0$.
We say that a $k$-algebra is \emph{rational}
if it is a finite direct product of
the perfect closures of finitely generated fields over $k$.
We denote the category of rational $k$-algebras by $k^{\rat}$.
This is essentially the category of (perfect) fields.
Give it the \'etale topology and denote the resulting site by $\Spec k^{\rat}_{\et}$.
We call this site the \emph{rational \'etale site} of $k$.

Let $K$ be a complete discrete valuation field
with ring of integers $\Order_{K}$ and residue field $k$.
We denote by $W$ the affine ring scheme of Witt vectors of infinite length.
The ring $\Order_{K}$ has a natural structure of a $W(k)$-algebra of profinite length
(which factors through the residue field $k$ when $K$ has equal characteristic).
We denote by $\hat{\tensor}$ the completed tensor product.
We define sheaves of rings on the site $\Spec k^{\rat}_{\et}$ by assigning to each $k' \in k^{\rat}$,
	\[
			\alg{O}_{K}(k')
		=
			W(k') \hat{\tensor}_{W(k)} \Order_{K},
		\quad
			\alg{K}(k')
		=
			\alg{O}_{K}(k') \tensor_{\Order_{K}} K.
	\]
Note that if $k'$ has only one direct factor, then
$\alg{K}(k')$ is the complete discrete valuation field
obtained from $K$ by extending its residue field from $k$ to $k'$.
The sheaves of invertible elements of $\alg{O}_{K}$ and $\alg{K}$ are denoted by
$\alg{U}_{K}$ and $\alg{K}^{\times}$, respectively.

We define a category $K_{\et} / k^{\rat}$ as follows.
An object is a pair $(L, k_{L})$,
where $k_{L} \in k^{\rat}$ and $L$ is an \'etale $\alg{K}(k_{L})$-algebra.
A morphism $(L, k_{L}) \to (L', k_{L'})$ consists of
a $k$-algebra homomorphism $k_{L} \to k_{L'}$ and a ring homomorphism $L \to L'$
such that the diagram
	\[
		\begin{CD}
				\alg{K}(k_{L})
			@>>>
				\alg{K}(k_{L'})
			\\
			@VVV
			@VVV
			\\
				L
			@>>>
				L'
		\end{CD}
	\]
commutes.
The composite of two morphisms is defined in an obvious way.
We say that a morphism $(L, k_{L}) \to (L', k_{L'})$ is \emph{\'etale}
if $k_{L} \to k_{L'}$ (and hence $L \to L'$) is \'etale.
Then we can define the \emph{relative \'etale site of $K$ over $k$},
denoted by $\Spec K_{\et} / k^{\rat}_{\et}$,
to be the category $K_{\et} / k^{\rat}$
endowed with the topology whose covering families
over an object $(L, k_{L}) \in K_{\et} / k^{\rat}$
are finite families $\{(L_{i}, k_{L_{i}})\}$ of objects \'etale over $(L, k_{L})$
with $\prod_{i} L_{i}$ faithfully flat over $L$.
The global version of this type of relative sites was defined by Artin-Milne
(\cite[\S 3]{AM76}; ``an auxiliary site $X / S_{\perf}$'').

The functor
	$
			k^{\rat}
		\to
			K_{\et} / k^{\rat}
	$,
	$
			k'
		\mapsto
			(\alg{K}(k'), k')
	$
defines a morphism of sites
	\[
			\pi
		\colon
			\Spec K_{\et} / k^{\rat}_{\et}
		\to
			\Spec k^{\rat}_{\et}.
	\]
We call this the \emph{structure morphism of $K$ over $k$}.
Note that there is no morphism $\Spec \Q_{p} \to \Spec \F_{p}$
on the level of schemes.
We denote by $\Ab(k^{\rat}_{\et})$ and $\Ab(K_{\et} / k^{\rat}_{\et})$
the categories of sheaves of abelian groups on $\Spec k^{\rat}_{\et}$
and on $\Spec K_{\et} / k^{\rat}_{\et}$, respectively,
and by $D(k^{\rat}_{\et})$ and $D(K_{\et} / k^{\rat}_{\et})$
their derived categories.
We denote
	\begin{gather*}
				\alg{\Gamma}(K_{\et}, \;\cdot\;)
			=
				\pi_{\ast},\;
				\alg{H}^{i}(K_{\et}, \;\cdot\;)
			=
				R^{i} \pi_{\ast}
			\colon
				\Ab(K_{\et} / k^{\rat}_{\et})
			\to
				\Ab(k^{\rat}_{\et}),
		\\
				R \alg{\Gamma}(K_{\et}, \;\cdot\;)
			=
				R \pi_{\ast}
			\colon
				D(K_{\et} / k^{\rat}_{\et})
			\to
				D(k^{\rat}_{\et}).
	\end{gather*}
For $A \in \Ab(K_{\et} / k^{\rat}_{\et})$,
the sheaf $\alg{H}^{i}(K_{\et}, A)$ is the \'etale sheafification of the presheaf
	\[
			k'
		\mapsto
			H^{i}(\alg{K}(k')_{\et}, A).
	\]
We denote by $\sheafhom_{k^{\rat}_{\et}}$ (resp.\ $\sheafext_{k^{\rat}_{\et}}^{i}$)
the sheaf-Hom (resp.\ the $i$-th sheaf-Ext)
and $R \sheafhom_{k^{\rat}_{\et}}$ its derived version.

Assume that $K$ has mixed characteristic.
For an \'etale group scheme $A$ over $K$,
we denote by $A^{\CDual}$ the Cartier dual of $A$
and by $A(1)$ the Tate twist.

\begin{MainThm} \label{thm: main theorem, duality}
	Assume that $K$ has mixed characteristic.
	\begin{enumerate}
		\item
			For any torsion \'etale group scheme $A$ over K ($p$-torsion being allowed),
			we have
				$
						\alg{H}^{i}(K_{\et}, A)
					=
						0
				$
			for $i \ne 0, 1$.
		\item
			The Kummer sequence gives an isomorphism
				\[
						\alg{H}^{1}(K_{\et}, \Q / \Z (1))
					=
						\alg{K}^{\times} \tensor_{\Z} \Q / \Z
				\]
			in $\Ab(k^{\rat}_{\et})$.
			By composing this with the valuation map
			$\alg{K}^{\times} \onto \Z$, we obtain a morphism
				\[
						\alg{H}^{1}(K_{\et}, \Q / \Z (1))
					\to
						\Q / \Z,
				\]
			which we call the \emph{trace map}.
		\item \label{ass: main theorem, cup pairing}
			Let $A$ be a finite \'etale group scheme over $K$.
			Consider the pairing
				\[
						R \alg{\Gamma}(K_{\et}, A^{\CDual})
					\times
						R \alg{\Gamma}(K_{\et}, A)
					\to
						\Q / \Z [-1]
				\]
			in $D(k^{\rat}_{\et})$
			given by the cup product and the trace map.
			The induced morphism
				\[
						R \alg{\Gamma}(K_{\et}, A^{\CDual})
					\to
						R \sheafhom_{k^{\rat}_{\et}}(R \alg{\Gamma}(K_{\et}, A), \Q / \Z [-1])
				\]
			is an isomorphism.
	\end{enumerate}
\end{MainThm}

For example, we have
	$
			\alg{H}^{2}(K_{\et}, \Q / \Z (1))
		=
			0
	$.
This is equivalent to the classical vanishing result of
$H^{2}(\alg{K}(k')_{\et}, \Q / \Z (1)) = $
the Brauer group of the complete discrete valuation field $\alg{K}(k')$
with any algebraically closed residue field $k'$ (over $k$).
Here it is essential to use the category $k^{\rat}$ of fields.
If $k'$ is replaced by a more general $k$-algebra $R$,
then the cohomology of a similarly defined ring $\alg{K}(R)$ is at least not classical.

The proof of Assertion \eqref{ass: main theorem, cup pairing} requires the following theorem
to understand extension groups over $\Spec k^{\rat}_{\et}$.
Note that the quotient $\alg{U}_{K} / (\alg{U}_{K})^{n}$ by $n$-th power elements
for any $n \ge 1$ is represented by a quasi-algebraic group of units over $k$
studied by Serre (\cite{Ser61}).
Recall that a quasi-algebraic group is the perfection (inverse limit along Frobenii)
of an algebraic group (\cite{Ser60}).
Let $\Alg / k$ be the category of commutative affine quasi-algebraic groups over $k$
and $\Ext_{\Alg / k}^{n}$ the $n$-th Ext functor for $\Alg / k$.

\begin{MainThm} \label{thm: main theorem, comparison of Ext}
	For any $A, B \in \Alg / k$ and any $n \ge 0$, we have
		\[
				\Ext_{k^{\rat}_{\et}}^{n}(A, B)
			=
				\Ext_{\Alg / k}^{n}(A, B).
		\]
\end{MainThm}

It is important for this theorem that
the category $k^{\rat}$ contains the generic points of $A, B \in \Alg / k$.
With this theorem, assuming $k$ algebraically closed, which we may,
the essential part of the spectral sequence
associated with the morphism in \eqref{ass: main theorem, cup pairing} is
the sequence
	\begin{align*}
				0
		&	\to
				\Ext_{\Alg / k}^{1}(\alg{U}_{K} / (\alg{U}_{K})^{n}, \Z / n \Z)
			\to
				H^{1}(K_{\et}, \Z / n \Z)
		\\
		&	\qquad
			\to
				\Hom_{\Alg / k}(\Z / n \Z (1)(\alg{K}), \Z / n \Z)
			\to
				0,
	\end{align*}
where $\Z / n \Z (1)(\alg{K})$ is the finite \'etale group over $k$ given by
the kernel of multiplication by $n$ on $\alg{K}^{\times}$.
We can show that this sequence agrees with the exact sequence given by
Serre's local class field theory.
This proves the exactness of our sequence and hence
Assertion \eqref{ass: main theorem, cup pairing}.

This paper does not contain essentially new duality results.
Trying to give such results is not the purpose of this paper.
Our purpose is to give a clear exposition of the use of sheaves on the category of fields
in the duality theories of B\'egueri, Bester and Kato.
We hope that similar techniques can be applied to other situations
to reduce cohomology of schemes to cohomology of fields.

When writing this paper,
the author was informed that
P\'epin \cite{Pep14} found a similar formulation to ours,
independently at almost the same time.
He views cohomology of $K$ as a functor on the category of fields over $k$,
which is a key common feature between his and our formulations.
This might explain how the method of the category of fields may naturally arise.


\subsection{Organization}

The details of the duality are explained in
Section \ref{sec: duality for local fields with perfect residue field}
except for the proof of Theorem \ref{thm: main theorem, comparison of Ext}.
The proof of Theorem \ref{thm: main theorem, duality} finishes
at Section \ref{sec: duality with coefficients in a finite flat group scheme}.
We treat the following more general setting.
Using fppf cohomology of $K$,
the equal and mixed characteristic cases are treated together.
Since the group $\alg{U}_{K}$ is an infinite-dimensional proalgebraic group,
we treat proalgebraic groups in order to give a transparent argument.
This option is actually necessary in the equal characteristic case,
since even the group $\alg{U}_{K} / (\alg{U}_{K})^{p}$ in that case is infinite-dimensional
and $\Ext_{k^{\rat}_{\et}}^{1}(\alg{U}_{K} / (\alg{U}_{K})^{p}, \Z / p \Z)$
does not agree with the extension group for the category $\Pro \Alg / k$ of proalgebraic groups.
This leads us to define a larger site than $\Spec k^{\rat}_{\et}$,
which we call the \emph{ind-rational \'etale site} of $k$,
denoted by $\Spec k^{\ind\rat}_{\et}$.
Its underlying category $k^{\ind\rat}$ consists of ind-rational $k$-algebras,
which are defined as filtered unions of rational $k$-algebras.
The corresponding generalization of Theorem \ref{thm: main theorem, comparison of Ext}
is Theorem \ref{thm: comparison of Ext, proalgebraic setting},
where $A$ is allowed to be proalgebraic.
Assuming Theorem \ref{thm: comparison of Ext, proalgebraic setting},
the proof of our duality theorem goes mostly in the manner outlined above.
Unfortunately, however, the option we take raises an additional problem.
We have to compute $H^{2}(\alg{K}(k')_{\et}, \Gm)$
with $k'$ ind-rational that may have infinitely many direct factors.
In this case, the ring $\alg{K}(k')$ is no longer a finite product of fields,
not even a direct limit of finite products of fields.
In Section \ref{sec: Cohomology of local fields with ind-rational base},
we compute this type of cohomology
by approximating it with cohomology of complete discrete valuation subfields.

We do not redo all results of B\'egueri, Bester and Kato.
We do, however, include a duality for varieties over $K$ (assuming $K$ has mixed characteristic)
with possibly $p$-torsion coefficients in
Section \ref{sec: duality for a variety over a local field}.
This is a simple combination of the duality for $K$ and
the Poincar\'e duality for varieties over an algebraic closure of $K$,
which does not seem to have been written down elsewhere.
We also restate Theorem \ref{thm: main theorem, duality} as a Verdier-type duality for $K$
and discuss a possible relation with the Albanese property of $\alg{K}^{\times}$ (\cite{CC94})
in Section \ref{sec: the right adjoint to the pushforward}.
When writing this paper,
the author realized that some categories closely related to our site
have already been studied by Rovinsky \cite{Rov05} and Jannsen-Rovinsky \cite{JR10}
in their study of motives.
We explain these relations in
Section \ref{sec: comparison with Jannsen-Rovinsky's dominant topology}.

The proof of Theorem \ref{thm: comparison of Ext, proalgebraic setting},
thereby Theorem \ref{thm: main theorem, comparison of Ext},
occupies the entire part of
Section \ref{sec: extensions of algebraic groups as sheaves on the rational etale site}.
We outline the proof.
First note that
if $A$ is an algebraic group over $k$ with generic point $\xi_{A}$,
then the group operation map
	\[
		\xi_{A} \times_{k} \xi_{A} \to A
	\]
is faithfully flat.%
\footnote{
	See \cite[II, \S 5, Lemma 1.2]{DG70} for example,
	applying the limit argument given in \cite[III, \S 3, Lemma 7.1]{DG70}.
}
In other words, any point of $A$ can be written as the sum of two generic points.
A key observation is to regard this well-known fact as saying that $A$ is \emph{covered by fields}.
Now let $\Spec k^{\perf}_{\et}$ be the category of perfect $k$-algebras (having invertible Frobenius)
endowed with the \'etale topology.
Breen's results \cite{Bre70} and \cite{Bre81} tell us that
	\[
			\Ext_{k^{\perf}_{\et}}^{n}(A, B)
		=
			\Ext_{\Pro \Alg / k}^{n}(A, B).
	\]
It does not seem possible to directly compare extensions
over $\Spec k^{\perf}_{\et}$ and $\Spec k^{\ind\rat}_{\et}$.
The \'etale topology is too coarse to treat the above morphism
$\xi_{A} \times_{k} \xi_{A} \to A$ as a covering.
Instead of trying a direct comparison, we define another site,
the \emph{perfect pro-fppf site}, denoted by $\Spec k^{\perf}_{\pro\fppf}$.
Its underlying category is the category of perfect $k$-algebras,
where a covering of a perfect affine $k$-scheme $X$ is
a jointly surjective finite family of filtered inverse limits
of perfect flat affine $X$-schemes of finite presentation.
This is a flat analog of Scholze's pro-\'etale site \cite{Sch13}.
Using the pro-fppf topology,
we can include the faithfully flat morphism
$\xi_{A} \times_{k} \xi_{A} \onto A$ above as a covering,
even though it is not of finite presentation or pro-\'etale.
There is no difference between cohomology with respect to
$\Spec k^{\perf}_{\pro\fppf}$ and $\Spec k^{\perf}_{\et}$
with coefficients in a quasi-algebraic group $B$
since $B$ is the perfection of a smooth algebraic group.
(There seems to be no known analogous comparison result for the fpqc cohomology instead of pro-fppf
for the third cohomology $H^{3}(R, \Gm)$ or higher.)
The comparison of extensions over $\Spec k^{\perf}_{\pro\fppf}$ and $\Spec k^{\ind\rat}_{\et}$
is complicated due to the fact that
the category $k^{\rat}$ of fields do not have all finite fiber sums
and the pullback functor for the natural continuous functor
$\Spec k^{\perf}_{\pro\fppf} \to \Spec k^{\ind\rat}_{\et}$ is not exact.
To overcome this, we follow Breen's method (\cite{Bre78}) to write extension groups of $A$
in terms of spectral sequences whose $E_{2}$-terms are given by
cohomology groups of products of $A$.
More precisely, we use Mac Lane's resolution
defined as the bar construction for the cubical construction (\cite{Mac57}),
and apply it to the left variable $A$ of the functor $\Ext^{n}(A, B)$.
A key property of the cubical construction is that it is an additive functor
up to a very explicit and simple chain homotopy called the \emph{splitting homotopy}
(\cite[\S 5, Lemme 2]{Mac57}).
This homotopy and variants of the covering $\xi_{A} \times_{k} \xi_{A} \onto A$
allow us to replace the cohomology groups appearing in the $E_{2}$-terms
by cohomology groups of fields.
This is the hardest part of this paper.

Here are the logical connections of the sections:
	\[
		\begin{CD}
			@.
				\text{\ref{sec: The ind-rational etale site}}
			@>>>
				\text{(\ref{sec: Comparison of Ext in low degrees: birational groups})}
			@.
			\\
			@.
			@VVV
			@VVV
			@.
			\\
				\text{\ref{sec: comparison with Jannsen-Rovinsky's dominant topology}}
			@<<<
				\text{\ref{sec: extensions of algebraic groups as sheaves on the rational etale site}}
			@>>>
				\text{\ref{sec: The relative fppf site of a local field}%
				--\ref{sec: duality with coefficients in a finite flat group scheme}}
			@>>>
				\text{\ref{sec: duality for a variety over a local field}%
				--\ref{sec: the right adjoint to the pushforward}}
		\end{CD}
	\]
Section \ref{sec: Comparison of Ext in low degrees: birational groups} does not substitute
Section \ref{sec: extensions of algebraic groups as sheaves on the rational etale site}.
It proves only the case $n = 0, 1$ of Theorem \ref{thm: main theorem, comparison of Ext},
which, though, gives some ideas about the general case.
Section \ref{sec: Cohomology of local fields with ind-rational base} can be skipped
if one is only interested in the mixed characteristic case.
Hence the quickest way to Theorem \ref{thm: main theorem, duality} may be
\ref{sec: The ind-rational etale site}%
--\ref{sec: duality with coefficients in a finite flat group scheme}
with \ref{sec: Cohomology of local fields with ind-rational base} skipped,
interpreting the ind-rational \'etale site as the rational \'etale site.

\begin{Ack}
	The author expresses his deep gratitude to Kazuya Kato
	for many helpful discussions and continuous support.
	Kato's duality theory was the starting point of this work.
	It is a pleasure to thank Jon Peter May and Takefumi Nosaka.
	Several discussions with each of them have helped and encouraged the author
	to attack to Theorem \ref{thm: main theorem, comparison of Ext}
	with the machinery from algebraic topology.
	The author is grateful to Teruhisa Koshikawa
	for suggesting a relation with Scholze's pro-\'etale site among other useful conversations,
	Madhav Nori for helpful conversations about flat morphisms,
	Alessandra Bertapelle for her comments on an earlier version of this paper,
	Aise Johan de Jong for suggesting a simpler proof of
	Proposition \ref{prop: composite of flat of ind-finite presentation},
	and the referee for their careful comments.
\end{Ack}

\begin{Notation}
	We fix a perfect field $k$ of characteristic $p > 0$.
	A perfect field over $k$ is said to be finitely generated
	if it is the perfection (direct limit along Frobenii) of a finitely generated field over $k$.
	The same convention is applied to morphisms of perfect $k$-algebras or $k$-schemes
	being finite type, finite presentation etc.
	A perfect $k$-scheme of finite type
	is also said to be quasi-algebraic following Serre's terminology \cite{Ser60}.
	The perfections of $\Ga$, $\Gm$ and $\Affine_{k}^{n}$ over $k$ are denoted by the same symbols
	$\Ga$, $\Gm$ and $\Affine_{k}^{n}$ by abuse of notation.
	A point of a scheme is a point of its underlying set unless otherwise noted,
	usually identified with the $\Spec$ of its residue field.
	If $X$ is a perfect (hence reduced) $k$-scheme of finite type,
	then its generic point, denoted by $\xi_{X}$,
	means the disjoint union of the generic points of its irreducible components.
	For a set $X$,
	we denote by $\Z[X]$ the free abelian group generated by $X$.
	We denote by $\Set$, $\Ab$, $\GrAb$, $\DGAb$
	the categories of sets, abelian groups,
	graded abelian groups,
	differential graded abelian groups, respectively.
	Set theoretic issues are omitted for simplicity
	as the main results hold independent of the choice of universes.
	All groups (except for Galois groups) are assumed to be commutative.
	We denote by $\Alg / k$
	the category of (commutative) affine quasi-algebraic groups over $k$
	in the sense of Serre \cite{Ser60},
	that is, group objects in the category of affine quasi-algebraic schemes over $k$.
	Its procategory $\Pro \Alg / k$ is the category of affine proalgebraic groups over $k$.
	We denote by $\Loc \Alg / k$
	the category of the perfections of smooth group schemes over $k$.
	For a Grothendieck site $S$ and a category $\mathcal{C}$,
	we denote by $\mathcal{C}(S)$ the category of sheaves on $S$ with values in $\mathcal{C}$.
	For an object $X$ of $S$,
	the category $S / X$ of objects of $S$ over $X$ is equipped with the induced topology
	(\cite[III, \S 3]{AGV72a}),
	which is the localization of $S$ at $X$ (\cite[III, \S 5]{AGV72a}).
	For $F \in \Set(S)$, we denote by $\Z[F]$ the sheafification of the presheaf
	$X \mapsto \Z[F(X)]$.
	By a continuous map $f \colon S' \to S$ between sites $S'$ and $S$
	we mean a continuous functor from the underlying category of $S$ to that of $S'$,
	which means that the right composition (or the pushforward $f_{\ast}$)
	sends sheaves on $S'$ to sheaves on $S$.
	By a morphism $f \colon S' \to S$ of sites
	we mean a continuous map whose pullback functor $\Set(S) \to \Set(S')$ is exact.
	For an abelian category $\mathcal{A}$,
	we denote by $\Ext_{\mathcal{A}}^{i}$ the $i$-th Ext functor for $\mathcal{A}$.
	The bounded, bounded below, bounded above and unbounded derived categories of $\mathcal{A}$ are denoted by
	$D^{b}(\mathcal{A})$, $D^{+}(\mathcal{A})$, $D^{-}(\mathcal{A})$ and $D(\mathcal{A})$, respectively.
	If $\mathcal{A} = \Ab(S)$,
	we also write $D^{\ast}(S) = D^{\ast}(\Ab(S))$ for $\ast = b, +, -$ and (blank),
	and $\Ext_{S}^{i} = \Ext_{\mathcal{A}}^{i}$.
	For sites such as $\Spec k^{\rat}_{\et}$,
	we also use $\Ext_{k^{\rat}_{\et}}$, $\Ab(k^{\rat}_{\et})$ etc.\ omitting $\Spec$ from the notation.
	When a continuous map $f \colon S' \to S$ of subcanonical sites is obtained
	by the identity functor on the underlying categories,
	we simply write $f_{\ast} F = F$ for a representable sheaf $F$.
\end{Notation}


\section{Duality for local fields with perfect residue field}
\label{sec: duality for local fields with perfect residue field}

\subsection{The ind-rational \'etale site}
\label{sec: The ind-rational etale site}
Before introducing the ind-rational \'etale site,
we first need to fix our terminology on Grothendieck sites \cite{AGV72a},
which we will use throughout the paper,
and recall some basic facts.
All sites we use are (fortunately) defined by covering families or pretopologies
\cite[II, D\'efinition 1.3]{AGV72a},
except for the dominant topology $\DM_{k}$ that we study in
Sections \ref{sec: the flat case} and \ref{sec: The dominant topology and the pro-fppf topology}.
As in Notation, a continuous map $f \colon S' \to S$ of sites
(called a continuous functor in \cite[III, D\'efinition 1.1]{AGV72a}) is a functor from $S$ to $S'$
such that the pushforward $f_{\ast}$ sends sheaves on $S'$ to sheaves on $S$.
In general, this is stronger than saying that $f$ sends coverings in $S$ to coverings in $S'$.
But these are equivalent if
	$
			f(Y \times_{X} Z)
		=
			f(Y) \times_{f(X)} f(Z)
	$
in $S'$ for any morphism $Y \to X$ in $S$ appearing in a covering family
and any morphism $Z \to X$ in $S$.
(See \cite[III, Proposition 1.6]{AGV72a} for these two facts.)
All continuous maps in this paper,
except for those in Sections \ref{sec: the flat case}
and \ref{sec: The dominant topology and the pro-fppf topology},
satisfy this extra condition.
The pushforward functor of such a continuous map (between sites defined by pretopologies)
sends acyclic sheaves to acyclic sheaves 
and hence induces Leray spectral sequences \cite[\S 2.4]{Art62}.
We say that $A \in \Ab(S)$ is \emph{acyclic} if
$H^{n}(X, A) = 0$ for any $X$ in $S$ and any $n \ge 1$.
This is called $S$-acyclic in \cite[V, D\'efinition 4.1]{AGV72b}
and flask in \cite[\S 2.4]{Art62}.

When $f^{\ast}$ is exact, we say that $f$ is a morphism of sites \cite[IV, \S 4.9]{AGV72a}.
We will have to use continuous maps without exact pullbacks
in Sections \ref{sec: extensions of algebraic groups as sheaves on the rational etale site}
and \ref{sec: comparison with Jannsen-Rovinsky's dominant topology}.
The pushforward functor of such a functor does not send injectives to injectives,
which prevents us from using the Grothendieck spectral sequence for example.
There are two sufficient conditions for
a continuous map $f \colon S' \to S$ to be a morphism of sites:
$S$ has all finite limits and $f$ commutes with finite limits \cite[III, Proposition 1.3 (5)]{AGV72a};
or $f$ admits a left adjoint \cite[III, Proposition 2.5]{AGV72a}.
All continuous maps in this section satisfy one of these conditions,
hence morphisms of sites.

Now we define key notions.

\begin{Def}
	We say that a perfect $k$-algebra $k'$ is \emph{rational}
	if it is a finite direct product of finitely generated perfect fields over $k$,
	and \emph{ind-rational} if it is a filtered union of rational $k$-subalgebras.
	The rational (resp.\ ind-rational) $k$-algebras form a full subcategory of
	the category of perfect $k$-algebras,
	which we denote by $k^{\rat}$ (resp.\ $k^{\ind\rat}$).
\end{Def}

Note that $k^{\ind\rat}$ is naturally equivalent to the ind-category of $k^{\rat}$
and $k^{\rat}$ consists of the compact objects in $k^{\ind\rat}$.
Rational $k$-algebras appear as the rings of rational functions on perfect $k$-schemes of finite type.
Examples of ind-rational $k$-algebras include a perfect field over $k$ and
the ring of global sections of the structure sheaf of a profinite set regarded as a $k$-scheme.
Any homomorphism in $k^{\ind\rat}$ is flat.
A homomorphism in $k^{\ind\rat}$ is faithfully flat
if and only if it is injective.

\begin{Prop}
	Any $k$-algebra \'etale over a rational (resp.\ ind-rational) $k$-algebra
	is rational (resp.\ ind-rational).
\end{Prop}

\begin{proof}
	Let $k_{2}$ be a $k$-algebra \'etale over a perfect $k$-algebra $k_{1}$.
	If $k_{1}$ is rational, clearly so is $k_{2}$.
	If $k_{1}$ is ind-rational,
	then we can write $k_{2} = k_{2}' \tensor_{k_{1}'} k_{1}$
	with $k_{1}'$ a rational $k$-subalgebra of $k_{1}$
	and $k_{2}'$ \'etale over $k_{1}'$.
	Write $k_{1}$ as a filtered union of rational $k$-subalgebras $k_{1, \lambda}$ containing $k_{1}'$.
	Then $k_{2}$ is the filtered union of the rational $k$-subalgebras $k_{2}' \tensor_{k_{1}'} k_{1, \lambda}$,
	hence ind-rational.
\end{proof}

The proposition above leads to the following definition.

\begin{Def}
	We define the \emph{rational \'etale site} $\Spec k^{\rat}_{\et}$
	(resp.\ \emph{ind-rational \'etale site} $\Spec k^{\ind\rat}_{\et}$)
	to be the category $k^{\rat}$ (resp.\ $k^{\ind\rat}$)
	endowed with the \'etale topology.
\end{Def}

Note that usual tensor products of rings do not always give fiber sums for $k^{\rat}$ or $k^{\ind\rat}$,
and not all fiber sums in $k^{\rat}$ or $k^{\ind\rat}$ exist.
Therefore a continuous map to the site $\Spec k^{\ind\rat}_{\et}$
does not always have an exact pullback functor
(an example is given in Proposition \ref{prop: pullback is not exact}).
In this section, we will always need to check that
the continuous maps we have are indeed morphisms of sites.

Some care is needed for localizations (see Notation) of
$\Spec k^{\rat}_{\et}$ and $\Spec k^{\ind\rat}_{\et}$.
If $k' \in k^{\ind\rat}$ is a field,
then a rational $k'$-algebra is ind-rational over $k$
since it is a finite product of field extensions of $k$.
Hence an ind-rational $k'$-algebra is ind-rational over $k$.
But a $k'$-algebra ind-rational over $k$ is not necessarily ind-rational over $k'$.%
\footnote{
	An example is given as follows.
	Let $k'_{n}$ be the perfection of $k(x_{1}, \dots, x_{n})$
	and $k''_{n} = (k'_{n})^{n} \times k'_{n + 1}$ with the natural $k'_{n}$-algebra structure.
	Consider the inclusion $k'_{n} \into k'_{n + 1}$ and
	the $k'_{n}$-algebra homomorphism $k''_{n} \into k''_{n + 1}$
	given by $(f_{1}, \dots, f_{n + 1}) \mapsto (f_{1}, \dots, f_{n + 1}, \varphi(f_{n + 1}))$,
	where $\varphi(f(x_{1}, \dots, x_{n + 1})) = f(x_{1}, \dots, x_{n}, x_{n + 2})$.
	Let $k' = \bigcup k'_{n}$ and consider the $k'$-algebra $k'' = \bigcup k''_{n}$ ind-rational over $k$.
	Then it can be shown that $k''$ is not ind-rational over $k'$.
}
Hence the category $k^{\ind\rat} / k'$ of objects over $k'$ in $k^{\ind\rat}$
can strictly contain $k'^{\ind\rat}$
and the localization $\Spec k^{\ind\rat}_{\et} / k'$ can be different from $\Spec k'^{\ind\rat}_{\et}$.
But if $k' \in k^{\rat}$ is a field
(that is, if $k'$ is a finitely generated perfect field over $k$),
then $k^{\ind\rat} / k' = k'^{\ind\rat}$ and $k^{\rat} / k' = k'^{\rat}$.
Therefore for any $k' \in k^{\rat}$, we can define
	\begin{gather*}
				k'^{\ind\rat}
			:=
				k^{\ind\rat} / k',
			\quad
				k'^{\rat}
			:=
				k^{\rat} / k',
		\\
				\Spec k'^{\ind\rat}_{\et}
			:=
				\Spec k^{\ind\rat}_{\et} / k',
			\quad
				\Spec k'^{\rat}_{\et}
			:=
				\Spec k^{\rat}_{\et} / k'
	\end{gather*}
without ambiguity.
At any rate,
the localization $\Spec k^{\ind\rat}_{\et} / k'$ for any $k' \in k^{\ind\rat}$
is the category of $k'$-algebras ind-rational over $k$
endowed with the \'etale topology.

Also if $k'$ is an algebraic extension field of $k$,
then $k^{\ind\rat} / k' = k'^{\ind\rat}$.
To see this, let $k'' = \bigcup k''_{\lambda}$ with $k''_{\lambda} \in k^{\rat}$
and suppose $k''$ has a structure of a $k'$-algebra.
Then $k''_{\lambda} \tensor_{k} k' \in k'^{\rat}$
since $k'$ is algebraic (hence separable) over $k$.
Let $k'''_{\lambda} \in k'^{\rat}$ be the image of $k''_{\lambda} \tensor_{k} k'$ in $k''$.
Then $k'' = \bigcup k'''_{\lambda} \in k'^{\ind\rat}$.
In particular, we have
	\[
			\closure{k}^{\ind\rat}
		=
			k^{\ind\rat} / \closure{k},
		\quad
			\Spec \closure{k}^{\ind\rat}_{\et}
		=
			\Spec k^{\ind\rat}_{\et} / \closure{k},
	\]
where $\closure{k}$ is an algebraic closure of $k$.

The cohomology theory of the site $\Spec k^{\ind\rat}_{\et}$ is
essentially Galois cohomology, as follows.

\begin{Prop} \label{prop: rational etale cohomology is Galois cohomology}
	Let $k' \in k^{\ind\rat}$.
	Let $f \colon \Spec k^{\ind\rat}_{\et} / k' \to \Spec k'_{\et}$ be the morphism defined by the identity.
	Then $f_{\ast}$ is exact.
	We have
		\[
				R \Gamma(k^{\ind\rat}_{\et} / k', A)
			=
				R \Gamma(k'_{\et}, f_{\ast} A)
		\]
	for any $A \in \Ab(k^{\ind\rat}_{\et} / k')$,
	where the left-hand side is the cohomology of the site $\Spec k^{\ind\rat}_{\et} / k'$
	(at the final object $k'$)
	with coefficients in $A$.
\end{Prop}

\begin{proof}
	First note that $f$ is a morphism of sites
	since the underlying category of the target site $\Spec k'_{\et}$ has all finite limits
	and $f$ commutes with finite limits.
	The exactness of $f_{\ast}$ is obvious.
	Hence the Grothendieck spectral sequence yields the result.
\end{proof}

For the rest of the paper, we will denote the object
$R \Gamma(k^{\ind\rat}_{\et} / k', A)$ appearing in the proposition
simply by $R \Gamma(k'_{\et}, A)$.
For any sheaf $A \in \Ab(k^{\ind\rat}_{\et})$,
the cohomology of $k' \in k^{\ind\rat}$ with coefficients in $A$
(that is, the derived functor of $\Gamma(k', \;\cdot\;) \colon \Ab(k^{\ind\rat}_{\et}) \to \Ab$)
is given by the cohomology of the site $\Spec k^{\ind\rat}_{\et} / k'$
with coefficients in the restriction $A|_{k'} \in \Ab(k^{\ind\rat}_{\et} / k')$
by the relation between cohomology and localization
\cite[V, \S 2.2, 1st paragraph]{AGV72b}, \cite[IV, \S 5.1, 1st paragraph]{AGV72a}.
Hence the above is enough for describing cohomology of any $k' \in k^{\ind\rat}$.

As in Notation,
we denote by $\Alg / k$ and $\Pro \Alg /k$
the categories of affine quasi-algebraic and affine proalgebraic groups over $k$, respectively.
We also denote by $\Loc \Alg / k$ the category of the perfections of smooth group schemes over $k$,
which contains $\Alg / k$ and $\Z$ for example.%
\footnote{
	The ``L'' stands for ``locally''.
	A related fact that will not be used later is that
	if a perfect group scheme $G$ over $k$ is covered by quasi-algebraic open affine subschemes,
	then it is the perfection of a smooth group scheme.
	Use the fact that the underlying topological space of $G$ is locally noetherian
	and hence the disjoint union of open connected components.
	The identity component $G_{0}$ of $G$ is quasi-compact and separated due to the group structure.
	Hence $G_{0}$ is the perfection of a smooth algebraic group $G_{0}'$
	by \cite[\S 1.4, Proposition 10]{Ser60}.
	If $N$ is the kernel of the natural morphism $G_{0} \onto G_{0}'$,
	then $G / N$ is a smooth group scheme since $\pi_{0}(G)$ is \'etale.
	Then $G$ is the perfection of $G / N$.
}
Recall from \cite[\S 3.6, Proposition 13]{Ser60} that
the $\Hom$ group and the $\Ext^{n}$ groups between the perfections of algebraic groups $A, B$
are the direct limits of those between $A, B$ along Frobenii.
We frequently apply results on algebraic groups to quasi-algebraic groups
by passing to limits along Frobenii, without giving the detailed procedures.
By evaluation,
we have natural functors from any of the categories
$\Alg / k$, $\Pro \Alg /k$ and $\Loc \Alg / k$
to $\Ab(k^{\rat}_{\et})$ and $\Ab(k^{\ind\rat}_{\et})$.
The functors $\Alg / k \to \Ab(k^{\rat}_{\et}) \into \Ab(k^{\ind\rat}_{\et})$ are exact
and hence induce homomorphisms
	\[
			\Ext_{\Alg / k}^{n}(A, B)
		\to
			\Ext_{k^{\rat}_{\et}}^{n}(A, B)
		\to
			\Ext_{k^{\ind\rat}_{\et}}^{n}(A, B)
	\]
for any $A, B \in \Alg / k$ and $n \ge 0$.
Hence for $A = \invlim A_{\lambda} \in \Pro \Alg / k$ with $A_{\lambda} \in \Alg / k$ and $B \in \Alg / k$,
we have a homomorphism
	\[
			\Ext_{\Pro \Alg / k}^{n}(A, B)
		=
			\dirlim_{\lambda}
				\Ext_{\Alg / k}^{n}(A_{\lambda}, B)
		\to
			\dirlim_{\lambda}
				\Ext_{k^{\ind\rat}_{\et}}^{n}(A_{\lambda}, B)
		\to
			\Ext_{k^{\ind\rat}_{\et}}^{n}(A, B)
	\]
(see \cite[\S 3.4, Proposition 7]{Ser60} for the first isomorphism).
We will prove Theorem \ref{thm: main theorem, comparison of Ext} and the following generalization
in Section \ref{sec: extensions of algebraic groups as sheaves on the rational etale site}.

\begin{Thm} \label{thm: comparison of Ext, proalgebraic setting}
	Let $A \in \Pro \Alg / k$, $B \in \Loc \Alg / k$ and
	$k' = \bigcup_{\nu} k'_{\nu} \in k^{\ind\rat}$ with $k'_{\nu} \in k^{\rat}$.
	Then for any $n \ge 0$, we have
		\[
				\Ext_{k^{\ind\rat}_{\et} / k'}^{n}(A, B)
			=
				\dirlim_{\nu}
					\Ext_{(k'_{\nu})^{\ind\rat}_{\et}}^{n}(A, B).
		\]
	If $k'$ is a field, then this is further isomorphic to
	$\Ext_{k'^{\ind\rat}_{\et}}^{n}(A, B)$.
	If $B \in \Alg / k$, then
		\begin{gather*}
					\Ext_{k^{\ind\rat}_{\et}}^{n}(A, B)
				=
					\Ext_{\Pro \Alg / k}^{n}(A, B),
			\\
					\Ext_{k^{\ind\rat}_{\et}}^{n}(A, \Q)
				=
					\Hom_{k^{\ind\rat}_{\et}}(A, \Z)
				=
					0,
			\\
					\Ext_{k^{\ind\rat}_{\et}}^{n + 1}(A, \Z)
				=
					\Ext_{k^{\ind\rat}_{\et}}^{n}(A, \Q / \Z)
				=
					\dirlim_{m}
						\Ext_{\Pro \Alg / k}^{n}(A, \Z / m \Z),
		\end{gather*}
	If $A \in \Alg / k$, then the isomorphisms in the last three lines also hold
	with $k^{\ind\rat}_{\et}$ replaced by $k^{\rat}_{\et}$
	and $\Pro \Alg / k$ by $\Alg / k$.
\end{Thm}

In particular, the functors $\Alg / k \to \Ab(k^{\rat}_{\et})$
and $\Pro \Alg / k \to \Ab(k^{\ind\rat}_{\et})$ are fully faithful.
Using this theorem,
we will interpret Serre's local class field theory as
a duality for local fields with coefficients in $\Gm$ and extensions by $\Z$.
(We denote the perfection of $\Gm$ simply by $\Gm$ as in Notation.)
The result is Theorem \ref{thm: duality, Z-coefficients}.
We will then deduce Theorem \ref{thm: main theorem, duality}
by trivializing finite (multiplicative) $A$
and embedding it into a product of copies of $\Gm$.

\begin{Rmk}
	Theorem \ref{thm: comparison of Ext, proalgebraic setting} holds
	also for commutative quasi-algebraic groups $A$ and $B$
	not necessarily affine,
	for example, abelian varieties.
	See Remark \ref{rmk: generalizing to abelian varieties}.
	Affine groups as stated above are sufficient
	for the purpose of duality for local fields with coefficients in finite group schemes.
\end{Rmk}


\subsection{Comparison of Ext in low degrees: birational groups}
\label{sec: Comparison of Ext in low degrees: birational groups}
In this subsection,
we give a proof of the following important part of
Theorem \ref{thm: main theorem, comparison of Ext}
(or \ref{thm: comparison of Ext, proalgebraic setting}):

\begin{Prop} \label{prop: special case of comparison of Ext}
	Let $A, B \in \Alg / k$.
	Then we have
		\[
				\Ext_{\Alg / k}^{n}(A, B)
			\isomto
				\Ext_{k^{\rat}_{\et}}^{n}(A, B)
		\]
	for $n = 0, 1$.
\end{Prop}

This part can be treated by means of
\emph{generic translations} and \emph{birational group laws} in the sense of Weil
(\cite{Wei55}; see also \cite[V, \S1.5]{Ser88}).
By evaluating an element of $\Ext_{k^{\rat}_{\et}}^{n}(A, B)$ at the generic point of $A$,
we can obtain a homomorphism of birational groups if $n = 0$
and a rational symmetric factor system (\cite[VII, \S 1.4]{Ser88}) if $n = 1$,
which are regularizable and come from $\Ext_{\Alg / k}^{n}(A, B)$ by means of generic translations.
This proof is illustrative and helpful to understand the general case $n \ge 2$,
where we will use Mac Lane's resolution and
apply a similar but more involved regularization (or generification) technique for it
instead of Weil's results.
The cases $n = 0, 1$ are important since
the groups $\Ext_{\Pro \Alg / k}^{n}(A, B)$ for $n \ge 2$
(resp.\ $n \ge 3$) with $k$ algebraically closed are
mostly (resp.\ always) zero (\cite[\S 10]{Ser60}).
Our explanation of the duality, however,
does need the same vanishing for $\Ext_{k^{\rat}_{\et}}^{n}(A, B)$.

\begin{Lem} \label{lem: n equals zero case of comparison of Ext}
	Proposition \ref{prop: special case of comparison of Ext} is true for $n = 0$.
\end{Lem}

\begin{proof}
	For the injectivity, let $\varphi \colon A \to B$ in $\Alg / k$
	become zero in $\Ab(k^{\rat}_{\et})$.
	Let $\xi_{A}$ be the generic point of $A$
	(which is defined to be the disjoint union of the generic points of all irreducible components of $A$).
	Then $\varphi(\xi_{A}) \in B(\xi_{A})$ is zero.
	This means that $\varphi$ is generically zero.
	A generic translation then shows that $\varphi$ is everywhere zero.
	
	For the surjectivity,
	let $\varphi \in \Hom_{k^{\rat}}(A, B)$.
	Consider the element $\varphi(\xi_{A}) \in B(\xi_{A})$.
	This corresponds to a rational map
	$\tilde{\varphi} \colon A \dashrightarrow B$.
	Consider the following commutative diagram in $\Set(k^{\rat}_{\et})$:
		\[
			\begin{CD}
					A \times_{k} A
				@>> \varphi \times \varphi >
					B \times_{k} B
				\\
				@VVV
				@VVV
				\\
					A
				@> \varphi >>
					B
			\end{CD}
		\]
	The vertical arrows are given by the group operations.
	Applying this commutativity for the element
	$\xi_{A \times A} \in A \times A$,
	we obtain an equality
		$
				\tilde{\varphi}(x_{1} + x_{2})
			=
				\tilde{\varphi}(x_{1}) + \tilde{\varphi}(x_{2})
		$
	as rational maps $A \times A \dashrightarrow B$.
	This and a generic translation imply that
	the rational map $\tilde{\varphi} \colon A \dashrightarrow B$ is
	everywhere regular and is a homomorphism of algebraic groups
	(\cite[V, \S 1.5, Lemma 6]{Ser88}).

	It remains to show that $\tilde{\varphi} = \varphi$ in $\Hom_{k^{\rat}}(A, B)$.
	Let $k' \in k^{\rat}$ and $a \in A(k')$.
	We want to decompose $a$ into the sum of two generic points.
	Let $A_{k'} = \Spec k' \times_{k} A$ and
	$\Spec k''$ the generic point of $A_{k'}$.
	The natural projection $A_{k'} \to A$ gives a point $x \in \xi_{A}(k'')$.
	We have $\varphi(x) = \tilde{\varphi}(x)$ by definition of $\tilde{\varphi}$.
	If we show that $a - x \in \xi_{A}(k'')$, then
	$\varphi(a - x) = \tilde{\varphi}(a - x)$, so
		\[
				\varphi(a)
			=
				\varphi(a - x) + \varphi(x)
			=
				\tilde{\varphi}(a - x) + \tilde{\varphi}(x)
			=
				\tilde{\varphi}(a),
		\]
	which proves $\varphi = \tilde{\varphi}$.
	To show that $a - x \in \xi_{A}(k'')$,
	note that the correspondence $y \leftrightarrow a - y$
	gives an automorphism of the scheme $A_{k'}$.
	Hence $x \in \xi_{A}(k'')$ implies that $a - x \in \xi_{A}(k'')$.
	This completes the proof of
	$\Hom_{\Alg / k}(A, B) = \Hom_{k^{\rat}}(A, B)$.
\end{proof}

In particular, the exact functor $\Alg / k \to \Ab(k^{\rat}_{\et})$ is fully faithful.

\begin{Lem}
	For the case $n = 1$ of Proposition \ref{prop: special case of comparison of Ext},
	it is enough to show the surjectivity of the map,
	and we may assume that $A$ and $B$ are connected
	and $H^{1}(k'_{\et}, B) = 0$ for all $k' \in k^{\rat}$.
\end{Lem}

\begin{proof}
	If $0 \to B \to C \to A \to 0$ is an extension in $\Alg / k$
	that admits a splitting $s \colon A \into C$ in $\Ab(k^{\rat}_{\et})$,
	then $s \in \Hom_{\Alg / k}(A, C)$ by the $n = 0$ case,
	hence the extension splits also in $\Alg / k$.
	This shows the injectivity.
	
	For the second half of the lemma,
	let $k''$ be a finite Galois extension of $k$.
	We first show that if the statement of the $n = 1$ case of the proposition is true
	for $A \times_{k} k''$, $B \times_{k} k''$, that is,
		\[
				\Ext_{\Alg / k''}^{1}(A, B)
			=
				\Ext_{k''^{\rat}_{\et}}^{1}(A, B),
		\]
	then it is true for $A, B$.
	For any $i \ge 0$, let $k''^{\tensor i + 1}$ be the tensor product of $i + 1$ copies of $k''$ over $k$
	and $B^{i}$ the Weil restriction $\Res_{k''^{\tensor i + 1} / k} B \in \Alg / k$.
	The usual formulas for coboundary maps of augmented \v{C}ech complexes define a complex
	$0 \to B \to B^{0} \to B^{1} \to \cdots$ in $\Alg / k$
	(see the proof of \cite[III, Theorem 3.9]{Mil80}).
	On $\closure{k}$-points, it is the augmented \v{C}ech complex of
	$k'' \tensor_{k} \closure{k} \, / \, \closure{k}$ with coefficients in $B$,
	which is exact since $\closure{k}$ has trivial cohomology.
	Hence the complex $B^{\bullet} = \{B^{i}\}_{i \ge 0}$ is a resolution of $B$ in $\Alg / k$.
	Let $F_{1}^{j} = \Ext_{\Alg / k}^{j}(A, \;\cdot\;)$
	and $F_{2}^{j} = \Ext_{k^{\rat}_{\et}}^{j}(A, \;\cdot\;)$ for $j \ge 0$.
	For $n = 1, 2$ and $j \ge 0$, denote the $i$-th cohomology of the complex
	$F_{n}^{j}(B^{\bullet}) = [0 \to F_{n}^{j}(B^{0}) \to F_{n}^{j}(B^{1}) \to \cdots]$
	by $H^{i} F_{n}^{j}(B^{\bullet})$.
	Since $\{F_{n}^{j}\}_{j \ge 0}$ is a $\delta$-functor for $n = 1, 2$,
	a simple diagram chase shows that we have a four term exact sequence
		\[
				0
			\to
				H^{1} F_{n}^{0}(B^{\bullet})
			\to
				F_{n}^{1}(B)
			\to
				H^{0} F_{n}^{1}(B^{\bullet})
			\to
				H^{2} F_{n}^{0}(B^{\bullet})
		\]
	for $n = 1, 2$.
	The functor $\Alg / k \to \Ab(k^{\rat}_{\et})$ induces a morphism of sequences
	from the above sequence for $n = 1$ to that for $n = 2$.
	The morphism $F_{1}^{0}(B^{i}) \to F_{2}^{0}(B^{i})$ is an isomorphism for $i \ge 0$
	by Lemma \ref{lem: n equals zero case of comparison of Ext}.
	For each $i \ge 0$, the $k$-algebra $k''^{\tensor i + 1}$ is a finite product of copies of $k''$
	and the group $B^{i}$ is a finite product of copies of $\Res_{k'' / k} B$.
	We have
		\begin{gather*}
					F_{1}^{1}(\Res_{k'' / k} B)
				=
					\Ext_{\Alg / k''}^{1}(A, B),
			\\
					F_{2}^{1}(\Res_{k'' / k} B)
				=
					\Ext_{k''^{\rat}_{\et}}^{1}(A, B).
		\end{gather*}
	Therefore if the statement is true for $A \times_{k} k''$, $B \times_{k} k''$,
	then $F_{1}^{1}(B^{i}) \to F_{2}^{1}(B^{i})$ is an isomorphism for $i \ge 0$
	and hence $F_{1}^{1}(B) \to F_{2}^{1}(B)$ is an isomorphism,
	which is the statement for $A, B$.
	
	Now assume that the statement of the $n = 1$ case of the proposition is true
	for connected $A, B \in \Alg / k$ with $H^{1}(k'_{\et}, B) = 0$ for $k' \in k^{\rat}$.
	We show that the same statement is true for arbitrary $A, B \in \Alg / k$.
	
	Case $A, B$ connected:
	We can take a finite Galois extension $k''$
	such that $B \times_{k} k''$ is a successive extension of copies of $\Ga$ and $\Gm$.
	In particular, $H^{1}(k'_{\et}, B) = 0$ for $k' \in k''^{\rat}$,
	and the statement is true for $A \times_{k} k''$, $B \times_{k} k''$.
	So is for $A, B$ by the above Galois descent.
	
	Case $A$ \'etale and $B$ arbitrary:
	We may assume $A = \Z / m \Z$ for some $m \ge 1$ by again extending $k$.
	The long exact sequence for $\Ext_{k^{\rat}_{\et}}^{\cdot}(\;\cdot\;, B)$
	associated with the short exact sequence
	$0 \to \Z \to \Z \to \Z / m \Z \to 0$ yields
	an exact sequence
		\[
				0
			\to
				B(k) / m B(k)
			\to
				\Ext_{k^{\rat}_{\et}}^{1}(\Z / m \Z, B)
			\to
				H^{1}(k_{\et}, B)[m]
			\to
				0,
		\]
	where $[m]$ denotes the $m$-torsion part.
	The group $\Ext_{\Alg / k}^{1}(\Z / m \Z, B)$ also fits in the middle of this sequence.
	Hence we have $\Ext_{k^{\rat}_{\et}}^{1}(\Z / m \Z, B) = \Ext_{\Alg / k}^{1}(\Z / m \Z, B)$.
	
	Case $A$ arbitrary and $B$ connected:
	We denote the identity component of $A$ by $A_{0}$.
	Let $0 \to B \to C \to A \to 0$ be an extension in $\Ab(k^{\rat}_{\et})$.
	Let $0 \to B \to C' \to A_{0} \to 0$ be its pullback by $A_{0} \into A$.
	It defines an element of $\Ext_{k^{\rat}_{\et}}^{1}(A_{0}, B)$.
	The connected case above implies that $C' \in \Alg / k$.
	We have an exact sequence $0 \to C' \to C \to \pi_{0}(A) \to 0$,
	or an element of $\Ext_{k^{\rat}_{\et}}^{1}(\pi_{0}(A), C')$.
	Since $\pi_{0}(A)$ is \'etale, the previous case implies that $C \in \Alg / k$.
	Then the sequence $0 \to B \to C \to A \to 0$ is an exact sequence in $\Alg / k$
	since $\Alg / k \to \Ab(k^{\rat}_{\et})$ is fully faithful.
	
	Case $A$ arbitrary and $B$ \'etale:
	This reduces to the previous case by embedding $B$ into a connected affine quasi-algebraic group
	and the five lemma.
	
	Case $A, B$ arbitrary:
	We denote the identity component of $B$ by $B_{0}$.
	Let $0 \to B \to C \to A \to 0$ be an extension in $\Ab(k^{\rat}_{\et})$.
	Let $0 \to \pi_{0}(B) \to C' \to A \to 0$ be its pushout by $B \onto \pi_{0}(B)$.
	It defines an element of $\Ext_{k^{\rat}_{\et}}^{1}(A, \pi_{0}(B))$.
	The previous case implies that $C' \in \Alg / k$.
	We have an exact sequence $0 \to B_{0} \to C \to C' \to 0$,
	or an element of $\Ext_{k^{\rat}_{\et}}^{1}(C', B_{0})$.
	Since $B_{0}$ is connected, the previous to the previous case implies that $C \in \Alg / k$.
	Then the sequence $0 \to B \to C \to A \to 0$ is an exact sequence in $\Alg / k$.
\end{proof}

The following finishes the proof of Proposition \ref{prop: special case of comparison of Ext}.

\begin{Lem}
	Proposition \ref{prop: special case of comparison of Ext}
	is true for $n = 1$ if $A$ and $B$ are connected and
	$H^{1}(k'_{\et}, B) = 0$ for all $k' \in k^{\rat}$.
\end{Lem}

\begin{proof}
	Let $0 \to B \to C \to A \to 0$ be an extension
	in $\Ab(k^{\rat}_{\et})$.
	By pulling back by the inclusion $\xi_{A} \into A$,
	we have an element of $H^{1}((\xi_{A})_{\et}, B)$,
	which is trivial by assumption.
	Hence we have a section $s \colon \xi_{A} \to C$ (in $\Set(k^{\rat}_{\et})$)
	to the projection $C \to A$.
	Consider the associated factor system
		\begin{gather*}
					f
				\colon
					\xi_{A \times A}
				\to
					B,
			\\
					(x_{1}, x_{2})
				\mapsto
					s(x_{1}) + s(x_{2}) - s(x_{1} + s_{2}).
		\end{gather*}
	We have equalities
		\begin{gather*}
					f(x_{2}, x_{3})
				-
					f(x_{1} + x_{2}, x_{3})
				+
					f(x_{1}, x_{2} + x_{3})
				-
					f(x_{1}, x_{2})
				=
					0,
			\\
					f(x_{1}, x_{2})
				=
					f(x_{2}, x_{1})
		\end{gather*}
	as morphisms $\xi_{A \times A \times A} \to B$
	and $\xi_{A \times A} \to B$.
	This means that $f$ as a rational map $A \times A \dashrightarrow B$
	is a rational symmetric factor system (\cite[VII, \S 1.4]{Ser88}).
	Hence $f$ defines a birational group,
	which comes from a true quasi-algebraic group $D$ by Weil's theorem,
	which fits in an extension $0 \to B \to D \to A \to 0$ of quasi-algebraic groups (\cite[VII, \S 1.4]{Ser88}).
	By construction, we have a rational section $A \dashrightarrow D$
	whose associated rational factor system agrees with $f$.
	Consider the Baer difference of the two extensions
		\[
				[0 \to B \to E \to A \to 0]
			:=
					[0 \to B \to C \to A \to 0]
				-
					[0 \to B \to D \to A \to 0]
		\]
	in $\Ab(k^{\rat}_{\et})$.
	This admits a section $t \colon \xi_{A} \to E$ (in $\Set(k^{\rat}_{\et})$)
	whose associated factor system is $f - f = 0$.
	This means the equality
		\begin{gather*}
					t(x_{1} + x_{2})
				=
					t(x_{1}) + t(x_{2})
				\text{ in }
					E(k')
			\\
				\text{ for }
					(x_{1}, x_{2}) \in \xi_{A^{2}}(k').
		\end{gather*}
	We want to extend $t \colon \xi_{A} \to E$
	to a homomorphism $\tilde{t} \colon A \to E$ in $\Ab(k^{\rat}_{\et})$ in a unique way.
	We need two sublemmas.
	
	\begin{Sublem} \label{sublem: generic homomorphism property} \BetweenThmAndList
		\begin{enumerate}
			\item \label{eq: generic homomorphism property}
					$
							t(x_{1} + \dots + x_{m})
						=
							t(x_{1}) + \dots + t(x_{m})
					$
				in $E(k')$ for $m \ge 2$ and
				$x_{1}, \dots, x_{m} \in \xi_{A}(k')$
				with
				$x_{1} + \dots + x_{m} \in \xi_{A}(k')$.
			\item \label{eq: generic inverse property}
				$t(- y) = - t(y)$ in $E(k')$
				for
				$y \in \xi_{A}(k')$.
		\end{enumerate}
	\end{Sublem}
	
	\begin{proof}[Proof of Sublemma \ref{sublem: generic homomorphism property}]
		\eqref{eq: generic homomorphism property}
		We prove this only for $m = 2$ as the general case is similar.
		As before, let $A_{k'} = \Spec k' \times_{k} A$
		and $k''$ the function field of $A_{k'}$.
		The natural projection $A_{k'} \to A$ gives a point $x \in \xi_{A}(k'')$.
		The morphism $x_{2} \colon \Spec k' \to A$
		corresponding to the point $x_{2} \in \xi_{A}(k')$ is flat.
		Taking the fiber product of this morphism with $A$,
		we have a flat morphism $A_{k'} \to A^{2}$,
		hence a morphism $\Spec k'' \to \xi_{A^{2}}$,
		which corresponds to the point $(x_{2}, x) \in \xi_{A^{2}}(k'')$.
		Similarly we have $(x_{1} + x_{2}, x) \in \xi_{A^{2}}(k'')$.
		The direct sum of the identity map $A_{k'} \to A_{k'}$
		and the addition $A_{k'} \to A_{k'}$ by $x_{2}$
		gives a $k'$-scheme automorphism
		$A_{k'}^{2} \isomto A_{k'}^{2}$.
		This sends $(x_{1}, x) \in A_{k'}^{2}(k'') = A^{2}(k'')$ to $(x_{1}, x_{2} + x)$.
		Hence $(x_{1}, x) \in \xi_{A^{2}}(k'')$ implies
		$(x_{1}, x_{2} + x) \in \xi_{A^{2}}(k'')$.
		Therefore
			\[
					t(x_{1} + x_{2}) + t(x)
				=
					t(x_{1} + x_{2} + x)
				=
					t(x_{1}) + t(x_{2} + x)
				=
					t(x_{1}) + t(x_{2}) + t(x),
			\]
		and hence the element
		$t(x_{1}) + t(x_{2}) - t(x_{1} + x_{2})$
		of $B(k')$ becomes zero in $B(k'')$.
		Since $B(k') \to B(k'')$ is injective,
		we have $t(x_{1}) + t(x_{2}) = t(x_{1} + x_{2})$ in $E(k')$.
		
		\eqref{eq: generic inverse property}
		Using \eqref{eq: generic homomorphism property} for $m = 3$,
		we have
			$
					t(y) + t(- y) + t(y)
				=
					t(y),
			$
		so $t(-y) = - t(y)$.
	\end{proof}
	
	\begin{Sublem} \label{sublem: an etale extension of rational sheaves is a rational sheaf}
		For any field extension $k'' / k'$ in $k^{\rat}$,
		the sequence
			\[
					0
				\to
					E(k')
				\to
					E(k'')
				\to
					E(\Frac(k'' \tensor_{k'} k''))
			\]
		is exact, where $\Frac$ denotes the total quotient ring,
		and the first homomorphism comes from the inclusion $k' \into k''$
		and the second the difference of the two homomorphisms coming from
		the inclusions $k'' \rightrightarrows \Frac(k'' \tensor_{k'} k'')$ into the two factors.%
		\footnote{This sublemma says that $E$ is a sheaf for the rational flat topology
		defined in Section \ref{sec: the flat case}.
		See Proposition \ref{prop: sheaf condition in the rational flat site}.
		The argument below shows that $A$ and $B$ are sheaves for the rational flat topology.
		Hence so is $E$.}
	\end{Sublem}
	
	\begin{proof}[Proof of Sublemma \ref{sublem: an etale extension of rational sheaves is a rational sheaf}]
		Consider the similar sequences
			\begin{gather*}
					0
				\to
					A(k')
				\to
					A(k'')
				\to
					A(\Frac(k'' \tensor_{k'} k'')),
				\\
					0
				\to
					B(k')
				\to
					B(k'')
				\to
					B(\Frac(k'' \tensor_{k'} k'')).
			\end{gather*}
		These are exact, since $A$ and $B$ are quasi-algebraic groups,
		and a rational map $g$ from a variety to a quasi-algebraic group (which in particular is separated) such that
		$g(x_{1}) = g(x_{2})$ for independent generic points $(x_{1}, x_{2})$
		is constant everywhere.
		Also the sequence
			\[
				0 \to B(k') \to E(k') \to A(k') \to 0
			\]
		is exact since $H^{1}(k'_{\et}, B) = 0$.
		These yield the required exactness by a diagram chase.
	\end{proof}
	
	Now let $a \in A(k')$.
	As above,
	let $A_{k'} = \Spec k' \times_{k} A$
	and $k''$ the function field of $A_{k'}$.
	The natural projection $A_{k'} \to A$ gives a point $x \in \xi_{A}(k'')$,
	which satisfies $a - x \in \xi_{A}(k'')$.
	We define
		\[
				t(a)
			=
				t(a - x) + t(x)
			\in
				E(k'').
		\]
	We show that $t(a) \in E(k')$.
	The two inclusions $k'' \rightrightarrows \Frac(k'' \tensor_{k'} k'')$
	and the point $x \in \xi_{A}(k'')$
	defines two points $x_{1}, x_{2} \in \xi_{A}(\Frac(k'' \tensor_{k'} k''))$.
	In view of Sublemma \ref{sublem: an etale extension of rational sheaves is a rational sheaf},
	we have to show that $t(a - x_{1}) + t(x_{1}) = t(a - x_{2}) + t(x_{2})$.
	This follows from \eqref{eq: generic homomorphism property} and
	\eqref{eq: generic inverse property} of Sublemma \ref{sublem: generic homomorphism property}.
	We show that $t(a + b) = t(a) + t(b)$ for $a, b \in A(k')$.
	This is equivalent to the equality
	$t(a + b - x) = t(a - x) + t(b - x) + t(x)$
	in $E(k'')$,
	which follows from \eqref{eq: generic homomorphism property}
	of Sublemma \ref{sublem: generic homomorphism property}.
	The homomorphism $A \to E$ thus obtained is
	the only possible extension of $t$.

	Therefore the extension
	$0 \to B \to E \to A \to 0$ in $\Ab(k^{\rat}_{\et})$ is trivial.
	Hence the two extensions
	$0 \to B \to C \to A \to 0$ and
	$0 \to B \to D \to A \to 0$ are equivalent.
	Since $D$ is a quasi-algebraic group,
	this proves that
	$\Ext_{\Alg / k}^{1}(A, B) = \Ext_{k^{\rat}_{\et}}^{1}(A, B)$.
\end{proof}

\begin{Rmk}
	The above proof works for an arbitrary commutative quasi-algebraic group $A$
	and affine commutative quasi-algebraic group $B$.
	To allow a non-affine quasi-algebraic group $B$,
	we need the following modification.
	An element of $\Ext_{k^{\rat}_{\et}}^{1}(A, B)$ defines a non-zero element $x$ of
	$H^{1}((\xi_{A})_{\et}, B)$ in general.
	This element is primitive in the sense of \cite[VII, \S 3.14]{Ser88},
	that is, $s^{\ast} x = \proj_{1}^{\ast} x + \proj_{2}^{\ast} x$ in $H^{1}((\xi_{A^{2}})_{\et}, B)$,
	where $s, \proj_{1}, \proj_{2} \colon \xi_{A^{2}} \to \xi_{A}$
	are the group operation, first and second projections, respectively.
	There is a homomorphism from the subgroup of primitive classes of $H^{1}((\xi_{A})_{\et}, B)$
	to a certain subquotient of $H^{0}((\xi_{A^{3}})_{\et}, B) \oplus H^{0}((\xi_{A^{2}})_{\et}, B)$,
	whose kernel consists of the classes coming from $\Ext_{k^{\rat}_{\et}}^{1}(A, B)$.
	This homomorphism comes from Eilenberg-Mac Lane's abelian complex (\cite{Bre69}).
	From such an element $x$, we can again define a birational group
	just as in the proof of \cite[VII, \S 3.15, Theorem 5]{Ser88}.
	
	In the next section, we will use the cubical construction and Mac Lane's resolution instead.
	With this, we can treat arbitrary commutative quasi-algebraic groups $A, B$
	and all higher Ext groups.
\end{Rmk}


\subsection{The relative fppf site of a local field}
\label{sec: The relative fppf site of a local field}

Let $K$ be a complete discrete valuation field
with ring of integers $\Order_{K}$ and perfect residue field $k$ of characteristic $p > 0$.
We denote by $W$ the affine ring scheme of Witt vectors of infinite length.
As in Introduction,
we define sheaves of rings on the site $\Spec k^{\ind\rat}_{\et}$
by assigning to each $k' \in k^{\ind\rat}$,
	\[
			\alg{O}_{K}(k')
		=
			W(k') \hat{\tensor}_{W(k)} \Order_{K},
		\quad
			\alg{K}(k')
		=
			\alg{O}_{K}(k') \tensor_{\Order_{K}} K.
	\]
We define a category $K / k^{\ind\rat}$ as follows.
An object is a pair $(S, k_{S})$,
where $k_{S} \in k^{\ind\rat}$ and $S$ is a $\alg{K}(k_{S})$-algebra of finite presentation.
A morphism $(S, k_{S}) \to (S', k_{S'})$ consists of
a $k$-algebra homomorphism $k_{S} \to k_{S'}$ and a ring homomorphism $S \to S'$
such that the diagram
	\[
		\begin{CD}
				\alg{K}(k_{S})
			@>>>
				\alg{K}(k_{S'})
			\\
			@VVV
			@VVV
			\\
				S
			@>>>
				S'
		\end{CD}
	\]
commutes.
The composite of two morphisms is defined in the obvious way.
We say that a morphism $(S, k_{S}) \to (S', k_{S'})$ is \emph{flat/\'etale}
if $S \to S'$ is flat and $k_{S} \to k_{S'}$ is \'etale.

\begin{Prop} \label{prop: canonical lifts and coverings}
	Let $(S, k_{S}) \to (S', k_{S'})$, $(S, k_{S}) \to (S'', k_{S''})$
	be two morphisms in $K / k^{\ind\rat}$.
	Assume that the first one is flat/\'etale.
	Then we have
		$
				\alg{K}(k_{S'}) \tensor_{\alg{K}(k_{S})} \alg{K}(k_{S''})
			=
				\alg{K}(k_{S'} \tensor_{k_{S}} k_{S''})
		$.
	The pair $(S' \tensor_{S} S'', k_{S'} \tensor_{k_{S}} k_{S''})$
	is the fiber sum of $(S', k_{S'})$ and $(S'', k_{S''})$ over $(S, k_{S})$
	in the category $K / k^{\ind\rat}$,
	and is flat/\'etale over $(S'', k_{S''})$.
\end{Prop}

\begin{proof}
	For the equality
		$
				\alg{K}(k_{S'}) \tensor_{\alg{K}(k_{S})} \alg{K}(k_{S''})
			=
				\alg{K}(k_{S'} \tensor_{k_{S}} k_{S''})
		$,
	note that we can write $k_{S'} = k'_{1} \tensor_{k_{1}} k_{S}$
	with $k_{1}$ a rational $k$-subalgebra of $k_{S}$ and $k'_{1}$ \'etale over $k_{1}$.
	Writing $k_{1}$ as a finite product of fields $k_{1, i}$,
	we know that $k_{1} \to k'_{1}$ is a finite product of
	finite free \'etale homomorphisms $k_{1, i} \to k'_{1, i}$
	(where $k'_{1, i}$ can be zero).
	Hence $k_{S} \to k_{S'}$ can be written as a finite product of
	finite free \'etale homomorphisms $k_{S, i} \to k_{S', i}$.
	Therefore $W(k_{S}) \to W(k_{S'})$ can be written as the product of
	the finite free \'etale homomorphisms $W(k_{S, i}) \to W(k_{S', i})$.
	Hence the equality
		$
				\alg{K}(k_{S'}) \tensor_{\alg{K}(k_{S})} \alg{K}(k_{S''})
			=
				\alg{K}(k_{S'} \tensor_{k_{S}} k_{S''})
		$
	follows.
	The rest is obvious.
\end{proof}

\begin{Def}
	We define the \emph{relative fppf site of $K$ over $k$},
	 $\Spec K_{\fppf} / k^{\ind\rat}_{\et}$,
	to be the category $K / k^{\ind\rat}$
	endowed with the topology whose covering families
	over an object $(S, k_{S}) \in K_{\fppf} / k^{\ind\rat}$
	are finite families $\{(S_{i}, k_{S_{i}})\}$ each flat/\'etale over $(S, k_{S})$
	with $\prod_{i} S_{i}$ faithfully flat over $S$.
\end{Def}

Note that in this definition,
$\prod_{i} k_{S_{i}}$ is not required to be faithfully flat over $k_{S}$.

\begin{Prop} \label{prop: etale residual extensions induce bijections on sections}
	Let $k' \to k''$ be \'etale in $k^{\ind\rat}$,
	$(S, k'') \in K / k^{\ind\rat}$
	and $A \in \Ab(K_{\fppf} / k^{\ind\rat}_{\et})$.
	Then the natural morphism $(S, k') \to (S, k'')$ in $K / k^{\ind\rat}$ induces an isomorphism
	$A(S, k') \isomto A(S, k'')$.
\end{Prop}

\begin{proof}
	The object $(S, k'')$ covers $(S, k')$.
	The sheaf condition says that the sequence
		\[
				0
			\to
				A(S, k')
			\to
				A(S, k'')
			\to
				A(S, k'' \tensor_{k'} k'')
		\]
	is exact.
	Since $k''$ is \'etale over $k'$,
	the diagonal homomorphism $k'' \tensor_{k'} k'' \onto k''$ is \'etale.
	Hence $(S, k'')$ covers $(S, k'' \tensor_{k'} k'')$.
	In particular, the homomorphism $A(S, k'' \tensor_{k'} k'') \to A(S, k'')$ is injective.
	This and the above sequence imply the result.
\end{proof}

A little closer observation shows that
a representable presheaf
$F = (S, k_{S})$ on $\Spec K_{\fppf} / k^{\ind\rat}_{\et}$ is a sheaf
if and only if $k_{S} = k$.
In particular, an affine scheme $\Spec S$ finitely presented over $K$
can be regarded as a sheaf on $\Spec K_{\fppf} / k^{\ind\rat}_{\et}$
by identifying it with the sheaf $(S, k)$.
By patching, any locally finitely presented scheme over $K$
can be regarded as a sheaf on $\Spec K_{\fppf} / k^{\ind\rat}_{\et}$.
The sheafification of $F = (S, k_{S})$ sends an object $(S', k_{S'})$ to
the filtered direct limit of $F(S', k')$,
where the limit is indexed by pairs $k'$ (an \'etale $k_{S'}$-algebra) and
$\alg{K}(k') \to S'$ (a $\alg{K}(k_{S'})$-algebra homomorphism).
When we write $\Z[F]$, we mean
the sheafification of the presheaf of free abelian groups generated by $F$.

The cohomology theory of the site $\Spec K_{\fppf} / k^{\ind\rat}_{\et}$
is essentially fppf cohomology of $K$, as follows.

\begin{Prop} \label{prop: cohomology of the relative site}
	Let $(S, k_{S}) \in K / k^{\ind\rat}$.
	Let $\Spec S_{\fppf}$ be the fppf site of the ring $S$
	(on the category of finitely presented $S$-algebras).
	Let $(\Spec K_{\fppf} / k^{\ind\rat}_{\et}) / (S, k_{S})$ be
	the localization of $\Spec K_{\fppf} / k^{\ind\rat}_{\et}$ at the object $(S, k_{S})$.
	Finally let
		\[
				f
			\colon
				(\Spec K_{\fppf} / k^{\ind\rat}_{\et}) / (S, k_{S})
			\to
				\Spec S_{\fppf}
		\]
	be the morphism of sites defined by sending a finitely presented $S$-algebra $S'$
	to the object $(S', k_{S})$ of $K / k^{\ind\rat}$.
	Then $f_{\ast}$ is exact.
	We have
		\[
				R \Gamma((S, k_{S})_{\fppf}, A)
			=
				R \Gamma(S_{\fppf}, f_{\ast} A)
		\]
	for any sheaf $A$ of abelian groups on the site
	$(\Spec K_{\fppf} / k^{\ind\rat}_{\et}) / (S, k_{S})$,
	where the left-hand side is the cohomology of this site
	(at the final object $(S, k_{S})$) with coefficients in $A$.
\end{Prop}

\begin{proof}
	First $f$ is a morphism of sites
	since the underlying category of the target $\Spec S_{\fppf}$ has all finite limits
	and the functor $S' \mapsto (S', k_{S})$ commutes with these limits (or tensor products).
	We show the exactness of $f_{\ast}$.
	Let $\varphi \colon A \onto B$ be a surjection
	of sheaves on  $\Spec(K_{\fppf} / k^{\ind\rat}_{\et}) / (S, k_{S})$,
	$S'$ a finitely presented $S$-algebra
	and $b \in (f_{\ast} B)(S') = B(S', k_{S})$.
	The surjectivity says that
	there exists a cover $(S'', k_{S''})$ of $(S', k_{S})$
	and an element $a \in A(S'', k_{S''})$ such that
	$\varphi(a) = b$ in $B(S'', k_{S''})$.
	By Proposition \ref{prop: etale residual extensions induce bijections on sections},
	we may assume that $k_{S''} = k_{S}$.
	Then $a \in (f_{\ast} A)(S'')$ and $\varphi(a) = b \in (f_{\ast} B)(S'')$.
	Therefore $f_{\ast} A \to f_{\ast} B$ is surjective.
	Hence $f_{\ast}$ is exact.
	The equality of cohomology complexes then follows from the Grothendieck spectral sequence.
\end{proof}

The above is enough for describing the cohomology theory of $\Spec K_{\fppf} / k^{\ind\rat}_{\et}$;
see the paragraph after Proposition \ref{prop: rational etale cohomology is Galois cohomology}.


\subsection{The structure morphism of a local field and the cup product pairing}
\label{sec: The structure morphism of a local field and the cup product pairing}

\begin{Prop}
	The functor $k^{\ind\rat} \to K / k^{\ind\rat}$ defined by
		\[
				k'
			\mapsto
				(\alg{K}(k'), k')
		\]
	sends \'etale coverings to \'etale (hence fppf) coverings
	and has a right adjoint given by
		\[
				(S, k_{S})
			\mapsto
				k_{S}.
		\]
\end{Prop}

\begin{proof}
	If $k'' / k'$ is an \'etale covering in $k^{\ind\rat}$,
	then
		$
				\alg{K}(k'') \tensor_{\alg{K}(k')} \alg{K}(k'')
			=
				\alg{K}(k'' \tensor_{k'} k'')
		$
	and $\alg{K}(k'') / \alg{K}(k')$ is an \'etale covering,
	as we saw in Proposition \ref{prop: canonical lifts and coverings} and its proof.
	The adjointness is obvious.
\end{proof}

\begin{Def}
	The functor $k' \mapsto (\alg{K}(k'), k')$ above defines a morphisms of sites
		\[
				\pi
			\colon
				\Spec K_{\fppf} / k^{\ind\rat}_{\et}
			\to
				\Spec k^{\ind\rat}_{\et}.
		\]
	We call this the (fppf) \emph{structure morphism of $K$ over $k$}.
	We denote
		\begin{gather*}
					\alg{\Gamma}(K_{\fppf}, \;\cdot\;)
				=
					\pi_{\ast},
				\quad
					\alg{H}^{i}(K_{\fppf}, \;\cdot\;)
				=
					R^{i} \pi_{\ast},
			\\
					R \alg{\Gamma}(K_{\fppf}, \;\cdot\;)
				=
					R \pi_{\ast}
				\colon
					D(K_{\fppf} / k^{\ind\rat}_{\et})
				\to
					D(k^{\ind\rat}_{\et})
		\end{gather*}
\end{Def}

Note that the pullback $\pi^{\ast}$ sends
a sheaf $A$ to the sheafification of the presheaf $(S, k_{S}) \mapsto A(k_{S})$ by the above proposition.
This is exact,
so $\pi$ is indeed a morphism of sites.
A remark is that by forgetting $k'$ from $(\alg{K}(k'), k')$,
we have a continuous map
from the fppf site of $K$-algebras to the site
$\Spec K_{\fppf} / k^{\ind\rat}_{\et}$.
But neither of it or its composite with our morphisms $\pi$
has an obvious reason to be a morphism of sites.
Therefore we only use $\pi \colon \Spec K_{\fppf} / k^{\ind\rat}_{\et} \to \Spec k^{\ind\rat}_{\et}$.
For the four operation formalism (about pushforward, pullback, Hom and tensor product)
for an arbitrary morphism of sites and unbounded derived categories,
see \cite[Chapter 18]{KS06}.

In a concrete term, we have
$\alg{\Gamma}(K_{\fppf}, A)(k') = A(\alg{K}(k'), k')$
for any $k' \in k^{\ind\rat}$ and $A \in \Ab(K_{\fppf} / k^{\ind\rat}_{\et})$,
and the sheaf $\alg{H}^{i}(K_{\fppf}, A)$ on $\Spec k^{\ind\rat}_{\et}$
is the \'etale sheafification of the presheaf
	\[
			k'
		\mapsto
			H^{i}(\alg{K}(k')_{\fppf}, f_{\ast} A),
	\]
where $(f_{\ast} A)(S) = A(S, k')$ as in Proposition \ref{prop: cohomology of the relative site}.
If $A$ is a locally algebraic group scheme over $K$
viewed as a sheaf on $\Spec K_{\fppf} / k^{\ind\rat}_{\et}$
(see the paragraph after Proposition \ref{prop: etale residual extensions induce bijections on sections}),
then $A(S, k') = A(S)$, hence $f_{\ast} A = A$.
Recall that if $k' \in k^{\rat}$, then
$\alg{K}(k')$ is a finite direct product of complete discrete valuation fields with perfect residue field.
Therefore $R \alg{\Gamma}(K_{\fppf}, A)$,
at least when restricted to the site $\Spec k^{\rat}_{\et}$,
is essentially a classical object
when $A$ is a locally algebraic group scheme over $K$.
In the next subsection, we will need to understand
cohomology of $\alg{K}(k')$ when $k'$ is ind-rational but not necessarily rational.

To simplify the notation, we will write
	\begin{gather*}
				\alg{\Gamma}(\;\cdot\;) = \pi_{\ast},
				\alg{H}^{i}(\;\cdot\;) = R^{i} \pi_{\ast}
			\colon
				\Ab(K)
			\to
				\Ab(k),
		\\
				R \alg{\Gamma}(\;\cdot\;) = R \pi_{\ast}
			\colon
				D(K)
			\to
				D(k)
	\end{gather*}
when there is no confusion.
Let $\sheafhom$ (resp.\ $\sheafext^{i}$) denote
the sheaf-Hom (resp.\ the $i$-th sheaf-Ext)
and $R \sheafhom$ its derived version.
We want to define a morphism
	\[
			R \alg{\Gamma}(\Z)
		\to
			R \sheafhom_{k}(\alg{K}^{\times}, \Z)
	\]
of duality with coefficients in $\Gm$.
We need two propositions.

\begin{Prop} \label{prop: functoriality}
	There exists a canonical morphism
		\begin{equation} \label{eq: functoriality}
				R \alg{\Gamma}
				R \sheafhom_{K}(A, B)
			\to
				R \sheafhom_{k}(R \alg{\Gamma}(A), R \alg{\Gamma}(B))
		\end{equation}
	in $D(k)$ for $A, B \in D(K)$,
	called the \emph{morphism of functoriality of $R \alg{\Gamma}$}.
	If a sequence
		\begin{equation} \label{eq: ext class}
				0 \to B \to C_{n} \to \dots \to C_{1} \to A \to 0
		\end{equation}
	in $\Ab(K)$ is exact with
	$\alg{H}^{i}(A) = \alg{H}^{i}(B) = \alg{H}^{i}(C_{j}) = 0$ for any $i \ge 1$ and any $j$,
	then the sequence
		\begin{equation} \label{eq: image of ext class}
				0 \to \alg{\Gamma}(B) \to \alg{\Gamma}(C_{n}) \to \dots \to \alg{\Gamma}(C_{1}) \to \alg{\Gamma}(A) \to 0
		\end{equation}
	is exact, and
	the extension class of \eqref{eq: ext class} maps to that of \eqref{eq: image of ext class}
	via the induced map
		\begin{equation} \label{eq: hom between Ext}
				\Ext_{K}^{n}(A, B)
			\to
				\Ext_{k}^{n}(\alg{\Gamma}(A), \alg{\Gamma}(B)).
		\end{equation}
\end{Prop}

\begin{proof}
	The adjunction between $\pi^{\ast}$ and $R \pi_{\ast} = R \alg{\Gamma}$
	induces morphisms
		\[
				R \alg{\Gamma}
				R \sheafhom_{K}(A, B)
			\to
				R \alg{\Gamma}
				R \sheafhom_{K}(\pi^{\ast} R \alg{\Gamma}(A), B)
			\to
				R \sheafhom_{k}(R \alg{\Gamma}(A), R \alg{\Gamma}(B)).
		\]
	See \cite[Theorem 18.6.9]{KS06} for this setting of a morphism of sites,
	sheaf-Hom and unbounded derived categories.
	This construction shows that \eqref{eq: hom between Ext} is the homomorphism
		\[
				\Hom_{D(K)}(A, B[n])
			\to
				\Hom_{D(k)}(\alg{\Gamma}(A), \alg{\Gamma}(B[n])).
		\]
	The extension class of \eqref{eq: ext class} corresponds to the morphism
		\[
				A
			\overset{\sim}{\leftarrow}
				[B \to C_{n} \to \dots \to C_{1}]
			\to
				B[n],
		\]
	where the $B$ in the complex in the middle is placed at degree $-n$.
	The assumptions show that
		\[
				R \alg{\Gamma} [B \to C_{n} \to \dots \to C_{1}]
			=
				[\alg{\Gamma}(B) \to \alg{\Gamma}(C_{n}) \to \dots \to \alg{\Gamma}(C_{1})].
		\]
	Hence we obtain the corresponding morphism
		\[
				\alg{\Gamma}(A)
			\overset{\sim}{\leftarrow}
				[\alg{\Gamma}(B) \to \alg{\Gamma}(C_{n}) \to \dots \to \alg{\Gamma}(C_{1})]
			\to
				\alg{\Gamma}(B[n]).
		\]
\end{proof}

Note that the morphism of functoriality of $R \alg{\Gamma}$ is equivalent to the cup-product pairing
	\[
				R \alg{\Gamma}(A)
			\tensor_{k}^{L}
				R \alg{\Gamma}(C)
		\to
				R \alg{\Gamma}(A \tensor_{K}^{L} C)
	\]
(where $\tensor_{k}^{L}$ and $\tensor_{K}^{L}$ denote
the derived tensor products over $k$ and $K$, respectively)
by the derived tensor-hom adjunction \cite[Theorem 18.6.4 (vii)]{KS06}
via the change of variables
$R \sheafhom_{K}(A, B) \leadsto C$ and
$A \tensor_{K}^{L} C \leadsto B$.
We prefer Ext groups rather than Tor for the treatment of proalgebraic groups.

Before the next proposition, we recall the functorial valuation map from \cite[\S 4.1]{SY12}.
(The definition of this map does not depend on the choice of the topology, fpqc or \'etale.)
For $k' \in k^{\ind\rat}$ and $\ideal{m} \in \Spec k'$,
the ring $\alg{K}(k' / \ideal{m})$ is a complete discrete valuation field extending $K$.
We denote by $v_{\ideal{m}}$ the composite of the natural surjection
$\alg{K}(k')^{\times} \onto \alg{K}(k' / \ideal{m})^{\times}$
and the normalized valuation $\alg{K}(k' / \ideal{m})^{\times} \onto \Z$.
For $\alpha \in \alg{K}(k')^{\times}$, the map
$\ideal{m} \mapsto v_{\ideal{m}}(\alpha)$
is a locally constant $\Z$-valued function on the underlying topological space of $\Spec k'$
(\cite[Proposition 4.1.1]{SY12}).
This defines a morphism of sheaves $\alg{K}^{\times} \to \Z$ in $\Ab(k^{\ind\rat}_{\et})$,
which we call the \emph{valuation map} for $\alg{K}^{\times}$.
Denote $\alg{U}_{K} = \alg{O}_{K}^{\times}$ as in Introduction.
The sequence
	\[
			0
		\to
			\alg{U}_{K}
		\to
			\alg{K}^{\times}
		\to
			\Z
		\to
			0
	\]
in $\Ab(k^{\ind\rat}_{\et})$ is split exact (\cite[Proposition 4.1.2]{SY12}).

\begin{Prop} \label{prop: cohomology of Gm}
	We have 
		\[
				R \alg{\Gamma}(\Gm)
			=
				\alg{\Gamma}(\Gm)
			=
				\alg{K}^{\times}.
		\]
	We call the composite of this isomorphism and the valuation map $\alg{K}^{\times} \onto \Z$ (see above)
	the \emph{trace map}.
\end{Prop}

We will prove this proposition in the next subsection.
Note that the sheaf
$\alg{H}^{i}(\Gm)$ on $\Spec k^{\ind\rat}_{\et}$ is the \'etale sheafification of the presheaf
	\[
			k' \in k^{\ind\rat}
		\mapsto
			H^{i}(\alg{K}(k')_{\fppf}, \Gm)
		=
			H^{i}(\alg{K}(k')_{\et}, \Gm),
	\]
where we used the fact that the fppf cohomology with coefficients in a smooth group scheme
agrees with the \'etale cohomology
(\cite[III, Remark 3.11 (b)]{Mil80}).
When restricted to $\Spec k^{\rat}_{\et}$,
its vanishing for $i \ge 1$ follows from
classical results on Galois cohomology of complete discrete valuation fields
with algebraically closed residue field:
for each field $k' \in k^{\rat}$ and $i \ge 1$, we have
	\[
			H^{i}(\alg{K}(\closure{k'})_{\et}, \Gm)
		=
			H^{i}(\alg{K}(k')^{\ur}_{\et}, \Gm)
		=
			0
	\]
by \cite[V, \S 4, Proposition 7, and X, \S 7, Proposition 11]{Ser79},
where $\alg{K}(k')^{\ur}$ is the maximal unramified extension of
the complete discrete valuation field $\alg{K}(k')$
and $\alg{K}(\closure{k'})$ its completion.

Using Propositions \ref{prop: functoriality} and \ref{prop: cohomology of Gm},
we have a desired morphism
	\[
			R \alg{\Gamma}(\Z)
		\to
			R \alg{\Gamma} R \sheafhom_{K}(\Gm, \Gm)
		\to
			R \sheafhom_{k}(\alg{K}^{\times}, \alg{K}^{\times})
		\to
			R \sheafhom_{k}(\alg{K}^{\times}, \Z).
	\]

\begin{Rmk}
	The site $\Spec K_{\fppf} / k^{\ind\rat}_{\et}$ has the following relationship with
	oriented products of sites defined by Laumon \cite[\S 3.1.3]{Lau83}.
	Consider the continuous map $\pi_{0} \colon \Spec K_{\fppf} \to \Spec k^{\ind\rat}_{\et}$
	defined by the functor $k' \mapsto \alg{K}(k')$
	(here we allow all $K$-algebras in the underlying category of the site $\Spec K_{\fppf}$).
	We denote the empty site by $\emptyset$.
	Then we have two continuous maps
	$\Spec K_{\fppf} \to \Spec k^{\ind\rat}_{\et} \gets \emptyset$.
	Their oriented product
	$\Spec K_{\fppf} \stackrel{\gets}{\times}_{\Spec k^{\ind\rat}_{\et}} \emptyset$
	can be defined as in loc.\ cit.,
	even though $\pi_{0}$ is not a morphism of sites,
	and it agrees with our site $\Spec K_{\fppf} / k^{\ind\rat}_{\et}$.
	The projection
		$
				\Spec K_{\fppf} \stackrel{\gets}{\times}_{\Spec k^{\ind\rat}_{\et}} \emptyset
			\to
				\Spec k^{\ind\rat}_{\et}
		$
	agrees with our morphism $\pi$ of sites.
	
	More generally,
	let $S, S'$ be sites defined by pretopologies.
	Let $\pi_{0} \colon S' \to S$ be a continuous map
	whose underlying functor (also denoted by $\pi_{0}$)
	sends coverings to coverings
	such that $\pi_{0}(Y \times_{X} Z) = \pi_{0}(Y) \times_{\pi_{0}(X)} \pi_{0}(Z)$
	if $Y \to X$ in $S$ appears in a covering family.
	Then we can define the relative site $S' / S$ in the same way as above
	and it agrees with $S' \stackrel{\gets}{\times}_{S} \emptyset$.
	The morphism $\pi \colon S' / S \to S$ of sites is similarly defined
	and agrees with the projection $S' \stackrel{\gets}{\times}_{S} \emptyset \to S$.
	The effect of multiplying $\emptyset$ is to force a continuous map to have an exact pullback functor.
	Assume moreover that if $Y \to X$ is a morphism in $S$ that appears in a covering family,
	then the diagonal morphism $Y \to Y \times_{X} Y$ also appears in a covering family.
	Then the cohomology theory of $S' / S$ is essentially that of $S'$
	in the sense of Proposition \ref{prop: cohomology of the relative site}.
\end{Rmk}


\subsection{Cohomology of local fields with ind-rational base}
\label{sec: Cohomology of local fields with ind-rational base}

To prove Proposition \ref{prop: cohomology of Gm},
we first need to understand the \'etale site of $\alg{K}(k')$ when $k' \in k^{\ind\rat}$.
We define subsheaves
$\alg{O}_{K}^{0} \subset \alg{O}_{K}$ and $\alg{K}^{0} \subset \alg{K}$
by
	\[
			\alg{O}_{K}^{0}(k')
		=
			\bigcup_{\lambda} \alg{O}_{K}(k'_{\lambda}),
		\quad
			\alg{K}^{0}(k')
		=
			\bigcup_{\lambda} \alg{K}(k'_{\lambda}).
	\]
for each $k' = \bigcup_{\lambda} k'_{\lambda} \in k^{\ind\rat}$
with $k'_{\lambda} \in k^{\rat}$.
The strategy is to compare $\alg{K}(k')_{\et}$ with $\alg{K}^{0}(k')_{\et}$,
the latter of which is described by the \'etale sites of complete discrete valuation fields.
The argument in this comparison goes basically in the same line
as the proof of Krasner's lemma using Hensel's lemma.
Additional complications come from the underlying topological space of $\Spec k'$,
which is a profinite space.
We treat this topology and the topology coming from the valuation simultaneously.
Then we will be reduced to considering the cohomology of the pushforward of $\Gm$
from $\alg{K}(k')_{\et}$ to $\alg{K}^{0}(k')_{\et}$.
The computation of this is essentially classical.
Up to a notational preparation,
we only need to recall the fact that for a finite extension $L / K$,
any element of $K^{\times}$ becomes a norm in $L^{\times}$
after a finite unramified extension (\cite[V, \S3]{Ser79}).

We need notation and several lemmas to prove Proposition \ref{prop: cohomology of Gm}.
We fix $k' \in k^{\ind\rat}$.
If $\ideal{m}$ is a maximal ideal of $k'$,
then the kernel $\alg{K}(\ideal{m})$ of the natural surjection
$\alg{K}(k') \onto \alg{K}(k' / \ideal{m})$ is a maximal ideal of $\alg{K}(k')$.
Conversely, we have:

\begin{Lem}
	Any maximal ideal of $\alg{K}(k')$ is of the form
	$\alg{K}(\ideal{m})$ for some maximal ideal $\ideal{m}$ of $k'$.
\end{Lem}

\begin{proof}
	Let $\ideal{n}$ be a maximal ideal of $\alg{K}(k')$.
	Let $\ideal{m} \subset k'$ be the ideal given by
	the image of the ideal $\alg{O}_{K}(k') \cap \ideal{n}$
	via the natural surjection $\alg{O}_{K}(k') \onto k'$.
	
	We show that for an element $a \in k'$ to be in $\ideal{m}$,
	it is necessary and sufficient that
	$\omega(a) \in \ideal{n}$, where $\omega$ is the Teichm\"uller lift.
	Clearly this is sufficient.
	For necessity, let $\alpha = \alg{O}_{K}(k') \cap \ideal{n}$ map to $a \in \ideal{m}$.
	Take a prime element $\pi \in \Order_{K}$
	and write $\alpha = \omega(a) + \pi \beta$, where $\beta \in \alg{O}_{K}(k')$.
	Since $k' \in k^{\ind\rat}$,
	we can write $a = u e$ for an idempotent $e \in k'$ and a unit $u \in k'^{\times}$.
	Then
		$
				\omega(e) \alpha
			=
				\omega(e) (\omega(u) + \pi \beta)
		$.
	Since $\omega(u) + \pi \beta \in \alg{O}_{K}(k')^{\times}$
	and $\alpha \in \ideal{n}$,
	we have $\omega(e) \in \ideal{n}$.
	Hence $\omega(a) \in \ideal{n}$.	
	This characterization shows that $\ideal{m} \subset k'$ is a prime ideal.
	It has to be maximal since $k' \in k^{\ind\rat}$.
	To finish the proof, it is enough to show that
	$\ideal{n} \subset \alg{K}(\ideal{m})$.
	Let $\alpha = \sum_{i \ge n} \omega(a_{i}) \pi^{i} \in \ideal{n}$ be any element
	($n \in \Z$).
	Since $\pi \in K^{\times}$, we have $\pi^{- n} \alpha \in \alg{O}_{K}(k') \cap \ideal{n}$.
	Hence $a_{n} \in \ideal{m}$ by definition of $\ideal{m}$.
	Hence $\omega(a_{n}) \pi^{n} \in \ideal{n}$ and
		$
				\alpha - \omega(a_{n}) \pi^{n}
			=
				\sum_{i \ge n + 1} \omega(a_{i}) \pi^{i}
			\in
				\alg{O}_{K}(k') \cap \ideal{n}
		$.
	Inductively, we have $a_{i} \in \ideal{m}$ for all $i$.
	Hence $\alpha \in \alg{K}(\ideal{m})$.
	Therefore $\ideal{n} \subset \alg{K}(\ideal{m})$.
\end{proof}

Hence the maximal spectrum of $\alg{K}(k')$ is in bijection with $\Spec k'$.
(The whole prime spectrum is much different.
Results from \cite{Arn73}, in the equal characteristic case, show that
$\alg{O}_{K}(k') \cong k'[[T]]$ and $\alg{K}(k') \cong k'[[T]][1 / T]$
have infinite Krull dimensions if $k'$ has infinitely many direct factors.)

\begin{Lem} \label{lem: Zariski comparison}
	Let $\ideal{m} \subset k'$ be maximal.
	A neighborhood base of the maximal ideal $\alg{K}(\ideal{m})$ in $\Spec \alg{K}(k')$
	is given by the family $\Spec \alg{K}(k')[1 / \omega(a)] = \Spec \alg{K}(k'[1 / a])$ of open sets,
	where $a \in k' \setminus \ideal{m}$.
	In particular, any Zariski covering of $\Spec \alg{K}(k')$ can be refined
	by a disjoint Zariski covering.
\end{Lem}

\begin{proof}
	Let $\alpha = \sum_{i \ge n} \omega(a_{i}) \pi^{i} \in \alg{K}(k') \setminus \alg{K}(\ideal{m})$
	be any element ($\pi$ a prime),
	so $a_{m} \notin \ideal{m}$ for some $m \ge n$.
	We may assume that $a_{n}, \dots, a_{m - 1} \in \ideal{m}$.
	Let $e_{< m} \in k'$ be the idempotent generating the ideal $(a_{n}, \dots, a_{m - 1})$ of $k'$.
	Write $a_{m} = u_{m} e_{m}$, where $e_{m}$ is an idempotent of $k'$ and $u_{m}$ is a unit of $k'$.
	Then
		\[
				\omega \bigl( (1 - e_{< m}) e_{m} \bigr) \alpha
			=
				\omega \bigl( (1 - e_{< m}) e_{m} \bigr)
				\left(
						\omega(u_{m}) \pi^{m}
					+
						\sum_{i \ge m + 1} \omega(a_{i}) \pi^{i}
				\right).
		\]
	The term in the large brackets is a unit in $\alg{K}(k')$.
	Hence
		\[
				\Spec \alg{K}(k') \bigl[
					1 / \omega \bigl(
						(1 - e_{< m}) e_{m}
					\bigr)
				\bigr]
			\subset
				\Spec \alg{K}(k')[1 / \alpha].
		\]
	This proves the lemma
	since $(1 - e_{< m}) e_{m} \notin \ideal{m}$.
\end{proof}

By a similar argument,
we know that if an element of $\alg{K}(k')$ becomes a unit in $\alg{K}(k' / \ideal{m})$
(resp.\ lies in $\alg{O}_{K}(k' / \ideal{m})$),
then it is a unit in $\alg{K}(k'[1 / a])$ (resp.\ lies in $\alg{O}_{K}(k'[1 / a])$)
for some $a \in k' \setminus \ideal{m}$.

\begin{Lem} \label{lem: refinement}
	Any \'etale covering of $\alg{K}(k')$ can be refined by
	a covering coming from an \'etale covering of $\alg{K}^{0}(k')$.
	More precisely, let $S$ be a faithfully flat \'etale $\alg{K}(k')$-algebra.
	Then there exist a faithfully flat \'etale $\alg{K}^{0}(k')$-algebra $L^{0}$
	and a $\alg{K}(k')$-algebra homomorphism $S \to L^{0} \tensor_{\alg{K}^{0}(k')} \alg{K}(k')$.
\end{Lem}

\begin{proof}
	We may assume that $S$ is a standard \'etale $\alg{K}(k')$-algebra
	by Lemma \ref{lem: Zariski comparison}.
	This means that $S = \alg{K}(k')[x]_{g(x)} / (f(x))$
	for some polynomials $f(x)$ and $g(x)$
	such that $f(x)$ is monic and $f'(x) \in S^{\times}$.
	Fix a maximal ideal $\ideal{m} \subset k'$.
	By the same lemma, it is enough to show the existence of such $L^{0}$
	after localizing $k'$ by an element not in $\ideal{m}$.
	We denote by $f_{\ideal{m}}(x)$ the image of $f(x)$ in $\alg{K}(k' / \ideal{m})[x]$.
	Since $S / S \alg{K}(\ideal{m}) \ne 0$,
	there exists a simple root $\alpha$ (in a separable closure of $\alg{K}(k' / \ideal{m})$)
	of the polynomial $f_{\ideal{m}}(x)$
	such that $g_{\ideal{m}}(\alpha) \ne 0$.
	Let $\alpha_{1}, \alpha_{2}, ...$ be the other roots of $f_{\ideal{m}}(x)$.
	We can take an algebraic element $\beta$ over $\alg{K}^{0}(k' / \ideal{m})$
	arbitrarily close to $\alpha$ since
	$\alg{K}(k' / \ideal{m})$ is the completion of
	$\alg{K}^{0}(k' / \ideal{m})$.
	We choose such $\beta$ so that:
	it is separable over $\alg{K}^{0}(k' / \ideal{m})$;
	$f'_{\ideal{m}}(\beta) \ne 0$;
	and there exists an element $\gamma \in K^{\times}$ such that
	$|| \alpha - \beta || < || \gamma || \le || \alpha_{i} - \beta||$ for all $i$
	($||\,\cdot\,||$ denotes an absolute value).
	Let $h(y) \in \alg{K}^{0}(k' / \ideal{m})[y]$ be the monic minimal polynomial of $\beta$.
	Take a rational $k$-subalgebra $k'' \subset k'$ such that
	$h(y) \in \alg{K}(k'' / k'' \cap \ideal{m})[y]$.
	We may assume that $k'' \cap \ideal{m} = 0$ and so $k''$ is a field
	by localizing $k'$ by an element not in $\ideal{m}$.
	We define $L = \alg{K}(k'')[y] / (h(y)) \cong \alg{K}(k'')[\beta]$.
	This is a finite separable extension of the complete discrete valuation field $\alg{K}(k'')$
	linearly disjoint from $\alg{K}(k' / \ideal{m})$.
	We define $L^{0} = \alg{K}^{0}(k')[y] / (h(y))$.
	This is a faithfully flat \'etale $\alg{K}^{0}(k')$-algebra.
	Since $f'_{\ideal{m}}(\beta) \ne 0$,
	we may assume that $f'(y)$ is a unit in $L^{0} \tensor_{\alg{K}^{0}(k')} \alg{K}(k')$
	by a similar localization
	(apply the argument right before this lemma
	to the norm $N_{L / \alg{K}(k'')} f'(y) \in \alg{K}(k')$).
	Consider the polynomial $(\gamma f'_{\ideal{m}}(\beta))^{-1} f_{\ideal{m}}(\gamma x + \beta)$
	with coefficients in the complete discrete valuation field
	$\alg{K}(k' / \ideal{m})[\beta] \cong L \tensor_{\alg{K}(k'')} \alg{K}(k' / \ideal{m})$.
	By looking at the Newton polygon and by the choice of $\beta$,
	we know that this polynomial has integral coefficients,
	the constant term has positive valuation,
	and the coefficient for $x$ is 1.
	Hence we can localize $k'$ by an element not in $\ideal{m}$ so that
	$(\gamma f'(y))^{-1} f(\gamma x + y)$ is in $\Order_{L} \tensor_{\alg{O}_{K}(k'')} \alg{O}_{K}(k')[x]$,
	the constant term in $x$ is in $\ideal{p}_{L} \tensor_{\alg{O}_{K}(k'')} \alg{O}_{K}(k')$,
	and the coefficient for $x$ is $1$,
	where $\Order_{L}$ is the ring of integers of $L$
	and $\ideal{p}_{L}$ its maximal ideal.
	The pair
		$
			(\Order_{L} \tensor_{\alg{O}_{K}(k'')} \alg{O}_{K}(k'),
			\ideal{p}_{L} \tensor_{\alg{O}_{K}(k'')} \alg{O}_{K}(k'))
		$
	is Henselian as
	the pairs $(W(k'), (p))$ and $(k'[[T]], (T))$ are so.
	Therefore the polynomial $(\gamma f'(y))^{-1} f(\gamma x + y)$ of $x$ has a unique root
	in $\Order_{L} \tensor_{\alg{O}_{K}(k'')} \alg{O}_{K}(k')$
	whose image in $k_{L} \tensor_{k''} k'$ is zero,
	where $k_{L}$ is the residue field of $\Order_{L}$.
	Write this root as $x = (z - y) / \gamma$,
	so $z \in L \tensor_{\alg{K}(k'')} \alg{K}(k')$ is a root of the polynomial $f(x)$.
	Sending $x$ to $z$, we have a $\alg{K}(k')$-algebra homomorphism
		$
				\alg{K}(k')[x] / (f(x))
			\to
				L \tensor_{\alg{K}(k'')} \alg{K}(k')
			=
				L^{0} \tensor_{\alg{K}^{0}(k')} \alg{K}(k')
		$.
	The image of $z$ in the field
		$
				L \tensor_{\alg{K}(k'')} \alg{K}(k' / \ideal{m})
			=
				\alg{K}(k' / \ideal{m})[\beta]
		$
	is $\alpha$ by uniqueness.
	Hence the image of $g(z)$ in the same field is $g_{\ideal{m}}(\alpha) \ne 0$.
	Therefore we can localize $k'$ by an element not in $\ideal{m}$ so that
	$g(z)$ is a unit in $L \tensor_{\alg{K}(k'')} \alg{K}(k')$.
	Thus we get a $\alg{K}(k')$-algebra homomorphism
	$S \to L^{0} \tensor_{\alg{K}^{0}(k')} \alg{K}(k')$,
	as required.
\end{proof}

\begin{Lem} \label{lem: fully faithful}
	The tensor product $(\;\cdot\;) \tensor_{\alg{K}^{0}(k')} \alg{K}(k')$
	induces a fully faithful embedding from
	the category of \'etale $\alg{K}^{0}(k')$-algebras
	into the category of \'etale $\alg{K}(k')$-algebras.
\end{Lem}

\begin{proof}
	Let $f \colon \Spec \alg{K}(k') \to \Spec \alg{K}^{0}(k')$ be the natural morphism.
	Let $L^{0}_{1}$ and $L^{0}_{2}$ be \'etale $\alg{K}^{0}(k')$-algebras.
	It is enough to show that the natural sheaf morphism
		\[
				\sheafhom_{\alg{K}^{0}(k')}(L^{0}_{1}, L^{0}_{2})
			\to
				f_{\ast}
				\sheafhom_{\alg{K}(k')}(
					L^{0}_{1} \tensor_{\alg{K}^{0}(k')} \alg{K}(k'),
					L^{0}_{2} \tensor_{\alg{K}^{0}(k')} \alg{K}(k')
				)
		\]
	in $\Set(\alg{K}^{0}(k')_{\et})$ is an isomorphism.
	We may check this \'etale locally on $\alg{K}^{0}(k')$.
	Note that $\alg{K}^{0}(k')$ is a filtered union of finite products of fields.
	Hence by taking an \'etale cover of $\alg{K}^{0}(k')$,
	we may assume that $L^{0}_{1}$ and $L^{0}_{2}$ are
	direct factors of $\alg{K}^{0}(k')$.
	Since $k' \in k^{\ind\rat}$, we know that
	$L^{0}_{i} = \alg{K}^{0}(k'_{i})$ for some direct factor $k'_{i}$ of $k'$
	for each $i = 1, 2$.
	The morphism above becomes
		\[
				\sheafhom_{\alg{K}^{0}(k')}(\alg{K}^{0}(k'_{1}), \alg{K}^{0}(k'_{2}))
			\to
				f_{\ast}
				\sheafhom_{\alg{K}(k')}(
					\alg{K}(k'_{1}),
					\alg{K}(k'_{2})
				).
		\]
	The both sides are a point if $\Spec k'_{2} \subset \Spec k'_{1}$
	and empty otherwise.
\end{proof}

\begin{Lem} \label{lem: push is exact}
	Let $f \colon \Spec \alg{K}(k') \to \Spec \alg{K}^{0}(k')$ be the natural morphism.
	Then the pushforward
	$f_{\ast} \colon \Ab(\alg{K}(k')_{\et}) \to \Ab(\alg{K}^{0}(k')_{\et})$
	is exact.
\end{Lem}

\begin{proof}
	This is a formal consequence of Lemmas \ref{lem: refinement} and \ref{lem: fully faithful}.
\end{proof}

Therefore we have
$H^{i}(\alg{K}(k')_{\et}, \Gm) = H^{i}(\alg{K}^{0}(k')_{\et}, f_{\ast} \Gm)$ for all $i$.
To compute this, we recall Tate cohomology sheaves from \cite[\S 3.4]{SY12}
and give a variant of the vanishing result \cite[Proposition 4.2]{SY12}.
Let $L / K$ be a finite Galois extension with Galois group $G$.
Let $\Order_{L}$ be the ring of integers of $L$.
Define sheaves $\alg{L}_{k}$, $\alg{U}_{L, k}$ on $\Spec k^{\ind\rat}_{\et}$ by
	\[
			k'' \in k^{\ind\rat} \mapsto L \tensor_{K} \alg{K}(k''),
		\quad
			k'' \in k^{\ind\rat} \mapsto (\Order_{L} \tensor_{\Order_{K}} \alg{O}_{K}(k''))^{\times},
	\]
respectively.
The finite group $G$ acts on them as morphisms of sheaves.
Define a complex
	\[
			\cdots
		\to
			\prod_{\sigma_{1}, \sigma_{2} \in G}
				\alg{L}_{k}^{\times}
		\to
			\prod_{\sigma_{1} \in G}
				\alg{L}_{k}^{\times}
		\to
			\alg{L}_{k}^{\times}
		\stackrel{N}{\to}
			\alg{L}_{k}^{\times}
		\to
			\prod_{\sigma_{1} \in G}
				\alg{L}_{k}^{\times}
		\to
			\prod_{\sigma_{1}, \sigma_{2} \in G}
				\alg{L}_{k}^{\times}
		\to
			\cdots
	\]
in $\Ab(k^{\ind\rat}_{\et})$ using the differentials for inhomogeneous chains and cochains
(\cite[VII, \S 3 and \S 4]{Ser79})
and the norm map $N$.
For each $i \in \Z$, define the Tate cohomology sheaf $\hat{H}^{i}(G, \alg{L}_{k}^{\times})$
to be the $i$-th cohomology of this complex.
This is the \'etale sheafification of the presheaf
$k'' \mapsto \hat{H}^{i}(G, \alg{L}_{k}^{\times}(k''))$ of Tate cohomology groups.
We will show below that
the sheaves $\hat{H}^{i}(G, \alg{L}_{k}^{\times})$ on $\Spec k^{\ind\rat}_{\et}$ are all zero.
The difference from \cite[Proposition 4.2]{SY12} is the choice of the site.
To kill a cocycle in $\hat{H}^{i}(G, \alg{L}_{k}^{\times}(k''))$
with $k'' \in k^{\ind\rat}$,
we only need to extend $k''$ to its \emph{finite} extension.

\begin{Lem} \label{lem: vanishing of Tate cohomology sheaves}
	We have $\hat{H}^{i}(G, \alg{L}_{k}^{\times}) = 0$ for all $i \in \Z$.
\end{Lem}

\begin{proof}
	For any $k'' \in k^{\ind\rat}$,
	any Zariski locally free sheaf on $\Spec \alg{K}(k'')$ of constant rank is free
	by Lemma \ref{lem: Zariski comparison}.
	Hence $\Pic(\alg{K}(k'')) = 0$ and
		\[
				\hat{H}^{1}(G, \alg{L}_{k}^{\times}(k''))
			=
				\Ker \bigl(
					\Pic(\alg{K}(k'')) \to \Pic(\alg{L}_{k}(k''))
				\bigr)
			=
				0,
		\]
	so $\hat{H}^{1}(G, \alg{L}_{k}^{\times}) = 0$.
	Therefore it is enough to show that $\hat{H}^{0}(G, \alg{L}_{k}^{\times}) = 0$
	by \cite[IX, \S 5, Theorem 8]{Ser79}
	(to apply this theorem on group cohomology groups to our situation of cohomology sheaves,
	take an arbitrary injective object $I$ in $\Ab(k^{\ind\rat}_{\et})$
	and consider a corresponding statement for
	$\hat{H}^{i}(G, \Hom_{k^{\ind\rat}_{\et}}(\alg{L}_{k}^{\times}, I))$).
	It is enough to show that the norm map
	$N \colon \alg{U}_{L, k} \to \alg{U}_{K}$ is surjective in $\Ab(k^{\ind\rat}_{\et})$.
	
	Let $k' \in k^{\ind\rat}$ and $x \in \alg{U}_{K}(k')$.
	The map $N$ is a surjective homomorphism of proalgebraic groups over $k$
	(\cite[\S 2.1; \S 2.2, Corollary to Proposition 1]{Ser61}).
	The norm computation given in \cite[V, \S 3]{Ser79} shows that
	almost all subquotients of the filtration $\{ \Ker(N) \cap \alg{U}_{L, k}^{n} \}$
	of $\Ker(N)$ are connected unipotent,
	where $\alg{U}_{L, k}^{n}$ is the group of $n$-th principal units.
	Assume that the subquotients are connected unipotent for $n \ge n_{0}$.
	Consider the cartesian diagram of surjections of proalgebraic groups
		\[
			\begin{CD}
					\alg{U}_{L, k} / \bigl( \Ker(N) \cap \alg{U}_{L, k}^{n_{0}} \bigr)
				@>> N >
					\alg{U}_{K}
				\\
				@VVV
				@VVV
				\\
					\alg{U}_{L, k} / \alg{U}_{L, k}^{n_{0}}
				@> N >>
					\alg{U}_{K} / N \alg{U}_{L, k}^{n_{0}}.
			\end{CD}
		\]
	The lower left term and hence the lower right term are quasi-algebraic.
	Hence the image of the $k'$-point $x$ in the lower right term
	lifts to a $k''$-point of the lower left term
	for some faithfully flat \'etale $k'$-algebra $k''$.
	Since the diagram is cartesian,
	the $k'$-point $x$ of the upper right term itself lifts
	to a $k''$-point $x_{0}$ of the upper left term.
	Consider the cartesian diagram of surjections of proalgebraic groups
		\[
			\begin{CD}
					\alg{U}_{L, k} / \bigl( \Ker(N) \cap \alg{U}_{L, k}^{n_{0} + 1} \bigr)
				@>>>
					\alg{U}_{L, k} / \bigl( \Ker(N) \cap \alg{U}_{L, k}^{n_{0}} \bigr)
				\\
				@VVV
				@VVV
				\\
					\alg{U}_{L, k} / \alg{U}_{L, k}^{n_{0} + 1}
				@>>>
					\alg{U}_{L, k} / \bigl( \alg{U}_{L, k}^{n_{0} + 1} + \Ker(N) \cap \alg{U}_{L, k}^{n_{0}} \bigr).
			\end{CD}
		\]
	The kernel $G$ of the lower horizontal morphism is connected unipotent quasi-algebraic by assumption.
	The image of the $k''$-point $x_{0}$ in the lower right term defines an element of $H^{1}(k''_{\et}, G)$.
	The group $H^{1}(k''_{\et}, \Ga)$ is trivial
	since it is isomorphic to coherent cohomology by \cite[III, Remark 3.8]{Mil80}
	and the affine scheme $\Spec k''$ has trivial coherent cohomology.
	This implies that $H^{1}(k''_{\et}, G) = 0$.
	Hence the image of the $k''$-point $x_{0}$ in the lower right term
	lifts a $k''$-point of the lower left term (with no need to extend $k''$).
	Since the diagram is cartesian,
	the $k''$-point $x_{0}$ of the upper right term itself lifts
	to a $k''$-point $x_{1}$ of the upper left term.
	Iteratively applying this argument, we obtain a $k''$-point of
		\[
				\invlim_{n \ge n_{0}}
					\alg{U}_{L, k} / \bigl( \Ker(N) \cap \alg{U}_{L, k}^{n} \bigr)
			=
				\alg{U}_{L, k}
		\]
	whose norm is $x \in \alg{U}_{K}(k')$.
	This shows the surjectivity of
	$N \colon \alg{U}_{L, k} \to \alg{U}_{K}$ in $\Ab(k^{\ind\rat}_{\et})$.
\end{proof}

\begin{proof}[Proof of Proposition \ref{prop: cohomology of Gm}]
	By Lemma \ref{lem: push is exact}, we have
		\[
				H^{i}(\alg{K}(k')_{\et}, \Gm)
			=
				H^{i}(\alg{K}^{0}(k')_{\et}, f_{\ast} \Gm),
		\]
	where $f \colon \Spec \alg{K}(k') \to \Spec \alg{K}^{0}(k')$ is the natural morphism.
	Write $k' = \bigcup k'_{\lambda}$ with $k'_{\lambda} \in k^{\rat}$.
	We have
		\[
				H^{i}(\alg{K}^{0}(k')_{\et}, f_{\ast} \Gm)
			=
				\dirlim_{\lambda, L}
				H^{i}(
					\Gal(L / \alg{K}(k'_{\lambda})),
					(L \tensor_{\alg{K}(k'_{\lambda})} \alg{K}(k'))^{\times}
				),
		\]
	where the $L$ runs through the finite Galois extensions of $\alg{K}(k'_{\lambda})$.
	The coefficient group on the right-hand side can be written as
	$\alg{L}^{\times}_{k'_{\lambda}}(k')$ in the notation above.
	Therefore for $i \ge 1$,
	its element is killed after an \'etale faithfully flat extension of $k'$
	by Lemma \ref{lem: vanishing of Tate cohomology sheaves}.
	Hence $\alg{H}^{i}(\Gm) = 0$ for $i \ge 1$.
	This proves the proposition.
\end{proof}

The following will be needed in the next subsection
to reduce the computation of $R \alg{\Gamma}(\Z)$
to that of cohomology of complete discrete valuation fields.

\begin{Prop} \label{prop: cohomology of constant sheaves when base is indrational}
	Let $k' \in k^{\ind\rat}$.
	Let $A$ be a constant sheaf of abelian groups.
	Then $R \Gamma(\alg{K}(k')_{\et}, A) = R \Gamma(\alg{K}^{0}(k')_{\et}, A)$.
\end{Prop}

\begin{proof}
	By Lemma \ref{lem: push is exact},
	this amounts to saying that $f_{\ast} A = A$ in $\Ab(\alg{K}^{0}(k')_{\et})$,
	where $f \colon \Spec \alg{K}(k') \to \Spec \alg{K}^{0}(k')$ is the natural morphism.
	Let $L^{0}$ be any \'etale $\alg{K}^{0}(k')$-algebra
	and set $L = L^{0} \tensor_{\alg{K}^{0}(k')} \alg{K}(k')$.
	We need to show that $A(L^{0}) = A(L)$.
	It is enough to show that the inclusion $L^{0} \into L$ induces a bijection
	on the set of idempotents.
	The set of idempotents of $L^{0}$, $L$ can be identified with
		\[
				\Hom_{\alg{K}^{0}(k')}(\alg{K}^{0}(k') \times \alg{K}^{0}(k'), L^{0}),
			\quad
				\Hom_{\alg{K}(k')}(\alg{K}(k') \times \alg{K}(k'), L),
		\]
	respectively.
	They are in bijection by Lemma \ref{lem: fully faithful}.
\end{proof}


\subsection{Duality with coefficients in $\Gm$}
\label{sec: Duality with coefficients in Gm}

The following states the duality for $K$ with coefficients in $\Gm$.

\begin{Thm} \label{thm: duality, Z-coefficients}
	The morphism
		\[
				R \alg{\Gamma}(\Z)
			\to
				R \sheafhom_{k}(\alg{K}^{\times}, \Z)
		\]
	in $D(k)$ defined at the end of
	Section \ref{sec: The structure morphism of a local field and the cup product pairing}
	is an isomorphism.
\end{Thm}

We prove this in this subsection.
This implies,
by taking $R \Gamma(k'_{\et}, \,\cdot\,)$ of the both sides for any $k' \in k^{\ind\rat}$,
an isomorphism
	\begin{equation} \label{eq: duality for Gm with ind-rational base}
			R \Gamma(\alg{K}(k')_{\et}, \Z)
		\to
			R \Hom_{k^{\ind\rat}_{\et} / k'}(\alg{K}^{\times}, \Z)
	\end{equation}
in $D(\Ab)$,
where we again used the comparison between fppf and \'etale cohomology
with coefficients in a smooth group scheme
(\cite[III, Remark 3.11 (b)]{Mil80}).
Conversely, this isomorphism for any $k' \in k^{\ind\rat}$ implies the theorem.
Let $k' = \bigcup_{\lambda} k'_{\lambda}$ with $k'_{\lambda} \in k^{\rat}$.
The left-hand side of \eqref{eq: duality for Gm with ind-rational base} is the direct limit of
$R \Gamma(\alg{K}(k'_{\lambda})_{\et}, \Z)$ in $\lambda$
by Proposition \ref{prop: cohomology of constant sheaves when base is indrational}
and \cite[III, Lemma 1.16, Remark 1.17 (a)]{Mil80}.
For the right-hand side, note that $\alg{K}^{\times} \cong \Z \times \alg{U}_{K}$,
where $\alg{U}_{K} = \alg{O}_{K}^{\times}$.
The sheaf $\alg{U}_{K}$ is represented
by the proalgebraic group of units of $K$ studied by Serre \cite{Ser61},
which is affine.
Hence Theorem \ref{thm: comparison of Ext, proalgebraic setting} shows that
$R \Hom_{k^{\ind\rat}_{\et} / k'}(\alg{U}_{K}, \Z)$ is the direct limit of 
$R \Hom_{(k'_{\lambda})^{\ind\rat}_{\et}}(\alg{U}_{K}, \Z)$ in $\lambda$.
On the other hand, Proposition \ref{prop: rational etale cohomology is Galois cohomology} shows that
$R \Hom_{k^{\ind\rat}_{\et} / k'}(\Z, \Z) = R \Gamma(k'_{\et}, \Z)$,
which is the direct limit of
$R \Hom_{(k'_{\lambda})^{\ind\rat}_{\et}}(\Z, \Z) = R \Gamma((k'_{\lambda})_{\et}, \Z)$ in $\lambda$
by \cite[III, Lemma 1.16, Remark 1.17 (a)]{Mil80}.
Hence the homomorphism \eqref{eq: duality for Gm with ind-rational base} is
the direct limit of
	\[
			R \Gamma(\alg{K}(k'_{\lambda})_{\et}, \Z)
		\to
			R \Hom_{(k'_{\lambda})^{\ind\rat}_{\et}}(\alg{K}^{\times}, \Z).
	\]
We want to show that this is an isomorphism for any $\lambda$.
Replacing $K$ with $\alg{K}(k'_{\lambda})$,
we only need to consider the case $k'_{\lambda} = k$:
	\[
			R \Gamma(K_{\et}, \Z)
		\to
			R \Hom_{k^{\ind\rat}_{\et}}(\alg{K}^{\times}, \Z).
	\]
Compare this morphism with the morphism for the case $k' = \closure{k}$:
	\[
			R \Gamma(\alg{K}(\closure{k})_{\et}, \Z)
		\to
			R \Hom_{k^{\ind\rat}_{\et} / \closure{k}}(\alg{K}^{\times}, \Z)
		=
			R \Hom_{\closure{k}^{\ind\rat}_{\et}}(\alg{K}^{\times}, \Z),
	\]
where we used the fact $\Spec \closure{k}^{\ind\rat}_{\et} = \Spec k^{\ind\rat}_{\et} / \closure{k}$
observed in Section \ref{sec: The ind-rational etale site}.
By applying $R \Gamma(k_{\et}, \,\cdot\,)$ to the latter,
we may assume that $k = \closure{k}$.
What we have to show is hence
		\[
				H^{i}(K, \Z)
			\isomto
				\Ext_{k}^{i}(\alg{K}^{\times}, \Z)
		\]
for algebraically closed $k$.
We first treat the part $i \ne 2$.

\begin{Prop} \label{prop: vanishing for duality with coefficients in Gm}
	We have
		\begin{gather*}
					\Gamma(K, \Z)
				=
					\Hom_{k}(\alg{K}^{\times}, \Z)
				=
					\Z,
			\\
					H^{i}(K, \Z)
				=
					\Ext_{k}^{i}(\alg{K}^{\times}, \Z)
				=
					0
				\quad \text{for} \quad
					i \ne 0, 2.
		\end{gather*}
\end{Prop}

\begin{proof}
	We compute $\Ext_{k}^{i}(\alg{U}_{K}, \Z)$ for $i \ne 2$.
	The case $i = 0$ is done in Theorem \ref{thm: comparison of Ext, proalgebraic setting}.
	If $i \ne 0$, then we have
		\[
				\Ext_{k}^{i}(\alg{U}_{K}, \Z)
			=
				\dirlim_{n} \Ext_{\Pro \Alg / k}^{i - 1}(\alg{U}_{K}, \Z / n \Z).
		\]
	by the same theorem.
	By \cite[\S 5.4, Corollary to Proposition 7]{Ser60}, this group
	is the Pontryagin dual of the $(i - 1)$-st homotopy group $\pi_{i - 1}(\alg{U}_{K})$ of
	the proalgebraic group $\alg{U}_{K}$ in the sense of Serre.
	We have $\pi_{i - 1}(\alg{U}_{K}) = 0$ if $i - 1= 0$ by the connectedness of $\alg{U}_{K}$,
	and if $i - 1\ge 2$ by \cite[\S 10, Theorem 2]{Ser60}.
	This finishes the proof of the statements for $\Ext_{k}^{i}(\alg{K}^{\times}, \Z)$.
	
	The groups $H^{i}(K, \Z)$ are Galois cohomology groups
	of the complete discrete valuation field $K$
	with algebraically closed residue field.
	Such a field has cohomological dimension $1$ by \cite[II, \S 3.3]{Ser02}.
	The result then follows.
\end{proof}

Therefore it is enough to show that the homomorphism
	\[
			H^{2}(K, \Z)
		\to
			\Ext_{k}^{2}(\alg{K}^{\times}, \Z)
	\]
is an isomorphism.
Note that the left-hand side is equal to
$H^{1}(K, \Q / \Z) = \Ext_{K}^{2}(\Z, \Z)$
and the right-hand side
	$
			\Ext_{k}^{1}(\alg{K}^{\times}, \Q / \Z)
		=
			\dirlim_{n}
				\Ext_{k}^{1}(\alg{K}^{\times}, \Z / n \Z)
	$
as shown in the proof of the above proposition.

\begin{Prop}
	The above homomorphism
		\[
				H^{1}(K, \Q / \Z)
			\to
				\Ext_{k}^{1}(\alg{K}^{\times}, \Q / \Z)
		\]
	sends a finite cyclic extension $L$ of $K$ of degree $n$
	with Galois group $G = \langle \sigma \rangle$
	to the extension class
		\[
				0
			\to
				G
			\stackrel{\sigma \mapsto \sigma(\pi_{L}) / \pi_{L}}{\longrightarrow}
				\alg{L}^{\times} / I_{G} \alg{U}_{L}
			\stackrel{N_{G}}{\to}
				\alg{K}^{\times}
			\to
				0,
		\]
	where $\alg{L}$ is defined from $L$ in the same way as we defined $\alg{K}$,
	$I_{G} = \Z[G] (\sigma - 1)$ the augmentation ideal of the integral group ring $\Z[G]$,
	$\pi_{L}$ a prime element of $\Order_{L}$,
	and $N_{G}$ the norm map.
\end{Prop}

\begin{proof}
	The extension $L / K$ as an element of $\Ext_{K}^{2}(\Z, \Z)$ is represented by
		\[
				0
			\to
				\Z
			\stackrel{N_{G}}{\to}
				\Z[G]
			\stackrel{\sigma - 1}{\to}
				\Z[G]
			\stackrel{\sigma \mapsto 1}{\to}
				\Z
			\to
				0,
		\]
	Its image in $\Ext_{K}^{2}(\Gm, \Gm)$ is
		\[
				0
			\to
				\Gm
			\stackrel{\incl}{\to}
				\Res_{L / K} \Gm
			\stackrel{\sigma - 1}{\to}
				\Res_{L / K} \Gm
			\stackrel{N_{G}}{\to}
				\Gm
			\to
				0,
		\]
	where $\Res_{L / K}$ is the Weil restriction functor.
	We have
		\[
				R \alg{\Gamma}(K_{\fppf}, \Res_{L / K} \Gm)
			=
				R \alg{\Gamma}(L_{\fppf}, \Gm)
			=
				\alg{L}^{\times}.
		\]
	Hence we can use the second half of Proposition \ref{prop: functoriality}.
	Applying $\alg{\Gamma}$, we have an exact sequence
		\[
				0
			\to
				\alg{K}^{\times}
			\stackrel{\incl}{\to}
				\alg{L}^{\times}
			\stackrel{\sigma - 1}{\to}
				\alg{L}^{\times}
			\stackrel{N_{G}}{\to}
				\alg{K}^{\times}
			\to
				0
		\]
	in $\Ab(k)$.
	Pushing out this extension by the valuation map $\alg{K}^{\times} \onto \Z$ from the left term,
	we find that the image in $\Ext_{k}^{2}(\alg{K}^{\times}, \Z)$ is given by
		\[
				0
			\to
				\Z
			\stackrel{1 \mapsto \pi_{K}}{\to}
				\alg{L}^{\times} / \alg{U}_{K}
			\stackrel{\sigma - 1}{\to}
				\alg{L}^{\times}
			\stackrel{N_{G}}{\to}
				\alg{K}^{\times}
			\to
				0
		\]
	where $\pi_{K}$ is a prime element of $\Order_{K}$.
	We have a natural quotient map from this extension to the extension
		\[
				0
			\to
				\Z
			\stackrel{n}{\to}
				\Z
			\stackrel{1 \mapsto \sigma(\pi_{L}) / \pi_{L}}{\longrightarrow}
				\alg{L}^{\times} / I_{G} \alg{U}_{L}
			\stackrel{N_{G}}{\to}
				\alg{K}^{\times}
			\to
				0.
		\]
	This extension as an element of the subgroup
	$\Ext_{k}^{1}(\alg{K}^{\times}, \Z / n)$ of $\Ext_{k}^{2}(\alg{K}^{\times}, \Z)$ is
		\[
				0
			\to
				G
			\stackrel{\sigma \mapsto \sigma(\pi_{L}) / \pi_{L}}{\longrightarrow}
				\alg{L}^{\times} / I_{G} \alg{U}_{L}
			\stackrel{N_{G}}{\to}
				\alg{K}^{\times}
			\to
				0,
		\]
	as required.
\end{proof}

\begin{proof}[Proof of Theorem \ref{thm: duality, Z-coefficients}]
	We have
		\[
				\Ext_{k^{\ind\rat}_{\et}}^{1}(\alg{K}^{\times}, \Q / \Z)
			=
				\dirlim_{n}
					\Ext_{\Pro \Alg / k}^{1}(\alg{U}_{K}, \Z / n \Z)
		\]
	as seen in the proof of Proposition \ref{prop: vanishing for duality with coefficients in Gm}.
	Under this isomorphism,
	the proposition above says that our homomorphism,
	or its Pontryagin dual
		\[
				\pi_{1}(\alg{U}_{K})
			\to
				\Gal(K^{\ab} / K),
		\]
	($\ab$ denotes the maximal abelian extension) agrees with
	the reciprocity isomorphism of Serre's local class field theory \cite{Ser61}.
	This proves the theorem.
\end{proof}

A corollary  is that
	\[
			H^{1}(K, \Q / \Z)
		=
			\Ext_{k}^{1}(\alg{K}^{\times}, \Q / \Z).
	\]
When $k$ is not necessarily algebraically closed,
this recovers the main theorem of \cite{SY12}.


\subsection{Duality with coefficients in a finite flat group scheme}
\label{sec: duality with coefficients in a finite flat group scheme}

Let $A$ be a finite flat group scheme over $K$.
Assume that $A$ does not have connected unipotent part
(which is always the case if $K$ has mixed characteristic).
As in the previous subsection,
we obtain a morphism
	\begin{align*}
		&
				R \alg{\Gamma}(A^{\CDual})
			\to
				R \alg{\Gamma} R \sheafhom_{K}(A, \Gm)
		\\
		&	\quad
			\to
				R \sheafhom_{k}(R \alg{\Gamma}(A), \Z)
			=
				R \sheafhom_{k}(R \alg{\Gamma}(A), \Q / \Z [-1]),
	\end{align*}
where the last equality is from the exact sequence
$0 \to \Z \to \Q \to \Q / \Z \to 0$
and that $A$ is finite.

\begin{Thm} \label{thm: duality, finite flat coefficients}
	The morphism
		\[
				R \alg{\Gamma}(A^{\CDual})
			\to
				R \sheafhom_{k}(R \alg{\Gamma}(A), \Q / \Z [-1])
		\]
	defined above is an isomorphism.
\end{Thm}

\begin{Lem}
	The theorem is true for $A = \mu_{l}$,
	where $l$ is a prime (possibly $l = p$).
\end{Lem}

\begin{proof}
	Consider the Kummer sequence
	$0 \to \mu_{l} \to \Gm \to \Gm \to 0$
	and its Cartier dual
	$0 \to \Z \to \Z \to \Z / l \Z \to 0$.
	They yield a commutative diagram
		\[
		\begin{CD}
				R \alg{\Gamma}(\Z)
			& @>>> &
				R \alg{\Gamma}(\Z)
			& @>>> &
				R \alg{\Gamma}(\Z / l \Z)
			\\
			@VVV & & @VVV & & @VVV
			\\
				R \alg{Hom}_{k}(R \alg{\Gamma}(\Gm), \Z)
			& @>>> &
				R \alg{Hom}_{k}(R \alg{\Gamma}(\Gm), \Z)
			& @>>> &
				R \alg{Hom}_{k}(R \alg{\Gamma}(\mu_{l}), \Z)
		\end{CD}
		\]
	whose two rows are distinguished triangles.
	The first and second vertical morphisms are isomorphisms
	by Theorem \ref{thm: duality, Z-coefficients}.
	Hence so is the third.
\end{proof}

The following finishes the proof of the theorem
in the mixed characteristic case.

\begin{Lem} \label{lem: duality for finite multiplicative A}
	The theorem is true for a multiplicative $A$.
\end{Lem}

\begin{proof}
	We may assume that $A$ has $l$-power order for a prime $l$ (possibly $l = p$).
	We reduce the lemma to the previously treated case $A = \mu_{l}$.
	Let $L$ be a finite Galois extension of $K$
	such that the \'etale group $A^{\CDual}$ becomes constant over $L$.
	Let $M$ be the intermediate field of $L / K$
	that corresponds to an $l$-Sylow subgroup of $\Gal(L / K)$.
	Then the $l$-power torsion finite abelian group $A^{\CDual}(L)$
	is equipped with the action of the $l$-group $\Gal(L / M)$.
	Therefore, over $M$, the group $A^{\CDual}$ (resp.\ $A$) has a filtration
	whose successive subquotients are all isomorphic to $\Z / l \Z$ (resp.\ $\mu_{l}$).
	Hence by the above lemma, we have an isomorphism
		\begin{equation} \label{eq: duality over finite extension}
				R \alg{\Gamma}(M, A^{\CDual})
			\to
				R \sheafhom_{k'}(R \alg{\Gamma}(M, A), \Q / \Z [-1])
		\end{equation}
	in $D(k')$, where $k'$ is the residue field of $M$.
	We have a commutative diagram
		\[
			\begin{CD}
					\Spec M_{\fppf} / k'^{\ind\rat}_{\et}
				@>>>
					\Spec k'^{\ind\rat}_{\et}
				\\
				@VVV
				@VVV
				\\
					\Spec K_{\fppf} / k^{\ind\rat}_{\et}
				@>>>
					\Spec k^{\ind\rat}_{\et}
			\end{CD}
		\]
	of morphisms of sites,
	where the left vertical morphism is defined by the functor
	$(S, k_{S}) \mapsto (S \tensor_{K} M, k_{S} \tensor_{k} k')$.
	Hence by applying the Weil restriction functor $\Res_{k' / k}$
	for the both sides of \eqref{eq: duality over finite extension}
	and using the duality for the finite \'etale morphism $\Spec k' \to \Spec k$
	\cite[V, Proposition 1.13]{Mil80},
	we have an isomorphism
		\[
				R \alg{\Gamma}(K, \Res_{M / K} A^{\CDual})
			\to
				R \sheafhom_{k}(R \alg{\Gamma}(K, \Res_{M / K} A), \Q / \Z [-1])
		\]
	in $D(k)$.
	The inclusion $A \into \Res_{M / K} A$ followed by
	the norm map $\Res_{M / K} A \onto A$ is given by
	multiplication by $[M : K]$.
	This is an isomorphism
	since $A$ is $l$-power torsion and $M / K$ is an extension of degree prime to $l$.
	Hence $A$ is a canonical direct summand of $\Res_{M / K} A$.
	A similar relation holds for $A^{\CDual}$.
	The above isomorphism induces an isomorphism on these direct summands.
	This proves the lemma.
\end{proof}

Note that the distinguished triangle
$R \alg{\Gamma}(\mu_{l}) \to R \alg{\Gamma}(\Gm) \to R \alg{\Gamma}(\Gm)$
used in the above proof
and Proposition \ref{prop: cohomology of Gm} show that
	\begin{gather*}
				\alg{H}^{i}(\mu_{l})
			=
				0
			\quad \text{for} \quad
				i \ne 0, 1,
		\\
				\alg{H}^{1}(\mu_{l})
			=
				\alg{K}^{\times} / (\alg{K}^{\times})^{l}.
	\end{gather*}

\begin{Lem} \label{lem: Z mod p Z equal char case of duality}
	The theorem is true for $A = \Z / p \Z$
	when $K$ has equal characteristic.
\end{Lem}

\begin{proof}
	By the previous lemma, the morphism in the theorem this case can be written as
		\[
				R \alg{\Gamma}(\mu_{p})
			\to
				R \sheafhom_{k} \bigl(
					R \sheafhom_{k}(R \alg{\Gamma}(\mu_{p}), \Q / \Z), \Q / \Z
				\bigr),
		\]
	which agrees with the canonical evaluation morphism.
	With the remark just above,
	we know that $R \alg{\Gamma}(\mu_{p})$ is acyclic outside degree $1$ and
		\[
				\alg{H}^{1}(\mu_{p})
			=
				\alg{K}^{\times} / (\alg{K}^{\times})^{p}
			=
				\Z / p \Z \times \alg{U}_{K}^{1} / (\alg{U}_{K}^{1})^{p}
			\cong
				\Z / p \Z \times \Ga^{\N},
		\]
	where the last isomorphism is given by the Artin-Hasse exponential map
	(\cite[V, \S 3.16, Proposition 9]{Ser88}).
	We need to show that the evaluation morphism
		\[
				\Ga^{\N}
			\to
				R \sheafhom_{k} \bigl(
					R \sheafhom_{k}(\Ga^{\N}, \Q / \Z), \Q / \Z
				\bigr)
		\]
	is an isomorphism.
	By the Breen-Serre duality on perfect unipotent groups
	(\cite[8.4, Remarque]{Ser60}, \cite[III, Theorem 0.14]{Mil06})
	and Theorem \ref{thm: comparison of Ext, proalgebraic setting},
	the evaluation morphism
		\[
				\Ga
			\to
				R \sheafhom_{k} \bigl(
					R \sheafhom_{k}(\Ga, \Q / \Z), \Q / \Z
				\bigr),
		\]
	is an isomorphism.%
	\footnote{The key here is that $\Ext_{k}^{i}(\Ga, \Q / \Z) = 0$ for $i \ne 1$
	and $\Ext_{k}^{1}(\Ga, \Q / \Z) = k$ via the Artin-Schreier isogeny,
	so $R \sheafhom_{k}(\Ga, \Q / \Z) = \Ga[-1]$.
	See \cite[\S 8]{Ser60}.}
	We have
		\[
				R \sheafhom_{k}(\Ga^{\N}, \Q / \Z)
			=
				R \sheafhom_{k}(\Ga, \Q / \Z)^{\bigoplus \N},
		\]
	since
		\begin{align*}
			&
					\Ext_{k}^{i}(\Ga^{\N}, \Q / \Z)
				=
					\dirlim_{m}
						\Ext_{\Pro \Alg / k}^{i}(\Ga^{\N}, \Z / m \Z)
			\\
			&	=
					\dirlim_{m}
						\Ext_{\Pro \Alg / k}^{i}(\Ga, \Z / m \Z)^{\bigoplus \N}
				=
					\Ext_{k}^{i}(\Ga, \Q / \Z)^{\bigoplus \N}
		\end{align*}
	for any $i$ by Theorem \ref{thm: comparison of Ext, proalgebraic setting},
	\cite[\S 3.4, Proposition 7]{Ser60}
	and the argument in the paragraph right after Theorem \ref{thm: duality, Z-coefficients}.
	Hence we have
		\begin{align*}
			&
					R \sheafhom_{k} \bigl(
						R \sheafhom_{k}(\Ga^{\N}, \Q / \Z), \Q / \Z
					\bigr)
			\\
			&	=
					R \sheafhom_{k} \bigl(
						R \sheafhom_{k}(\Ga, \Q / \Z)^{\bigoplus \N}, \Q / \Z
					\bigr)
			\\
			&	=
					R \prod_{n \in \N}
					R \sheafhom_{k} \bigl(
						R \sheafhom_{k}(\Ga, \Q / \Z), \Q / \Z
					\bigr)
			\\
			&	=
					R \prod_{n \in \N} \Ga,
		\end{align*}
	where $R \prod_{n \in \N}$ denotes the derived functor of the direct product functor
	(\cite{Roo06}) for $\Ab(k^{\ind\rat}_{\et})$.
	Hence it is enough to show that $R^{i} \prod_{n \in \N} \Ga = 0$ for $i \ge 1$.
	This follows from \cite[Proposition 1.6]{Roo06}
	and that $H^{i}(k'_{\et}, \Ga) = 0$ for any $k' \in k^{\ind\rat}$.
\end{proof}

\begin{Lem}
	The theorem is true for an \'etale unipotent $A$.
\end{Lem}

\begin{proof}
	We can reduce this to the previously treated case $A = \Z / p \Z$
	by a similar method used in the proof of
	Lemma \ref{lem: duality for finite multiplicative A}.
\end{proof}

\begin{proof}[Proof of Theorem \ref{thm: duality, finite flat coefficients}]
	A general $A$ with no infinitesimal unipotent part
	is an extension of an \'etale unipotent one by a multiplicative one.
	This finishes the proof of the theorem.
\end{proof}

Theorem \ref{thm: duality, finite flat coefficients} implies
		\[
				R \Gamma(A^{\CDual})
			=
				R \Hom_{k}(R \alg{\Gamma}(A), \Q / \Z [-1])
		\]
by taking $R \Gamma(k, \;\cdot\;)$ of the both sides.

From the proof, it follows that
$\alg{H}^{i}(A) = 0$ for $i \ge 2$.
The sheaf $\alg{\Gamma}(A) \in \Ab(k^{\ind\rat}_{\et})$ is a finite \'etale group.
The sheaf $\alg{H}^{1}(A)$ is an extension of a finite \'etale group
by a connected ind-pro-unipotent group
(that is, a filtered direct limit of connected unipotent proalgebraic groups).

If $K$ has mixed characteristic, this ind-pro-unipotent part
is a (finite-dimensional) unipotent group.
Hence the theorem remains valid
if we use the site $\Spec k^{\rat}_{\et}$ instead of $\Spec k^{\ind\rat}_{\et}$
by Theorem \ref{thm: comparison of Ext, proalgebraic setting}.
Also we can use the \'etale topology for $K$ instead of the fppf topology
since $A$ over $K$ is \'etale and
the \'etale cohomology with coefficients in $A$ agrees with the fppf cohomology.
This means that we can use $\Spec K_{\et} / k^{\rat}_{\et}$ instead of $\Spec K_{\fppf} / k^{\ind\rat}_{\et}$.
The argument of the paragraph before Lemma \ref{lem: Z mod p Z equal char case of duality}
shows that $\alg{H}^{1}(\Z / n \Z(1)) = \alg{K}^{\times} \tensor \Z / n \Z$ for any $n \ge 1$.
Hence we have $\alg{H}^{1}(\Q / \Z(1)) = \alg{K}^{\times} \tensor \Q / \Z$.
In this manner,
Theorem \ref{thm: duality, finite flat coefficients} for mixed characteristic $K$ becomes
Theorem \ref{thm: main theorem, duality} in Introduction.

\begin{Rmk} \BetweenThmAndList \label{rmk: about duality for finite coefficients}
	\begin{enumerate}
	\item
		There are some generalizations of Theorem \ref{thm: duality, finite flat coefficients},
		essentially covered in B\'egueri's and Bester's papers.
		Here we merely indicate the formulation in our setting without proof.
		\begin{enumerate}
			\item
				Theorem \ref{thm: duality, finite flat coefficients} holds
				even if $A$ is a general finite flat group scheme over $K$.
				It is reduced to the case $A = \alpha_{p}$.
				This case can be treated by the same argument as \cite{AM76} or \cite{Bes78}.
				Namely, by pushing forward $A$ from
				the relative fppf site $\Spec K_{\fppf} / k^{\ind\rat}_{\et}$
				to the relative \'etale site $\Spec K_{\et} / k^{\ind\rat}_{\et}$,
				and using the $\dlog$ map and the Cartier operator,
				our duality is reduced to the duality for coherent coefficients over $K$.
			\item
				We can also formulate and prove a similar duality for the ring $\Order_{K}$ of integers.
				Namely, we can define a similar site $(\Spec \Order_{K})_{\fppf} / k^{\ind\rat}_{\et}$,
				a morphism
					$
							\pi
						\colon
							(\Spec \Order_{K})_{\fppf} / k^{\ind\rat}_{\et}
						\to
							\Spec k^{\ind\rat}_{\et}
					$
				and fppf cohomology $R \alg{\Gamma}_{x}$ with support on the closed point $x = \Spec k$
				(\cite[\S 2.2]{Bes78}, \cite[III, \S 0]{Mil06}).
				Again by classical results on \'etale cohomology of $K$ and $\Order_{K}$,
				we have $R \alg{\Gamma}_{x}((\Order_{K})_{\fppf}, \Gm) = \Z[-1]$,
				and the morphisms
					\begin{gather*}
								R \alg{\Gamma}_{x}((\Order_{K})_{\fppf}, \Z)
							\to
								R \sheafhom_{k}(R \alg{\Gamma}((\Order_{K})_{\fppf}, \Gm), \Z[-1]),
						\\
								R \alg{\Gamma}_{x}((\Order_{K})_{\fppf}, A^{\CDual})
							\to
								R \sheafhom_{k}(R \alg{\Gamma}((\Order_{K})_{\fppf}, A), \Q / \Z[-2])
					\end{gather*}
				with finite flat $A$ are isomorphisms in $D(k)$.
		\end{enumerate}
	\item
		If we avoid Theorem \ref{thm: comparison of Ext, proalgebraic setting}
		(hence avoid the entire part of
		Section \ref{sec: extensions of algebraic groups as sheaves on the rational etale site})
		and use Proposition \ref{prop: special case of comparison of Ext} instead,
		then what we can prove for mixed characteristic $K$
		is an apparently weaker statement
			\[
					R \alg{\Gamma}(A^{\CDual})
				=
					\tau_{\le 1}
					R \sheafhom_{k^{\rat}_{\et}}(R \alg{\Gamma}(A), \Q / \Z[-1]),
			\]
		where $\tau$ denotes the truncation functor.
	\item \label{point: shift in the cohomology of local fields}
		If one feels strange about the shift,
		one can define the compact support cohomology $R \alg{\Gamma}_{c}(K, A) = R \pi_{!} A$ of $K$
		to be $R \alg{\Gamma}(K, A)[-1]$.
		Then Theorem \ref{thm: duality, finite flat coefficients} may be written as
			\[
					R \alg{\Gamma}(K, A^{\CDual}[2])
				=
					R \sheafhom_{k^{\ind\rat}_{\et}}(R \alg{\Gamma}_{c}(K, A), \Q / \Z).
			\]
		Note that if $K$ has mixed characteristic,
		then $A^{\CDual}[2] = A^{\PDual}(1)[2]$,
		where $\PDual = \sheafhom_{K}(\;\cdot\;, \Q / \Z)$ denotes the Pontryagin dual.
		The complex $R \alg{\Gamma}(K, A)$ is concentrated in degrees $0, 1$
		and $R \alg{\Gamma}_{c}(K, A)$ in degrees $1, 2$.
		Hence our duality takes the same form as
		the Poincar\'e duality for $l$-adic cohomology of an affine curve over a closed field,
		just as if $\Spec K$ were a punctured disc over $\Spec k$ via the structure morphism $\pi$.
		
		Similarly, the duality for the cohomology of $\Order_{K}$ above takes the same form
		as the duality for an open disc by setting
		$R \alg{\Gamma}_{c}(\Order_{K}, \;\cdot\;) = R \alg{\Gamma}_{x}(\Order_{K}, \;\cdot\;)$.
		We can show that $R \alg{\Gamma}(\Order_{K}, j_{!} \;\cdot\;) = 0$,
		where $j \colon \Spec \Order_{K} \into \Spec K$ is the open immersion,
		and the localization exact sequence induces an isomorphism
		$R \alg{\Gamma}_{c}(K, \;\cdot\;) \isomto R \alg{\Gamma}_{c}(\Order_{K}, j_{!} \;\cdot\;)$.
		This justifies our definition of $R \alg{\Gamma}_{c}(K, \;\cdot\;)$.
	\end{enumerate}
\end{Rmk}


\subsection{Duality for a variety over a local field}
\label{sec: duality for a variety over a local field}

In this subsection, we assume that the complete discrete valuation field $K$ has mixed characteristic.
As in Introduction, we use the site $\Spec k^{\rat}_{\et}$ instead of $\Spec k^{\ind\rat}_{\et}$
and $\Spec K_{\et} / k^{\rat}_{\et}$ instead of $\Spec K_{\fppf} / k^{\ind\rat}_{\et}$.
We denote the \'etale structure morphism by
	\[
			\pi_{K / k}
		\colon
			\Spec K_{\et} / k^{\rat}_{\et}
		\to
			\Spec k^{\rat}_{\et}.
	\]
Theorem \ref{thm: duality, finite flat coefficients} remains valid with these sites and this morphism
as we saw in the last paragraph of
Section \ref{sec: duality with coefficients in a finite flat group scheme}.

Let $X$ be a $K$-scheme of finite type.
Consider the category $X_{\et} / k^{\rat}$
whose objects are pairs $(Y, k_{Y})$ with $k_{Y} \in k^{\rat}$
and $Y$ an \'etale scheme over $X \times_{K} \alg{K}(k_{Y})$.
A morphism $(Z, k_{Z}) \to (Y, k_{Y})$
consists of a $k$-algebra homomorphism $k_{Y} \to k_{Z}$ and
a scheme morphism $Z \to Y$ such that the diagram
	\[
		\begin{CD}
					Z
			@>>>
				Y
			\\
			@VVV
			@VVV
			\\
				X \times_{K} \alg{K}(k_{Z})
			@>>>
				X \times_{K} \alg{K}(k_{Y})
		\end{CD}
	\]
commutes.
The composite of two morphisms is defined in the obvious way.
A morphism $(Z, k_{Z}) \to (Y, k_{Y})$ is said to be \emph{\'etale}
if $k_{Y} \to k_{Z}$ is \'etale
(in which case $Z \to Y$ is \'etale).
A finite family $\{(Y_{i}, k_{i})\}$ of objects \'etale over
$(Y, k_{Y}) \in X_{\et} / k^{\rat}$
is said to be an \emph{\'etale covering} if $\{Y_{i}\}$ covers $Y$.
This defines a Grothendieck topology on $X_{\et} / k^{\rat}$.
We call the resulting site the \emph{relative \'etale site of $X$ over $k$}
and denote it by $X_{\et} / k^{\rat}_{\et}$.
For $(Y, k_{Y}) \in X_{\et}/ k^{\rat}$
and $A \in \Ab(X_{\et}/ k^{\rat}_{\et})$,
we have
	\[
			R \Gamma((Y, k_{Y}), A)
		=
			R \Gamma(Y_{\et}, f_{\ast} A),
	\]
where $(f_{\ast} A)(Y') = A(Y', k_{Y})$
for \'etale $Y$-schemes $Y'$.
For $(Y, k_{Y}) \in X_{\et} / k^{\rat}$, if $k_{Y}$ is a field,
then we denote by $\pi_{0}(Y) = \Spec L_{Y}$
the finite \'etale $\alg{K}(k_{Y})$-scheme of connected components of $Y$
(\cite[I, \S 4, Definition 6.6]{DG70}).
If $k_{Y}$ is not a field but a product of fields $\prod k_{Y, i}$,
then we set $\pi_{0}(Y) = \Spec L_{Y} = \bigsqcup \pi_{0}(Y_{i})$
where $Y_{i}$ is the fiber over $\alg{K}(k_{Y, i})$.
The functor
	\[
			K_{\et} / k^{\rat}
		\to
			X_{\et} / k^{\rat},
		\quad
			(L, k_{L})
		\mapsto
			(X \times_{K} L, k_{L})
	\]
admits a left adjoint given by
$(Y, k_{Y}) \mapsto (L_{Y}, k_{Y})$.
Hence it defines a morphism of sites
	\[
			\pi_{X / K}
		\colon
			X_{\et} / k^{\rat}_{\et}
		\to
			\Spec K_{\et} / k^{\rat}_{\et}.
	\]
The composite $\pi_{K / k} \compose \pi_{X / K}$ is denoted by $\pi_{X / k}$:
	\[
			\pi_{X / k}
		\colon
			X_{\et} / k^{\rat}_{\et}
		\stackrel{\pi_{X / K}}{\longrightarrow}
			\Spec K_{\et} / k^{\rat}_{\et}
		\stackrel{\pi_{K / k}}{\longrightarrow}
			\Spec k^{\rat}_{\et}.
	\]
We denote
	\[
			\alg{\Gamma}(X_{\et}, \;\cdot\;)
		=
			\pi_{X / k, \ast},
		\quad
			\alg{H}^{i}(X_{\et}, \;\cdot\;)
		=
			R^{i} \pi_{X / k, \ast},
		\quad
			R \alg{\Gamma}(X_{\et}, \;\cdot\;)
		=
			R \pi_{X / k, \ast}
	\]
(where we denote the pushforward $(\pi_{X / k})_{\ast}$ by $\pi_{X / k, \ast}$ for brevity).
For $d \in \Z$ and a torsion sheaf $A \in \Ab(X_{\et} / k^{\rat}_{\et})$
or $A \in \Ab(K_{\et} / k^{\rat}_{\et})$,
we denote the $d$-th Tate twist of $A$ by $A(d)$.
We have $R \pi_{X / K, \ast} (A(d)) = (R \pi_{X / K, \ast})(d)$
since $\Z / n \Z (d)$ for any $n \ge 1$ is \'etale locally constant on $K$.
Similarly Tate twists commute with
$R \sheafhom_{X}$ or $R \sheafhom_{K}$.

Now assume that $X$ is proper, smooth and geometrically connected over $K$.
Let $\closure{X} = X \times_{K} \closure{K}$.
We are going to regard the Poincar\'e duality for the variety $\closure{X}$ over $\closure{K}$
as duality for $\pi_{X / K}$
and combine it with the duality for $\pi_{K / k}$
established in Theorem \ref{thm: duality, finite flat coefficients}
to obtain a duality for $\pi_{X / k}$.
Let $d = \dim(X)$.
A part of the Poincar\'e duality for $\closure{X}$
gives the isomorphism
	\[
			\tau_{\ge 2 d}
			R \Gamma(\closure{X}_{\et}, \Q / \Z(d))
		=
			\Q / \Z [- 2 d],
	\]
where $\tau$ denotes the truncation functor.
The same is true when $X$ over $K$ is replaced by $X \times_{K} \alg{K}(k')$ over $\alg{K}(k')$
for any field $k' \in k^{\rat}$.
Hence
	\[
			\tau_{\ge 2 d}
			R \pi_{X / K, \ast} \Q / \Z(d)
		=
			\Q / \Z [- 2 d],
	\]
and we have a morphism
	\begin{align*}
				R \alg{\Gamma}(X_{\et}, \Q / \Z(d + 1))
		&	=
				R \alg{\Gamma}(K_{\et}, R \pi_{X / K, \ast} \Q / \Z (d + 1))
		\\
		&	\to
				R \alg{\Gamma}(K_{\et}, \Q / \Z (1))[- 2 d]
		\\
		&	\to
				\alg{H}^{1}(K_{\et}, \Q / \Z(1)) [- 2 d - 1]
		\\
		&	\to
				\Q / \Z [- 2 d - 1]
	\end{align*}
using the trace map (Theorem \ref{thm: main theorem, duality}).
Let $A$ be a constructible sheaf on $X_{\et}$.
We denote by $A^{\PDual} = R \sheafhom_{X}(A, \Q / \Z)$
the Pontryagin dual of $A$.
We have a morphism
	\begin{align*}
		&		R \alg{\Gamma} \bigl(
					X_{\et}, A^{\PDual}(d + 1)
				\bigr)
		\\
		&	\quad \to
				R \sheafhom_{k} \bigl(
					R \alg{\Gamma}(X_{\et}, A),
					R \alg{\Gamma}(X_{\et}, \Q / \Z(d + 1))
				\bigr)
		\\
		&	\quad \to
				R \sheafhom_{k} \bigl(
					R \alg{\Gamma}(X_{\et}, A),
					\Q / \Z [- 2 d - 1]
				\bigr).
	\end{align*}

\begin{Thm}
	The above defined morphism
		\[
				R \alg{\Gamma} \bigl(
					X_{\et}, A^{\PDual}(d + 1)
				\bigr)
			\to
				R \sheafhom_{k} \bigl(
					R \alg{\Gamma}(X_{\et}, A),
					\Q / \Z [- 2 d - 1]
				\bigr)
		\]
	is an isomorphism.
\end{Thm}

\begin{proof}
	By the Poincar\'e duality for the variety $\closure{X}$ over $\closure{K}$,
	we have
		\[
				R \Gamma \bigl(
					\closure{X}_{\et}, A^{\PDual}(d)
				\bigr)
			=
				R \Gamma(\closure{X}_{\et}, A)^{\PDual}[- 2 d].
		\]
	The same is true when $X$ over $K$ is replaced by $X \times_{K} \alg{K}(k')$ over $\alg{K}(k')$
	for any field $k' \in k^{\rat}$.
	Hence
		\[
					R \pi_{X / K, \ast}
					(A^{\PDual})(d)
			=
					(R \pi_{X / K, \ast} A)^{\PDual}[- 2 d].
		\]
	The complexes $R \pi_{X / K, \ast} A$, $R \pi_{X / K, \ast} (A^{\PDual})$
	are bounded complexes in $D(K)$
	whose cohomology at each degree is a finite \'etale group scheme over $K$
	by the constructibility of $A$ and the proper base change theorem.
	Therefore we can apply Theorem \ref{thm: duality, finite flat coefficients} to get
		\begin{align*}
			&
					R \alg{\Gamma} \bigl(
						X_{\et}, A^{\PDual}(d + 1)
					\bigr)
			\\
			&	=
					R \alg{\Gamma} \bigl(
						K_{\et}, (R \pi_{X / K, \ast} A)^{\PDual}(1)
					\bigr) [- 2 d]
			\\
			&	=
					R \sheafhom_{k} \bigl(
						R \alg{\Gamma}(K_{\et}, R \pi_{X / K, \ast} A),
						\Q / \Z
					\bigr) [- 2 d - 1]
			\\
			&	=
					R \sheafhom_{k} \bigl(
						R \alg{\Gamma}(X_{\et}, A),
						\Q / \Z
					\bigr) [- 2 d - 1],
		\end{align*}
	as desired.
\end{proof}

\begin{Rmk} \label{rmk: about shift in duality for varieties}
	As in Remark \ref{rmk: about duality for finite coefficients}
	\eqref{point: shift in the cohomology of local fields},
	one might want to define the compact support cohomology
	$R \alg{\Gamma}_{c}(X_{\et}, A) = R \pi_{X / k, !} A$ to be
	$R \alg{\Gamma}(X_{\et}, A)[-1] = R \alg{\Gamma}_{c}(K, R \pi_{X / K, \ast} A)$
	or $R \pi_{X / k, !} = R \pi_{K / k, !} \compose R \pi_{X / K, \ast}$.
	Then the above theorem may be written as
		\[
				R \alg{\Gamma} \bigl(
					X_{\et}, A^{\PDual}(d + 1)[2 d + 2]
				\bigr)
			=
				R \sheafhom_{k} \bigl(
					R \alg{\Gamma}_{c}(X_{\et}, A),
					\Q / \Z
				\bigr),
		\]
	as if $X$ were a $(d + 1)$-dimensional smooth variety over $k$.
\end{Rmk}


\subsection{Duality as adjunction}
\label{sec: the right adjoint to the pushforward}

In this subsection, we make no assumption on the characteristic of $K$
and return to the fppf structure morphism
$\pi \colon \Spec K_{\fppf} / k^{\ind\rat}_{\et} \to \Spec k^{\ind\rat}_{\et}$.
Recall that we have denoted $R \pi_{\ast} = R \alg{\Gamma}$.
As in Remarks \ref{rmk: about duality for finite coefficients}
\eqref{point: shift in the cohomology of local fields} and
\ref{rmk: about shift in duality for varieties},
we set
	\[
		R \pi_{!} = R \pi_{\ast}[-1].
	\]
We denote by $\dtensor_{k}$, $\dtensor_{K}$ the derived tensor products for
$D(k)$, $D(K)$, respectively.
Let $B \in \Ab(k)$
and $C \in \Ab(K)$.
By the paragraph after Proposition \ref{prop: functoriality} and adjunction,
we have a natural morphism and a cup product pairing
	\[
			B \dtensor_{k} R \pi_{\ast} C
		\to
			R \pi_{\ast} \pi^{\ast} B \dtensor_{k} R \pi_{\ast} C
		\to
			R \pi_{\ast}(\pi^{\ast} B \dtensor_{K} C).
	\]
With a shift, we have a morphism
	\begin{equation} \label{eq: projection formula}
			B \dtensor_{k} R \pi_{!} C
		\to
			R \pi_{!}(\pi^{\ast} B \dtensor_{K} C)
	\end{equation}
in $D(k)$.

\begin{Prop} \label{prop: projection formula}
	Assume that $B$ is a commutative finite \'etale group scheme over $k$.
	The morphism \eqref{eq: projection formula} is an isomorphism.
\end{Prop}

\begin{proof}
	We may check the statement \'etale locally on $k$.
	Therefore we may assume that $B$ is a constant group.
	Hence we may replace $B$ by $\Z$.
	This case is obvious.
\end{proof}

We continue assuming $B$ to be a finite \'etale group scheme over $k$.
We define
	\begin{equation} \label{eq: !-pullback}
			\pi^{!} B
		=
			\pi^{\ast} B \dtensor_{K} \Gm[1].
	\end{equation}
An explicit description is given as follows.
If $K$ has mixed characteristic, we have
$\pi^{\ast} B \dtensor_{K} \Gm \cong \pi^{\ast} B \dtensor_{K} \Q / \Z(1)$
since the quotient of $\Gm$ by its torsion subgroup $\Q / \Z(1)$ is uniquely divisible.
The exact sequence $0 \to \Z \to \Q \to \Q / \Z \to 0$ induces an isomorphism
$\pi^{\ast} B \dtensor_{K} \Q / \Z(1) \cong (\pi^{\ast} B)(1)[1]$.
Thus $\pi^{!} B \cong (\pi^{\ast} B) (1)[2]$
(again, compare this with Remarks \ref{rmk: about duality for finite coefficients}
\eqref{point: shift in the cohomology of local fields} and
\ref{rmk: about shift in duality for varieties}).
For general $K$, let $\Tor_{n}^{K}$ be the $n$-th left derived functor of $\tensor_{K}$.
Then we have $\pi^{!} B \cong \Tor_{1}^{K}(\pi^{\ast} B, \Gm)[2]$,
or in other words, $\Tor_{n}^{K}(\pi^{\ast} B, \Gm) = 0$ for $n \ne 1$,
and $\Tor_{1}^{K}(\pi^{\ast} B, \Gm)$ is a finite multiplicative group scheme over $K$.
To see this, we may assume that $B$ is constant and moreover $B = \Z / m \Z$ for $m \ge 1$
since we may check the statement \'etale locally on $k$
making an unramified extension of $K$.
The exact sequence $0 \to \Z \to \Z \to \Z / n \Z \to 0$ shows
$\Tor_{n}^{K}(\Z / m \Z, \Gm) = 0$ for $n \ne 1$
and $\Tor_{1}^{K}(\Z / m \Z, \Gm) = \mu_{m}$.
Hence the statement follows.

The above proposition and the trace map $R \pi_{!} \Gm[1] = R \alg{\Gamma}(\Gm) \to \Z$
yield a morphism
	\[
			R \pi_{!} \pi^{!} B
		\cong
			B \dtensor_{k} R \pi_{!} \Gm[1]
		\to
			B.
	\]
Let $A \in \Ab(K_{\fppf} / k^{\ind\rat}_{\et})$.
We have a morphism
	\begin{align*}
		&
				R \pi_{\ast}
				\tau_{\le -1}
				R \sheafhom_{K}(A, \pi^{!} B)
			\to
				R \pi_{\ast}
				R \sheafhom_{K}(A, \pi^{!} B)
		\\
		&	\quad
			\to
				R \sheafhom_{k}(R \pi_{!} A, R \pi_{!} \pi^{!} B)
			\to
				R \sheafhom_{k}(R \pi_{!} A, B).
	\end{align*}
The truncation $\tau_{\le -1}$ here is needed
to ignore extensions of degree $\ge 2$ as sheaves over positive characteristic $K$
that do not come from extensions as commutative group schemes (see also \cite{Bre69}).
This is unnecessary if $K$ has characteristic zero.

\begin{Prop} \label{prop: right adjoint to the pushforward}
	Assume that $A$ (resp.\ $B$) is a commutative finite flat (resp.\ \'etale) group scheme
	over $K$ (resp.\ $k$) without connected unipotent part.
	Then the morphism
		\begin{equation} \label{eq: right adjoint to the pushforward}
				R \pi_{\ast}
				\tau_{\le -1}
				R \sheafhom_{K}(A, \pi^{!} B)
			\to
				R \sheafhom_{k}(R \pi_{!} A, B)
		\end{equation}
	in $D(k)$ defined above is an isomorphism.
\end{Prop}

\begin{proof}
	Take a resolution $0 \to B_{1} \to B_{2} \to B \to 0$ of $B$ by
	\'etale group schemes $B_{1}$, $B_{2}$ over $k$
	whose groups of geometric points are finite free abelian groups.
	We have a distinguished triangle
		\[
				R \sheafhom_{K}(A, \pi^{\ast} B_{1} \tensor_{K} \Gm[1])
			\to
				R \sheafhom_{K}(A, \pi^{\ast} B_{2} \tensor_{K} \Gm[1])
			\to
				R \sheafhom_{K}(A, \pi^{!} B).
		\]
	We have $\sheafext_{K}^{1}(A, \Gm) = 0$ by Cartier duality (\cite[III, Lemma 4.17]{Mil80}).
	Hence the sheaves $\sheafext_{K}^{1}(A, \pi^{\ast} B_{i} \tensor_{K} \Gm) = 0$ for $i = 1, 2$
	are locally zero and hence zero.
	Hence we have a distinguished triangle
		\[
				\sheafhom_{K}(A, \pi^{\ast} B_{1} \tensor_{K} \Gm)[1]
			\to
				\sheafhom_{K}(A, \pi^{\ast} B_{2} \tensor_{K} \Gm)[1]
			\to
				\tau_{\le -1} R \sheafhom_{K}(A, \pi^{!} B).
		\]
	Therefore, what we need to show is that the morphism
		\[
				R \pi_{\ast}
				\sheafhom_{K}(A, \pi^{\ast} B_{i} \tensor_{K} \Gm)[1]
			\to
				R \sheafhom_{k}(R \pi_{!} A, B_{i})
		\]
	is an isomorphism for $i = 1, 2$.
	This is reduced to the case $B_{i} = \Z$ for $i = 1, 2$ by
	extending the base $k$ to its finite extension.
	The statement is then Theorem \ref{thm: duality, finite flat coefficients}.
\end{proof}

\begin{Rmk} \BetweenThmAndList
	\begin{enumerate}
	\item
		Propositions \ref{prop: projection formula} and \ref{prop: right adjoint to the pushforward}
		are true for any \'etale group scheme $B$,
		using the same definition \eqref{eq: !-pullback} of $\pi^{!} B$.
		The second proposition needs the fact that the functors
		$R \Hom_{K}(A, \;\cdot\;), R \Hom_{k}(R \pi_{!}A, \;\cdot\;)$
		with $A$ a finite flat group scheme over $K$
		commute with filtered direct limits.
		See Lemma \ref{lem: lim and Ext commutes}.
	\item
		In fact, the definition \eqref{eq: !-pullback} of $\pi^{!} B$ makes sense
		for any sheaf $B \in \Ab(k^{\ind\rat}_{\et})$.
		But Proposition \ref{prop: projection formula} seems false in general.
		An interesting observation in this direction is the following.
		Let $A = \Gm$ and $B \in \Alg / k$.
		About the left-hand side of \eqref{eq: right adjoint to the pushforward},
		we have a natural morphism
			\[
					R \pi_{\ast} \pi^{\ast} B[1]
				\to
					R \pi_{\ast}
					R \sheafhom_{K}(\Gm, \pi^{\ast} B \dtensor_{K} \Gm[1])
				=
					R \pi_{\ast}
					R \sheafhom_{K}(\Gm, \pi^{!} B),
			\]
		and about the right,
		an isomorphism
			\[
					R \sheafhom_{k}(R \pi_{!} \Gm, B)
				=
					R \sheafhom_{k}(\alg{K}^{\times}, B)[1].
			\]
		We do not seem to have a morphism
			\[
					R \pi_{\ast} \pi^{\ast} B
				\to
					R \sheafhom_{k}(\alg{K}^{\times}, B)
			\]
		between them.
		Even taking $H^{0}$ and evaluating the both sheaves by $k$,
		we still do not seem to have a morphism
			\[
					(\pi^{\ast} B)(K)
				\to
					\Hom_{k}(\alg{K}^{\times}, B).
			\]
		A remark is that this hypothetical morphism is taking its shape
		quite similar to the Albanese property of $\alg{K}^{\times}$ (\cite{CC94})
		in the equal characteristic case.
		We recall this property here.
		
		Assume that $K$ has equal characteristic.
		Let $k^{\sch}$ be the category of (not necessarily perfect) $k$-algebras.
		The functor $\alg{K}$ makes sense also on $k^{\sch}$ by the same definition:
		it sends a $k$-algebra $R$ to $(R \hat{\tensor}_{k} \Order_{K}) \tensor_{\Order_{K}} K$.
		Since $K$ is a $k$-algebra,
		we have a morphism $\pi_{0} \colon \Spec K \to \Spec k$ of schemes.
		Let $B$ be a commutative smooth algebraic group over $k$
		and $\pi_{0}^{\ast} B = B \times_{k} K$ the base extension of $B$ from $k$ to $K$.
		Then there exists an isomorphism
			\[
					(\pi_{0}^{\ast} B)(K)
				\cong
					\Hom_{k^{\sch}}(\alg{K}^{\times}, B).
			\]
		This is the Albanese property of $\alg{K}^{\times}$ (\cite{CC94}).
		
		We do not know whether or not there exists a way to modify
		our formulation and Proposition \ref{prop: right adjoint to the pushforward} on duality
		so that it contains the Albanese property of $\alg{K}^{\times}$
		as the special case where $A = \Gm$ and $B \in \Alg / k$.
	\end{enumerate}
\end{Rmk}


\section{Extensions of algebraic groups as sheaves on the rational \'etale site}
\label{sec: extensions of algebraic groups as sheaves on the rational etale site}

In this section, we prove Theorem \ref{thm: comparison of Ext, proalgebraic setting}.
This and the next sections do not involve complete discrete valuation fields.

\subsection{The pro-fppf site}
\label{sec: pro-rational and pro-fppf sites}

Some general results on perfect schemes in modern language can be found in \cite[\S 3]{BS17}.
We first make some remarks about the fppf topology on the category $k^{\perf}$ of perfect $k$-algebras.
As in Notation, we say that a homomorphism $R \to S$ in $k^{\perf}$ is finitely presented
if it can be written as the perfection of a $k$-algebra homomorphism $R \to S_{0}$ of finite presentation.
For a perfect $k$-algebra $R$,
we say that a perfect $R$-algebra $S$ is \emph{flat of finite presentation}
if it is the perfection of a flat $R$-algebra of finite presentation.
(It is not clear if this is stronger than saying that
$S$ is flat \emph{and} the perfection of an $R$-algebra of finite presentation).

Note that if $R$ is the perfection of a $k$-algebra $R_{0}$ and
$S$ is a perfect $R$-algebra flat of finite presentation in the above sense,
then there exists an $R_{0}$-algebra $S_{0}$ flat of finite presentation (in the usual sense)
whose perfection is isomorphic to $S$ as an $R$-algebra.
To see this, 
write $S$ as the perfection of a flat $R$-algebra $S_{0}'$ of finite presentation.
Let $R_{0} \to R_{0}^{(1)} \to R_{0}^{(2)} \to \cdots$ be the Frobenius morphisms over $k$,
where $R_{0}^{(n)}$ is the ring $R_{0}$ with $k$-algebra structure given by
$a \cdot b := a^{p^{n}} b$, $a \in k$, $b \in R_{0}$.
There is an $R_{0}^{(n)}$-algebra $S_{0}''$ of finite presentation for some $n$
such that $R \tensor_{R_{0}^{(n)}} S_{0}'' \cong S_{0}'$ as $R$-algebras.
The flatness of $S_{0}'$ over $R$ implies that
we can take $S_{0}''$ to be flat over $R_{0}^{(n)}$ (for some larger $n$)
by the permanence property of flatness under passage to limits (\cite[Corollary 11.2.6.1]{Gro66}).
Then $S_{0}''^{(-n)}$ is an $R_{0}$-algebra flat of finite presentation
and we have $R \tensor_{R_{0}} S_{0}''^{(-n)} \cong S_{0}'^{(-n)}$ as $R^{(-n)} = R$-algebras.
Hence $S$ is the perfection of the flat $R_{0}$-algebra $S_{0} = S_{0}''^{(-n)}$ of finite presentation.

The permanence property of flatness under passage to limits used above
has the following variant for perfect flat algebras of finite presentation:
if $\{R_{\lambda} \to S_{\lambda}\}$ is a filtered direct system of finitely presented homomorphisms
between perfect $k$-algebras
such that $S_{\mu} = S_{\lambda} \tensor_{R_{\lambda}} R_{\mu}$ for any $\mu \ge \lambda$
and if its direct limit $R \to S$ is a perfect $k$-algebra homomorphism flat of finite presentation,
then so is $R_{\mu} \to S_{\mu}$ for some $\mu$.
To see this,
take a flat $R$-algebra $S_{0}$ of finite presentation in the usual sense whose perfection is $S$.
For some $\lambda$, there exists a flat $R_{\lambda}$-algebra $S_{0 \lambda}$
of finite presentation in the usual sense
such that $S_{0} \cong S_{0 \lambda} \tensor_{R_{\lambda}} R$
by \cite[Corollary 11.2.6.1]{Gro66}.
Set $S_{0 \mu} = S_{0 \lambda} \tensor_{R_{\lambda}} R$ for $\mu \ge \lambda$.
Take an $R_{\lambda}$-algebra $S_{1 \lambda}$ of finite presentation in the usual sense
whose perfection is $S_{\lambda}$
and set $S_{1} = S_{1 \lambda} \tensor_{R_{\lambda}} R$
and $S_{1 \mu} = S_{1 \lambda} \tensor_{R_{\lambda}} R_{\mu}$ for $\mu \ge \lambda$.
Since the perfections of $S_{0}$ and $S_{1}$ are both isomorphic to $S$ as $R$-algebras,
there exist $R$-algebra homomorphisms
$\varphi \colon S_{0} \to S_{1}^{(n)}$ and $\psi \colon S_{1} \to S_{0}^{(m)}$ for some $n, m \ge 0$
such that $\psi \compose \varphi$ and $\varphi \compose \psi$ are $p^{n + m}$-th power Frobenius homomorphisms.
For some $\mu \ge \lambda$,
the homomorphisms $\varphi$ and $\psi$ come (via base change $\tensor_{R_{\mu}} R$)
from some $R_{\mu}$-algebra homomorphisms
$\varphi_{\mu} \colon S_{0 \mu} \to S_{1 \mu}^{(n)}$ and
$\psi_{\mu} \colon S_{1 \mu} \to S_{0 \mu}^{(m)}$
such that $\psi_{\mu} \compose \varphi_{\mu}$ and $\varphi_{\mu} \compose \psi_{\mu}$ are
$p^{n + m}$-th power Frobenius homomorphisms
by \cite[Theorem 8.8.2 (i)]{Gro66}.
Hence the perfections of $S_{0 \mu}$ and $S_{1 \mu}$ are isomorphic to each other as $R_{\mu}$-algebras.
Therefore $S_{\mu}$ is a perfect flat $R_{\mu}$-algebra of finite presentation.

Note also that if $R \to S$ and $S \to T$ are flat homomorphisms of finite presentation
between perfect $k$-algebras,
then so is the composite $R \to T$.
To see this, write $S$ as the perfection of a flat $R$-algebra $S_{0}$ of finite presentation.
By what we saw in the second last paragraph,
$T$ is the perfection of an $S_{0}$-algebra $T_{0}$ flat of finite presentation.
Hence $T$ is the perfection of the flat $R$-algebra $T_{0}$ of finite presentation.

Therefore we can define the \emph{perfect fppf site} $\Spec k^{\perf}_{\fppf}$
to be the category $k^{\perf}$ of perfect $k$-algebras
endowed with the topology whose covering families over $R \in k^{\perf}$
are finite families $\{S_{i}\}$ of perfect flat $R$-algebras of finite presentation
with $\prod_{i} S_{i}$ faithfully flat over $R$.
The perfection functor from the category of all $k$-algebras to $k^{\perf}$
induces a morphism of sites $\Spec k^{\perf}_{\fppf} \to \Spec k_{\fppf}$
to the usual fppf site of $k$,
whose pushforward functor is exact.
Hence it induces an isomorphism on cohomology theory.
For a perfect $k$-algebra $R$, the category of objects over $R$ in $k^{\perf}$
is nothing but the category of perfect $R$-algebras,
in contrast to the case of the category of ind-rational $k$-algebras $k^{\ind\rat}$.
Hence we can write the localization of $\Spec k^{\perf}_{\fppf}$ at $R$ by
$\Spec k^{\perf}_{\fppf} / R = \Spec R^{\perf}_{\fppf}$ without ambiguity.

To define the perfect pro-fppf topology, we need the following class of homomorphisms.

\begin{Def}
	Let $R$ be a perfect $k$-algebra.
	We say that a perfect $R$-algebra $S$ is \emph{flat of ind-finite presentation}
	if there exists a filtered direct system $\{R_{\lambda}\}$ of perfect $R$-algebras
	such that each $R_{\lambda}$ is flat of finite presentation over $R$ and
	$S$ is isomorphic to $\dirlim_{\lambda} R_{\lambda}$
	as an $R$-algebra.
	In this case, we also say that
	$\Spec S \to \Spec R$ is \emph{flat of profinite presentation}.
	If moreover $S$ is faithfully flat over $R$,
	we say that $S$ is \emph{faithfully flat of ind-finite presentation} over $R$
	and $\Spec S$ is \emph{faithfully flat of profinite presentation} over $\Spec R$.
\end{Def}

Note that the transition homomorphisms $R_{\lambda} \to R_{\mu}$ are
not required to be flat.
An example of a flat homomorphism of ind-finite presentation
is a localization of a perfect $k$-algebra by an arbitrary multiplicative subset.
Another example is a surjection in $\Pro \Alg / k$.
Note that not every flat homomorphism is of ind-finite presentation.%
\footnote{\textit{Stacks Project}, \texttt{http://stacks.math.columbia.edu}, Tag 0ATE.}

The permanence property of flatness under passage to limits used above
with a standard limit argument shows the following.

\begin{Prop} \label{prop: composite of flat of ind-finite presentation}
	If $R \to S$ and $S \to T$ are flat homomorphisms of ind-finite presentation
	between perfect $k$-algebras,
	then so is the composite $R \to T$.
\end{Prop}

\begin{proof}
	Write $S = \dirlim_{\lambda} S_{\lambda}$
	with $S_{\lambda}$ perfect flat of finite presentation over $R$
	and $T = \dirlim_{\mu} T_{\mu}$
	with $T_{\mu}$ perfect flat of finite presentation over $S$.
	By \cite[Lemma 1.5]{Swa98} (plus perfection),
	it is enough to show that any $R$-algebra homomorphism
	$T' \to T$ from a perfect $R$-algebra $T'$ of finite presentation
	factors through a perfect $R$-algebra $T''$ flat of finite presentation.
	
	The $R$-algebra homomorphism $T' \to T$ factors through $T_{\mu}$ for some $\mu$.
	For some $\lambda$, there is a perfect $S_{\lambda}$-algebra $T_{\lambda \mu}$ of finite presentation
	such that $S \tensor_{S_{\lambda}} T_{\lambda \mu} \cong T_{\mu}$ as $S$-algebras.
	For $\lambda' \ge \lambda$, let $T_{\lambda' \mu} = S_{\lambda'} \tensor_{S_{\lambda}} T_{\lambda \mu}$.
	Then $T_{\lambda' \mu}$ is a perfect $S_{\lambda'}$-algebra flat of finite presentation
	for some $\lambda' \ge \lambda$
	by the perfect-algebras variant of the permanence property of flatness under passage to limits shown above.
	The $R$-algebra homomorphism $T' \to T_{\mu}$ factors through $T_{\lambda'' \mu}$
	for some $\lambda'' \ge \lambda'$.
	Therefore the $R$-algebra homomorphism $T' \to T$ factors through
	the perfect $R$-algebra $T_{\lambda'' \mu}$ flat of finite presentation.
\end{proof}

Clearly the base extension of a flat homomorphism $R \to S$ of ind-finite presentation
by arbitrary homomorphism $R \to R'$ is flat of ind-finite presentation.
Hence we can make the following definition.

\begin{Def}
	We define the \emph{perfect pro-fppf site} $\Spec k^{\perf}_{\pro\fppf}$
	to be the category $k^{\perf}$ of perfect $k$-algebras
	endowed with the topology whose covering families over $R \in k^{\perf}$
	are finite families $\{S_{i}\}$ of perfect flat $R$-algebras of ind-finite presentation
	with $\prod_{i} S_{i}$ faithfully flat over $R$.
	For a perfect affine scheme $X$,
	the localization $\Spec k^{\perf}_{\pro\fppf} / X$ of $\Spec k^{\perf}_{\pro\fppf}$ at $X$
	(see Notation in Introduction) is also denoted by $X_{\pro\fppf}$.
\end{Def}

\begin{Rmk} \BetweenThmAndList
	\begin{enumerate}
		\item
			In the same way,
			we can define pro-smooth morphisms and the perfect pro-smooth site
			using \cite[Proposition 17.7.8]{Gro67} instead.
			The rest of this section also works
			when we use the perfect pro-smooth site instead of pro-fppf.
		\item
			If we want to allow proper varieties to the underlying category of the pro-fppf site,
			we can use the category of perfect quasi-compact quasi-separated schemes over $k$
			instead of affine ones,
			with the help of the absolute noetherian approximation.
	\end{enumerate}
\end{Rmk}


\subsection{Basic relations with the ind-rational \'etale site}
\label{sec: basic relations with the rational etale site}

We have continuous maps of sites
	\[
			\Spec k^{\perf}_{\pro\fppf}
		\to
			\Spec k^{\perf}_{\fppf}
		\to
			\Spec k^{\perf}_{\et}
		\to
			\Spec k^{\ind\rat}_{\et}
	\]
defined by the identity.
The first two are morphisms of sites (that is, have exact pullback functors),
but the third is not (see Proposition \ref{prop: pullback is not exact} below for an example).

\begin{Def} \label{def: points in profinite setting}
	Let $X = \Spec R$ be a perfect affine $k$-scheme.
	Write $R = \bigcup_{\lambda} R_{\lambda}$
	as a filtered union of perfect $k$-subalgebras $R_{\lambda}$
	of finite type.
	\begin{enumerate}
		\item \label{def: profinite set of points}
			A \emph{profinite set of points of $X$} is a subsystem
			$x = \invlim_{\lambda} x_{\lambda} \subset X = \invlim_{\lambda} \Spec R_{\lambda}$
			consisting of finite subsets $x_{\lambda}$ of $\Spec R_{\lambda}$.
		\item \label{def: generic point}
			Assume that all the transition homomorphisms $R_{\lambda} \to R_{\mu}$ are flat.
			Then we have $\Frac R = \dirlim_{\lambda} \Frac R_{\lambda} \in k^{\ind\rat}$,
			where $\Frac$ denotes the total quotient ring.
			We call the profinite set of points $\xi_{X} = \Spec \Frac R$
			the \emph{generic point of $X$}.
	\end{enumerate}
	These notions are independent of the presentation $R = \bigcup_{\lambda} R_{\lambda}$
	(with flat transition homomorphisms in the second definition).
\end{Def}

Note that a profinite set of points is the $\Spec$ of an ind-rational $k$-algebra.
The set of all profinite sets of points of a given $X$ is a directed set by inclusion.
Any morphism from a perfect $k$-scheme $\Spec k'$ with $k' \in k^{\ind\rat}$ to $X$
factors through a uniquely determined profinite set of points of $X$.
Hence we have the following.

\begin{Prop} \label{prop: schemes as sheaves on the rational site}
	A perfect affine $k$-scheme $X$ regarded as a sheaf on $\Spec k^{\ind\rat}_{\et}$
	is the filtered directed union of the profinite sets of points of $X$.
	If $X$ is of finite type over $k$,
	then $X$ regarded as a sheaf on $\Spec k^{\ind\rat}_{\et}$
	is the disjoint union of points of $X$.
\end{Prop}

If $f \colon Y \to X$ is a morphism in $k^{\perf}$ and $y$ a profinite set of points of $Y$,
then $f(y)$ is a profinite set of points of $X$.
Under the assumption made in \eqref{def: generic point} of Definition \ref{def: points in profinite setting},
the natural inclusion $\xi_{X} \into X$ is flat of profinite presentation and dominant.
We can speak of the generic point of a proalgebraic group $A \in \Pro \Alg / k$
since $A$ satisfies the assumption in \eqref{def: generic point}.
Any $k$-algebra homomorphism from an ind-rational $k$-algebra to a perfect $k$-algebra is
flat of ind-finite presentation,
which is faithfully flat if it is injective.

We should be careful when using the continuous map to $\Spec k^{\ind\rat}_{\et}$
due to the following.

\begin{Prop} \label{prop: pullback is not exact}
	Let $h \colon \Spec k^{\perf}_{\pro\fppf} \to \Spec k^{\ind\rat}_{\et}$
	be the continuous map defined by the identity.
	Let $h^{\ast\set} \colon \Set(k^{\ind\rat}_{\et}) \to \Set(k^{\perf}_{\pro\fppf})$
	be the pullback functor for sheaves of sets.
	Let $\Affine^{1}$ be the affine line over $k$.
	Then the natural morphism
		\[
				h^{\ast\set} (\Affine^{1} \times_{k} \Affine^{1})
			\to
				h^{\ast\set} \Affine^{1} \times_{k} h^{\ast\set} \Affine^{1}
		\]
	is \emph{not} an isomorphism.
	In particular, $h^{\ast\set}$ is not exact
	and $h$ is not a morphism of sites.
\end{Prop}

\begin{proof}
	By Proposition \ref{prop: schemes as sheaves on the rational site},
	we have
		\[
				\Affine^{1} \times_{k} \Affine^{1}
			=
				\bigsqcup_{z \in \Affine^{2}} z,
			\quad
				\Affine^{1}
			=
				\bigsqcup_{x \in \Affine^{1}} x
		\]
	in $\Set(k^{\ind\rat}_{\et})$,
	where the disjoint unions are over points of the underlying set of the schemes.
	Hence
		\[
				h^{\ast\set} (\Affine^{1} \times_{k} \Affine^{1})
			=
				\bigsqcup_{z \in \Affine^{2}} z,
			\quad
				h^{\ast\set} \Affine^{1} \times_{k} h^{\ast\set} \Affine^{1}
			=
				\bigsqcup_{x, y \in \Affine^{1}} x \times_{k} y
		\]
	in $\Set(k^{\perf}_{\pro\fppf})$.
	The natural morphism between them is clearly not an isomorphism.
\end{proof}

By the same reason,
the continuous maps from $\Spec k^{\perf}_{\fppf}$ or $\Spec k^{\perf}_{\et}$
to $\Spec k^{\ind\rat}_{\et}$ are not morphisms of sites.


\subsection{Pro-fppf, fppf and \'etale cohomology}
\label{sec: Pro-fppf, fppf and etale cohomology}

\begin{Prop} \label{prop: pro-fppf and fppf cohomology}
	Let $f \colon \Spec k^{\perf}_{\pro\fppf} \to \Spec k^{\perf}_{\fppf}$
	be the morphism defined by the identity.
	Let $B \in \Ab(k^{\perf}_{\pro\fppf})$.
	Assume that $B$ as a functor on $k^{\perf}$ commutes with filtered direct limits.
	Then we have
		\[
				R \Gamma(R_{\pro\fppf}, B)
			=
				R \Gamma(R_{\fppf}, f_{\ast} B)
		\]
	for any $R \in k^{\perf}$.
	In other words, we have $R^{j} f_{\ast} B = 0$ for any $j \ge 1$.
\end{Prop}

We need a lemma.
Let $B$ and $R$ as in the proposition.
We denote by $\check{H}^{i}(R_{\pro\fppf}, B)$ the pro-fppf \v{C}ech cohomology
and by $\underline{H}_{\pro\fppf}^{i}(B)$ the presheaf
$S \mapsto H^{i}(S_{\pro\fppf}, B)$.
The same notation is applied to the fppf cohomology.
Note that \v{C}ech cohomology is defined for any presheaf.
Hence both the pro-fppf \v{C}ech cohomology and the fppf \v{C}ech cohomology
with coefficients in $\underline{H}_{\fppf}^{j}(f_{\ast }B)$
make sense.

\begin{Lem}
	Let $B$ and $R$ as in the proposition.
	We have
		\[
				\check{H}^{i}(R_{\pro\fppf}, \underline{H}_{\fppf}^{j}(f_{\ast }B))
			=
				\check{H}^{i}(R_{\fppf}, \underline{H}_{\fppf}^{j}(f_{\ast }B))
		\]
	for any $i, j \ge 0$.
\end{Lem}

\begin{proof}
	Let $S \in k^{\perf}$ be faithfully flat of ind-finite presentation over $R$.
	By definition, we can write
	$S = \dirlim_{\lambda} S_{\lambda}$
	by a filtered directed system of faithfully flat perfect $R$-algebras $S_{\lambda}$ of finite presentation.
	It is enough to show that
		\[
				\check{H}^{i}(S / R, \underline{H}_{\fppf}^{j}(f_{\ast }B))
			=
				\dirlim_{\lambda}
				\check{H}^{i}(S_{\lambda} / R, \underline{H}_{\fppf}^{j}(f_{\ast }B)).
		\]
	We have
		$
				S^{\tensor_{R} (i + 1)}
			=
				\dirlim_{\lambda} S_{\lambda}^{\tensor_{R} (i + 1)}
		$.
	Since $B$ commutes with filtered direct limits, we have
		\[
				H^{j}((S^{\tensor_{R} (i + 1)})_{\fppf}, f_{\ast} B)
			=
				\dirlim_{\lambda}
				H^{j}((S_{\lambda}^{\tensor_{R} (i + 1)})_{\fppf}, f_{\ast }B)
		\]
	for $j \ge 0$ (argue inductively using the \v{C}ech-to-derived functor spectral sequence).
	Therefore the \v{C}ech complex of $S / R$
	with coefficients in $\underline{H}_{\fppf}^{j}(f_{\ast} B)$
	is the direct limit of the \v{C}ech complex of $S_{\lambda} / R$
	with coefficients in $\underline{H}_{\fppf}^{j}(f_{\ast }B)$.
	The result follows by taking cohomology.
\end{proof}

\begin{proof}[Proof of Proposition \ref{prop: pro-fppf and fppf cohomology}]
	We need to prove that the $j'$-th cohomology groups
	of the both sides are isomorphic for all $j' \ge 0$.
	We prove this by induction.
	The case $j' = 0$ is obvious.
	Fix $j' \ge 1$.
	Assume the equality of the $j$-th cohomology for $j = 0, 1, \dots, j' - 1$.
	Consider the \v{C}ech-to-derived functor spectral sequences
		\begin{gather*}
				E_{2, \pro\fppf}^{i j}
			=
				\check{H}^{i}(R_{\pro\fppf}, \underline{H}_{\pro\fppf}^{j}(B))
			\Longrightarrow
				H^{i + j}(R_{\pro\fppf}, B),
			\\
				E_{2, \fppf}^{i j}
			=
				\check{H}^{i}(R_{\fppf}, \underline{H}_{\fppf}^{j}(f_{\ast }B))
			\Longrightarrow
				H^{i + j}(R_{\fppf}, f_{\ast }B).
		\end{gather*}
	The induction hypothesis implies that
	$\underline{H}_{\pro\fppf}^{j}(B) = \underline{H}_{\fppf}^{j}(f_{\ast} B)$
	for $j = 0, 1, \dots, j' - 1$.
	Therefore we have $E_{\infty, \pro\fppf}^{i j} = E_{\infty, \fppf}^{i j}$
	for $i + j = j'$, $j = 0, 1, \dots, j' -  1$ by the above lemma.
	On the other hand, we have
	$E_{2, \pro\fppf}^{0 j'} = E_{2, \fppf}^{0 j'} = 0$
	by \cite[III, Proposition 2.9]{Mil80}.
	Hence the case $j = j'$ follows.
\end{proof}

Objects of $\Alg / k$ and, more generally, $\Loc \Alg / k$
are perfections of smooth group schemes
and hence commutes with filtered direct limits as functors on $k^{\perf}$.
Since the fppf cohomology with coefficients in a smooth group scheme agrees with
the \'etale cohomology (\cite[III, Remark 3.11 (b)]{Mil80}),
we have the following.

\begin{Cor} \label{cor: profppf and etale cohomology of an algebraic group}
	Let $B \in \Loc \Alg / k$.
	Then we have
		\[
				R \Gamma(R_{\pro\fppf}, B)
			=
				R \Gamma(R_{\et}, B)
		\]
	for any $R \in k^{\perf}$.
	In other words, we have $R^{j} f_{\ast} B = 0$ for any $j \ge 1$,
	where $f \colon \Spec k^{\perf}_{\pro\fppf} \to \Spec k^{\perf}_{\et}$
	is the morphism defined by the identity.
\end{Cor}


\subsection{Review of Mac Lane's resolution}
\label{sec: review of Mac Lane's resolution}

We review Mac Lane's resolution \cite{Mac57}.
We review only the part necessary for our constructions and proofs.
In particular, the base ring is taken to be $\Z$.
We denote by $\Ab$, $\GrAb$, $\DGAb$
the categories of abelian groups,
graded abelian groups,
differential graded abelian groups, respectively.
For a set $X$, we denote by $\Z[X]$ the free abelian group generated by $X$.

Let $A$ be an abelian group.
Let $C_{n} = \{0, 1\}^{n}$ for $n \ge 0$.
Define a graded abelian group $Q'(A) = \bigoplus_{n \ge 0} Q_{n}'(A)$ by
	\[
			Q_{n}'(A)
		=
			\Z[\Map(C_{n}, A)]
		=
			\Z[A^{2^{n}}].
	\]
There are certain differential maps $\boundary \colon Q_{n}'(A) \to Q_{n - 1}'(A)$
coming from the combinatorics of the vertices $C_{n}$ of the $n$-cubes
and the group structure of $A$.
The explicit definition is not needed for our purpose,
but see \cite[\S 4]{Mac57} especially for formulas in low degrees.
The assignment $A \mapsto Q'(A)$ defines a (non-additive) functor
$Q' \colon \Ab \to \DGAb$.
There is a certain differential graded subgroup $N_{A}$ of $Q'(A)$
generated by the so-called ``norms''.
The cubical construction $Q(A)$ is the normalization defined as
the differential graded abelian group $Q(A) = Q'(A) / N_{A}$.
The assignment $A \mapsto Q(A)$ defines a (non-additive) functor
$Q \colon \Ab \to \DGAb$.
For each $n \ge 0$, the degree $n$ part $Q_{n}(A)$ is a direct summand of the abelian group $Q_{n}'(A)$
and a splitting $Q_{n}'(A) \cong Q_{n}(A) \oplus (N_{A})_{n}$ can be taken functorially
(\cite[\S 5]{Pir96}, \cite[Proposition 2.6]{JP91}).

Mac Lane's resolution $M(A)$ of $A$ can be written functorially as a graded abelian group as
	\[
			M(A)
		=
			Q(A) \tensor_{\Z} \mathcal{B},
	\]
where $\mathcal{B}$ is a certain graded abelian group
$\bar{B}(0, Q(\Z), \eta_{Q})$ in Mac Lane's notation
\cite[\S 7, Remarque 1]{Mac57}.
The group $\mathcal{B}$ does not depend on $A$ and
has homogeneous parts all free abelian groups.
In particular,
each homogeneous part of $M(A)$ is a direct summand of a direct sum of groups of the form
$\Z[A^{m}]$ for various $m \ge 0$.
There are certain differential maps $\boundary \colon M_{n}(A) \to M_{n - 1}(A)$
(which is not the tensor product of the differentials for $Q(A)$ and $\mathcal{B}$).
See \cite[\S 7]{Mac57} for formulas in low degrees.
The assignment $A \mapsto M(A)$ defines a (non-additive) functor
$M \colon \Ab \to \DGAb$.
(In the language of two-sided bar constructions,
the group $\mathcal{B}$ is the unnormalized bar construction $B(\Z, Q(\Z), \Z)$
for the augmented differential graded ring $Q(\Z)$
and the group $M(A)$ is $B(Q(A), Q(\Z), \Z)$.)
Mac Lane's theorem is that $M(A)$ is a resolution of $A$, namely
	\begin{gather*}
			H_{n}(M(A)) = 0
		\quad \text{for} \quad
			n > 0,
		\\
			H_{0}(M(A)) = A
		\quad \text{(functorial isomorphism).}
	\end{gather*}
For later use, we define $M'(A) = Q'(A) \tensor_{\Z} \mathcal{B} \in \GrAb$.
Each homogeneous part of $M(A)$ is a direct sum of groups of the form
$\Z[A^{m}]$ for various $m \ge 0$.
Since $Q(A)$ is a functorial direct summand of $Q'(A)$ in $\GrAb$,
we know that $M(A)$ is a functorial direct summand of $M'(A)$ in $\GrAb$.

In the proof of our theorem,
we will need splitting homotopy $V$ for $Q(A)$ and $M(A)$
\cite[\S 5 and \S 8, respectively]{Mac57} with respect to additive projections,
so we recall it here.
Assume that $A$ is the direct sum of two abelian groups $A_{0}$ and $A_{1}$.
Let $p_{0}, p_{1} \colon A \to A$ be the corresponding projections.
These maps induce endomorphisms
$p_{0}, p_{1} \colon Q(A) \to Q(A)$ and $p_{0}, p_{1} \colon M(A) \to M(A)$
of differential graded abelian groups.
Note that $p_{0} + p_{1} \ne \id$ on $Q(A)$ or $M(A)$ in general
since $Q$ and $M$ are not additive functors.
Define an endomorphism $V$ on the graded abelian group $Q(A)$ of degree $+ 1$ by
	\[
			(V t)(\varepsilon, e)
		=
			p_{\varepsilon}(t(e))
	\]
for $t \colon C_{n} \to A$, $e \in C_{n}$ and $\varepsilon = 0, 1$.
Then
	\[
			- \boundary V - V \boundary = \id - p_{0} - p_{1}
		\quad \text{on} \quad
			Q(A).
	\]
The endomorphism $V$ induces the endomorphism $V = V \tensor \id$ 
on the graded abelian group $M(A) = Q(A) \tensor_{\Z} \mathcal{B}$ of degree $+ 1$.
Then
	\[
			- \boundary V - V \boundary = \id - p_{0} - p_{1}
		\quad \text{on} \quad
			M(A).
	\]

We can do these constructions
with arbitrary pairs of homomorphisms generalizing projections
in the following way (see also \cite[Theorem 11.2]{EM51}).
Let $A$ and $B$ be abelian groups and $p_{0}, p_{1} \colon A \to B$ any homomorphisms.
Set $p = p_{0} + p_{1} \colon A \to B$.
The product homomorphism $(p_{0}, p_{1}) \colon A \to B^{2}$ and the functoriality of $Q$ induce a morphism
$Q(A) \to Q(B^{2})$ in $\DGAb$.
As above, we have the splitting homotopy $Q(B^{2}) \to Q(B^{2})$ with respect to the natural two projections on $B^{2}$.
The summation map $B^{2} \onto B$ and the functoriality of $Q$ induce a morphism
$Q(B^{2}) \to Q(B)$ in $\DGAb$.
Composing these three morphisms in this order,
we obtain a homomorphism $Q(A) \to Q(B)$ of graded abelian groups of degree $+ 1$,
which we again call the splitting homotopy (with respect to $p_{0}$ and $p_{1}$)
and denote by the same symbol $V$.
A similar construction gives a homomorphism $V \colon M(A) \to M(B)$ of graded abelian groups of degree $+ 1$.
The splitting  homotopy conditions above yield equalities
	\begin{gather*}
				- \boundary V - V \boundary = p - p_{0} - p_{1}
			\colon
				Q(A)
			\to
				Q(B),
		\\
				- \boundary V - V \boundary = p - p_{0} - p_{1}
			\colon
				M(A)
			\to
				M(B).
	\end{gather*}

\begin{Rmk}
	In \cite[\S 5 and \S 8]{Mac57},
	the splitting homotopy $V$ was defined
	also for non-additive projections.
	The non-additivity makes the behavior of $V$ much more complicated.
	In this paper, we use $V$ only for additive projections.
\end{Rmk}


\subsection{Pullback of Mac Lane's resolution from the category of fields}
\label{sec: pullback of Mac Lane's resolution}

Let $\mathcal{A}$ be either $\Ab$, $\GrAb$ or $\DGAb$
and $F \colon \Ab \to \mathcal{A}$ a (non-additive) functor,
for example $F = \Z[\;], Q', Q, M', M$,
where $\Z[\;] \colon A \mapsto \Z[A]$ is the free abelian group functor as above.
For a sheaf $A$ of abelian groups on a Grothendieck site $S$,
we define $F(A)$ to be the sheafification of the presheaf $X \mapsto F(A(X))$.
Then $M(A)$ is a resolution of the sheaf $A$ since sheafification is exact.
We have
	\[
			Q_{n}'(A)
		=
			\Z[A^{2^{n}}]
		\in
			\Ab(S),
		\quad
			M'(A)
		=
			Q'(A) \tensor_{\Z} \mathcal{B}
		\in
			\GrAb(S),
	\]
where $\mathcal{B}$ is regarded as a constant sheaf.
The quotient morphisms $Q'_{n}(A) \onto Q_{n}(A)$ and $M'_{n}(A) \onto M_{n}(A)$ in $\Ab(S)$ admit sections.

Let $k' \in k^{\ind\rat}$.
Let $h \colon \Spec k'^{\perf}_{\pro\fppf} \to \Spec k^{\ind\rat}_{\et} / k'$ be
the continuous map defined by the identity.
For $A' \in \Ab(k'^{\perf}_{\pro\fppf})$, we consider
	\[
			h^{\ast} F(h_{\ast} A')
		\in
			\mathcal{A}(k'^{\perf}_{\pro\fppf}).
	\]
We first define a morphism $h^{\ast} F(h_{\ast} A') \to F(A')$ in
$\mathcal{A}(k'^{\perf}_{\pro\fppf})$ using adjunction.
Let $h^{-1}$ be the pullback functor for presheaves with values in $\mathcal{A}$.
Let $\sh^{\perf}_{\pro\fppf}$ (resp.\ $\sh^{\ind\rat}_{\et}$) be the sheafification functor
with respect to the perfect pro-fppf topology (resp.\ the ind-rational \'etale topology).
Let $F^{\pre}(A')$ be the presheaf $X \mapsto F(A'(X))$.
Then we have natural isomorphisms and an adjunction morphism
	\begin{align*}
				h^{\ast} F(h_{\ast} A')
		&	=
				\sh^{\perf}_{\pro\fppf} h^{-1} \sh^{\ind\rat}_{\et} F^{\pre}(h_{\ast} A')
		\\
		&	=
				\sh^{\perf}_{\pro\fppf} h^{-1} h_{\ast} F^{\pre}(A')
		\\
		&	\to
				\sh^{\perf}_{\pro\fppf} F^{\pre}(A')
			=
				F(A').
	\end{align*}
If $A \in \Pro \Alg / k$,
we denote $A' = A \times_{k} k'$ and
write $h_{\ast} A' = A'$.

\begin{Prop} \label{prop: generic resolution}
	Let $k' \in k^{\ind\rat}$.
	Let $F = Q$ or $M$.
	For any $A \in \Pro \Alg / k$,
	the above defined morphism
		\[
				h^{\ast} F(A') \to F(A')
		 \quad \text{in} \quad
		 	\DGAb(k'^{\perf}_{\pro\fppf})
		\]
	with $A' = A \times_{k} k'$ is a quasi-isomorphism.
	In particular, we have
		\begin{gather*}
				H_{n}(h^{\ast} M(A'))
			=
				0
			\quad \text{for} \quad n > 0 \quad{and}
			\\
				H_{0}(h^{\ast} M(A'))
			=
				A'.
		\end{gather*}
\end{Prop}

This is the hardest proposition in this paper.
Its proof occupies the next subsection.
Before the proof,
we need to write $h^{\ast} F(A')$ more explicitly.
When $F = \Z[\;]$, we have
	$
			h^{\ast} \Z[A']
		=
			\Z[h^{\ast\set} A'],
	$
where $h^{\ast\set}$ denotes the pullback for sheaves of sets.
By Proposition \ref{prop: schemes as sheaves on the rational site},
we have
	\[
			h^{\ast\set} A'
		=
			\bigcup_{x \subset A'}
			x,
		\quad
			h^{\ast} \Z[A']
		=
			\Z[h^{\ast\set} A']
		=
			\bigcup_{x \subset A'}
				\Z[x],
	\]
where $x$ runs through all profinite sets of points of $A'$.
The $x$ are $\Spec$'s of objects of $k^{\ind\rat} / k'$.
The morphism $h^{\ast} \Z[A'] \to \Z[A']$ is induced from the natural inclusion
$h^{\ast\set} A' \subset A'$, or $x \subset A'$.
Also we have
	\begin{gather*}
			h^{\ast} Q_{n}'(A')
		=
			\bigcup_{x \subset A'^{2^{n}}}
				\Z[x]
		\subset
			\Z[A'^{2^{n}}]
		=
			Q_{n}'(A'),
		\\
			h^{\ast} M'(A')
		=
			h^{\ast} Q'(A') \tensor_{\Z} \mathcal{B}.
	\end{gather*}

Therefore, roughly speaking,
Proposition \ref{prop: generic resolution} says that
$A \in \Pro \Alg / k$ can completely be resolved and described by its field-valued points and group operation.

\begin{Rmk}
	The special case $k' = k$ of Proposition \ref{prop: generic resolution}
	does not imply the general case by the following reason.
	The tensor product $(\,\cdot\,) \tensor_{k} k'$
	does not define a functor $k^{\ind\rat} \to k^{\ind\rat} / k'$,
	so the localization morphism $\Set(k^{\ind\rat}_{\et} / k') \to \Set(k^{\ind\rat}_{\et})$ of topoi
	(whose pullback functor is the restriction functor; \cite[IV, \S 5.2]{AGV72a})
	does not come from a morphism $\Spec k^{\ind\rat}_{\et} / k' \to \Spec k^{\ind\rat}_{\et}$ of sites.
	Moreover, the diagram
		\[
			\begin{CD}
					\Set(k'^{\perf}_{\pro\fppf})
				@<< h^{\ast\set} <
					\Set(k^{\ind\rat}_{\et} / k')
				\\
				@AAA
				@AAA
				\\
					\Set(k^{\perf}_{\pro\fppf})
				@<<<
					\Set(k^{\ind\rat}_{\et})
			\end{CD}
		\]
	of pullback functors is not commutative.
	For example, a representable sheaf $X = \Spec k'' \in k^{\ind\rat}$
	in the right-lower corner
	maps to $\dirlim x$ ($x$ runs through all profinite sets of points of
	$X \times_{k} k' = \Spec (k'' \tensor_{k} k')$)
	through the right-upper corner
	and to $X \times_{k} k'$ through the left-lower corner.
	Therefore arguments about $\Spec k^{\perf}_{\pro\fppf} \to \Spec k^{\ind\rat}_{\et}$
	do not naturally restrict to
	$h \colon \Spec k'^{\perf}_{\pro\fppf} \to \Spec k^{\ind\rat}_{\et} / k'$.
\end{Rmk}


\subsection{Proof of the acyclicity of the pullback of Mac Lane's resolution}
\label{sec: proof of generic resolution}

The idea of proof of Proposition \ref{prop: generic resolution} is the following.
Let $X \in k'^{\perf}$ and $a \in A'(X)$.
We need to modify $a$ so as to put it into $h^{\ast\set} A'$,
which is the union of the profinite sets of points of $A'$.
By ``points'' in the definition of profinite sets of points,
we actually mean generic points of closed subschemes.
The $X$-valued point $a$ is not generic to a closed subscheme of $A'$ in general.
We can write $a = (a - \tilde{a}) + \tilde{a}$
by any other $\tilde{a} \in A'(X)$.
Since $\xi_{A'} \times_{k'} \xi_{A'} \onto A'$ is pro-fppf,
by suitably extending $X$ to its pro-fppf cover $Z$,
we can choose $\tilde{a} \in A'(Z)$ so that
each of $\tilde{a}$ and $a - \tilde{a}$ is generic to a closed subscheme of $A'$.
Since $F$ is a non-additive functor,
the decomposition of this type $a = (a - \tilde{a}) + \tilde{a}$ in $A'$
cannot directly be translated into a similar decomposition in $F(A')$.
This difference is managed by the splitting homotopy $V$
recalled in Section \ref{sec: review of Mac Lane's resolution}.
We choose $Z$ and $\tilde{a}$ carefully
so that the assignment $a \mapsto \tilde{a}$ is a homomorphism,
which is the requirement for $V$ to behave simply.
The technical core is Lemma \ref{lem: generic splitting} below.

We fix $A \in \Pro \Alg / k$ and $k' \in k^{\ind\rat}$,
so that $A' = A \times_{k} k'$.
We will need the following two lemmas
to create sufficiently many pro-fppf covers.

\begin{Lem} \label{lem: dominant is stable under flat base change}
	Let $Y \to X$, $X' \to X$ be morphisms in $k^{\perf}$
	and let $Y' = X' \times_{X} Y$.
	Assume that $X' \to X$ is flat.
	If $Y \to X$ is dominant,
	so is $Y' \to X'$.
\end{Lem}

\begin{proof}
	Write $X = \Spec R$, $Y = \Spec S$, $X' = \Spec R'$ and $Y' = \Spec S'$.
	Note that the homomorphism $R \to S$ between perfect (hence reduced) rings
	being dominant is equivalent that it is injective.
	The definition of flatness then says that $R' \to S'$ is injective.
\end{proof}

\begin{Lem} \label{lem: making fpqc covers}
	Let $Z_{i} \to Y \to X$ be morphisms in $k^{\perf}$,
	$i = 1, \dots, n$, and
	let $Z = Z_{1} \times_{Y} \dots \times_{Y} Z_{n}$.
	Assume that $Y / X$ is faithfully flat of profinite presentation.
	Assume also the following conditions for each $i$:
		\begin{itemize}
			\item
				$Z_{i} \to Y$ is flat of profinite presentation,
			\item
				the morphism $(Z_{i})_{x} \to Y_{x}$ on the fiber over any point $x \in X$ is dominant.
		\end{itemize}
	Then $Z$ also satisfies these two conditions.
	In particular, $Z / X$ is faithfully flat of profinite presentation.
\end{Lem}

\begin{proof}
	This follows from the previous lemma.
\end{proof}

Let $L$ be a finitely generated abelian group
regarded as a constant sheaf over $k$.
To simplify the notation in the next lemma,
we denote the sheaf-Hom $\sheafhom_{k}(L, A) \in \Ab(k^{\perf}_{\pro\fppf})$ by $[L, A]$.
We have $[L, A](X) = \Hom(L, A(X))$ for any $X \in k^{\perf}$.
We have $[\Z, A] = A$.
Take an exact sequence $0 \to \Z^{m} \to \Z^{n} \to L \to 0$.
The left exactness of sheaf-Hom yields
an exact sequence $0 \to [L, A] \to A^{n} \to A^{m}$.
Hence $[L, A] \in \Pro \Alg / k$.
If $a \in L$, we denote by $\tilde{a}$ the homomorphism
$[L, A] \to A$ given by evaluation at $a$.

\begin{Lem} \label{lem: generic splitting}
	Let $X \in k'^{\perf}$,
	$L$ a finitely generated subgroup of $A'(X)$
	and $Y = X \times_{k} [L, A] = X \times_{k'} [L, A]'$.
	For any element $a \in L$,
	there exists $Z \in k'^{\perf}$
	and a $k'$-morphism $Z \to Y$ satisfying the two conditions in Lemma \ref{lem: making fpqc covers}
	such that the natural images $\tilde{a}, a - \tilde{a} \in A'(Z)$ are contained in the subset
	$(h^{\ast\set} A')(Z)$.
	If $a \in (h^{\ast\set} A')(X)$,
	then we can take $Z$ so that
		$
				(a - \tilde{a}, \tilde{a})
			\in
				(h^{\ast\set} (A'^{2}))(Z)
		$.
\end{Lem}

\begin{proof}
	Consider the following commutative diagram with a cartesian square:
		\[
			\begin{CD}
					\tilde{a}^{-1}(\xi_{\Im(\tilde{a})'})
				@>> \tilde{a} >
					\xi_{\Im(\tilde{a})'}
				\\
				@VV \text{incl} V
				@V \text{incl} VV
				\\
					[L, A]'
				@> \tilde{a} >>
					\Im(\tilde{a})'
				@> \text{incl} >>
					A',
			\end{CD}
		\]
	where $\Im(\tilde{a})' = \Im(\tilde{a}) \times_{k} k'$
	and $\xi_{\Im(\tilde{a})'}$ its generic point
	(Definition \ref{def: points in profinite setting}).
	Note that $\Im(\tilde{a}) \in \Pro \Alg / k$
	and hence $\Im(\tilde{a})'$ satisfies the assumption of
	Definition \ref{def: points in profinite setting} \eqref{def: generic point}.
	The bottom arrow in the square is faithfully flat of profinite presentation
	since it is a surjection of proalgebraic groups base-changed to $k'$.
	The right arrow is dominant flat of profinite presentation.
	Hence the left arrow is dominant flat of profinite presentation
	by Lemma \ref{lem: dominant is stable under flat base change}.
	We define $Z_{1} = X \times_{k'} \tilde{a}^{-1}(\xi_{\Im(\tilde{a})})$.
	Then the natural morphism $Z_{1} \to Y$ satisfies
	the two conditions in Lemma \ref{lem: making fpqc covers}
	by Lemma \ref{lem: dominant is stable under flat base change}.
	The natural image $\tilde{a} \in A'(Z_{1})$ is a morphism
	$Z_{1} \to A'$ that factors through
		$
				\xi_{\Im(\tilde{a})'}
			\subset
				h^{\ast\set} A'
		$.
	Hence $\tilde{a} \in (h^{\ast\set} A')(Z_{1})$.
	
	The natural inclusion $L \subset A'(X)$ defines a morphism
	$\iota \colon X \to [L, A]'$.
	We have an automorphism of the $X$-scheme $Y = X \times_{k'} [L, A]'$ given by
	$(x, \varphi) \leftrightarrow (x, \iota(x) - \varphi)$.
	The composite of this with the morphism $\tilde{a} \colon Y \to A'$
	is $a - \tilde{a}$.
	We define $Z_{2} \to Y$ to be the inverse image of
	the morphism $Z_{1} \to Y$ by this $X$-automorphism of $Y$.
	Then we have $a - \tilde{a} \in (h^{\ast\set} A')(Z_{2})$
	and $Z_{2}$ satisfies the two conditions in Lemma \ref{lem: making fpqc covers}.
	We define $Z = Z_{1} \times_{Y} Z_{2}$.
	Then we have $\tilde{a}, a - \tilde{a} \in (h^{\ast\set} A')(Z)$
	and $Z$ satisfies the two conditions in Lemma \ref{lem: making fpqc covers}.
	
	Next assume that $a \in (h^{\ast\set} A')(X)$.
	Consider the automorphism $(b, c) \leftrightarrow (b + c, c)$ of the group $A'^{2}$,
	which maps $(a - \tilde{a}, \tilde{a})$ to $(a, \tilde{a})$.
	Hence it is enough to show that we can take $Z$ so that
	$(a, \tilde{a}) \in (h^{\ast\set} (A'^{2}))(Z)$.
	We identify $a \colon X \to h^{\ast\set} A'$ with its image,
	which is an object of $k^{\ind\rat} / k'$,
	so that we have a faithfully flat morphism $a \colon X \onto a$
	of profinite presentation.
	Consider the following commutative diagram with a cartesian square:
		\[
			\begin{CD}
					(a, \tilde{a})^{-1}(\xi_{a \times_{k'} \Im(\tilde{a})'})
				@>> (a, \tilde{a}) >
					\xi_{a \times_{k'} \Im(\tilde{a})'}
				\\
				@VV \text{incl} V
				@V \text{incl} VV
				\\
					X \times_{k'} [L, A]'
				@> (a, \tilde{a}) >>
					a \times_{k'} \Im(\tilde{a})'
				@> \text{incl} >>
					A'^{2}
			\end{CD}
		\]
	Again, the generic point of $a \times_{k'} \Im(\tilde{a})' = a \times_{k} \Im(\tilde{a})$ is well-defined.
	We define $Z = (a, \tilde{a})^{-1}(\xi_{a \times_{k'} \Im(\tilde{a})'})$.
	Then $(a, \tilde{a}) \in (h^{\ast\set} (A'^{2}))(Z)$
	by the same argument as above.
	The square in the above diagram can be split into the following two cartesian squares:
		\[
			\begin{CD}
					Z
				@>> (\id, (a, \tilde{a})) >
					X \times_{a} \xi_{a \times_{k'} \Im(\tilde{a})'}
				@>> \proj_{2} \times \proj_{3} >
					\xi_{a \times_{k'} \Im(\tilde{a})'}
				\\
				@VV \text{incl} V
				@VV \proj_{1} \times \proj_{3} V
				@V \text{incl} VV
				\\
					Y
				@> (\id, \tilde{a}) >>
					X \times_{k'} \Im(\tilde{a})'
				@> (a, \id) >>
					a \times_{k'} \Im(\tilde{a})'
			\end{CD}
		\]
	The bottom two arrows are faithfully flat of profinite presentation.
	The third vertical arrow is dominant flat of profinite presentation.
	By pulling back the left square by a point of $X$
	and using Lemma \ref{lem: dominant is stable under flat base change},
	we see that the morphism $Z \to Y$ satisfies the two conditions of Lemma \ref{lem: making fpqc covers}.
\end{proof}

Let $X$, $L \subset A'(X)$, $Y = X \times_{k'} [L, A]'$
be as in the lemma.
Define two homomorphisms
	\[
			p_{0}, p_{1}
		\colon
			L
		\to
			A'(X) \oplus A'([L, A]'),
		\quad
			p_{0}(a)
		=
			a - \tilde{a},
		\quad
			p_{1}(a)
		=
			\tilde{a}.
	\]
The composites of them with the natural map
$A'(X) \oplus A'([L, A]') \to A'(X \times_{k'} [L, A]') = A'(Y)$
are denoted by the same letters $p_{0}, p_{1}$.
Their sum is the inclusion
$L \subset A'(X) \subset A'(Y)$.
Let $F \colon \Ab \to \mathcal{A}$ be one of the functors
$\Z[\;]$, $Q'$, $Q$, $M'$ or $M$.
By functoriality, we have homomorphisms
	\begin{gather*}
			p_{0}, p_{1}, \incl
		\colon
			F(L)
		\to
			F(A'(Y))
		=
			F^{\pre}(A')(Y)
		\to
			F(A')(Y),
		\\
			(p_{0}, p_{1})
		\colon
			F(L)
		\to
			F(A'(Y)^{2})
		=
			F^{\pre}(A'^{2})(Y)
		\to
			F(A'^{2})(Y).
	\end{gather*}
Define
	\begin{gather*}
			T
		\colon
			F(L)
		\to
			F(A')(Y),
		\quad
			T t
		=
			p_{0}(t) + p_{1}(t),
		\\
			V
		\colon
			F(L)
		\to
			F(A'^{2})(Y),
		\quad
			V t
		=
			(p_{0}(t), p_{1}(t)).
	\end{gather*}
Note that $T \ne \incl$ since $F$ is non-additive.
For $F = Q$ or $M$, let $F' = Q'$ or $M'$, respectively.
For each $n \ge 0$, the identification $\Z[(A'(Y)^{2})^{2^{n}}] \cong \Z[A'(Y)^{2^{n + 1}}]$ defines identifications
$F'_{n}(A'(Y)^{2}) \cong F'_{n + 1}(A'(Y))$
and $F'_{n}(A'^{2}) \cong F'_{n + 1}(A')$,
where $F'_{n}$ is the degree $n$ part of $F'$.
We have a commutative diagram
	\[
		\begin{CD}
				F'_{n}(L)
			@>> V >
				F'_{n}(A'^{2})(Y) \cong F'_{n + 1}(A')(Y)
			\\
			@VVV
			@VVV
			\\
				F_{n}(L)
			@> V >>
				F_{n + 1}(A')(Y),
		\end{CD}
	\]
where: the upper $V$ is defined above;
the lower $V$ is the splitting homotopy $F_{n}(L) \to F_{n + 1}(A'(Y))$
recalled in Section \ref{sec: review of Mac Lane's resolution}
followed by the natural homomorphism $F_{n + 1}(A'(Y)) \to F_{n + 1}(A')(Y)$;
and the vertical morphisms are the natural quotient maps.
In this sense, the above $V$ is compatible with the splitting homotopy.
By the splitting homotopy condition in Section \ref{sec: review of Mac Lane's resolution},
we have
	\begin{equation} \label{eq: splitting homotopy}
			- \boundary V - V \boundary
		=
			\incl - T
		\colon
			F(L)
		\to
			F(A')(Y).
	\end{equation}

\begin{Lem} \label{lem: generic homotopy}
	In the setting of the previous lemma and for $F = \Z[\;]$, $Q'$, $Q$, $M'$ or $M$,
	let $t \in F(L)$.
	Then there exists $Z \in k'^{\perf}$
	and a $k'$-morphism $Z \to Y$ with the composite $Z \to Y \to X$ faithfully flat of profinite presentation
	such that the natural image $T t \in F(A')(Z)$ is contained in the subgroup
	$(h^{\ast} F(A'))(Z)$.
	If $t \in (h^{\ast} F(A'))(X)$,
	then we can take $Z$ so that
	$V t \in (h^{\ast} F(A'^{2}))(Z)$.
\end{Lem}

\begin{proof}
	We only need to show this for $F = \Z[\;]$
	in view of the structures of homogeneous parts of $Q', Q, M', M$
	recalled in Section \ref{sec: review of Mac Lane's resolution}.
	Write $t = \sum_{i = 1}^{n} m_{i} (a_{i})$,
	$m_{i} \in \Z$, $a_{i} \in L$,
	where $(a_{i})$ is the image of $a_{i}$ in $\Z[L]$.
	For any $i$, take $Z_{i}$ corresponding to $a_{i} \in L$ given in the previous lemma.
	Then $T (a_{i}) \in \Z[h^{\set\ast} A'](Z_{i}) = (h^{\ast} \Z[A'])(Z_{i})$.
	Let $Z = Z_{1} \times_{Y} \dots \times_{Y} Z_{n}$.
	Then $Z / X$ is faithfully flat of profinite presentation by Lemma \ref{lem: making fpqc covers},
	and we have $T t = \sum m_{i} T (a_{i}) \in (h^{\ast} \Z[A'])(Z)$.
	The statement for $V$ is similar.
\end{proof}

\begin{proof}[Proof of Proposition \ref{prop: generic resolution}]
	Let $F = Q$ or $M$.
	We want to show that 
		\[
				H_{n}(h^{\ast} F(A')) \isomto H_{n}(F(A'))
			\quad \text{in} \quad
				\Ab(k'^{\perf}_{\pro\fppf})
		\]
	for all $n \ge 0$.
	We want to show that the inverse is given by $T$
	(with pro-fppf locally defined chain homotopy to the identity given by $V$).
	To give a short and rigorous proof,
	we avoid doing this directly but as follows.
	
	We first treat the injectivity.
	Let $F' = Q'$ or $M'$ if $F = Q$ or $M$, respectively.
	Let $t \in (h^{\ast} F_{n}(A'))(X)$ be an element such that
	$t = \boundary s$ for some $s \in F_{n + 1}(A'(X))$.
	(We are ignoring the difference between $F'_{n}(A'(X)) = F_{n}^{\prime \pre}(A')(X)$ and $F'_{n}(A')(X)$
	since an element of the latter comes from an element of the former
	after taking a pro-fppf cover of $X$
	and it is enough to argue pro-fppf locally on $X$.)
	Since $F'_{n}(A') \onto F_{n}(A')$ is a split surjection,
	we can take a lift $t' \in (h^{\ast} F'_{n}(A'))(X)$ of $t$.
	Take a finitely generated subgroup $L$ of $A'(X)$ large enough so that
	$t' \in F'_{n}(L)$ and $s \in F_{n + 1}(L)$.
	Then by the above lemma,
	there exists $Z \in k'^{\perf}$ faithfully flat of profinite presentation over $X$
	such that $V t' \in (h^{\ast} F'_{n}(A'^{2}))(Z)$ and
	$T s \in (h^{\ast} F_{n + 1}(A'))(Z)$.
	Hence $V t \in (h^{\ast} F_{n + 1}(A'))(Z)$.
	A simple computation using the splitting homotopy condition \eqref{eq: splitting homotopy}
	shows that $t = \boundary(T s - V t)$ in $(h^{\ast} F_{n}(A'))(Z)$.
	This shows the injectivity.
	
	Next we show the surjectivity.
	Let $t \in F_{n}(A'(X))$ be an element with $\boundary t = 0$.
	Take a finitely generated subgroup $L$ of $A'(X)$ large enough so that
	$t \in F_{n}(L)$.
	Then by the above lemma,
	there exists $Z \in k'^{\perf}$ faithfully flat of profinite presentation over $X$
	such that $T t \in (h^{\ast} F_{n}(A'))(Z)$.
	Then
		\[
				t
			=
				T t - \boundary V t - V \boundary t
			=
				T t - \boundary V t.
		\]
	This shows the surjectivity.
\end{proof}

\begin{Rmk} \label{rmk: stable homology} \BetweenThmAndList
	\begin{enumerate}
	\item
		Proposition \ref{prop: generic resolution} is not true
		if we use the fppf topology instead of the pro-fppf topology.
		Namely, let $h_{0} \colon \Spec k'^{\perf}_{\fppf} \to \Spec k^{\ind\rat}_{\et} / k'$
		be the continuous map defined by the identity.
		Then we can show that $H_{1}(h_{0}^{\ast}Q(\Ga)) \ne 0$,
		even though $H_{1}(Q(\Ga))$, which is the first stable homology (see below), vanishes.
		This difference comes from the fact that the inclusion
		$\xi_{\Im(\tilde{a})'} \into \Im(\tilde{a})'$
		used in the proof of Lemma \ref{lem: generic splitting}
		is flat of profinite presentation but not of finite presentation,
		even if $A \in \Alg / k$ and $k' = k$.
	\item
		The groups $H_{n}(Q(A'))$ are the integral stable homology groups of
		the Eilenberg-Mac Lane spectrum $H A'$
		(see \cite{Pir96} for a simple proof of this fact).
		Therefore the part of Proposition \ref{prop: generic resolution} for $Q$ says that
		the complex $h^{\ast} Q(A')$ correctly calculates these groups.
	\end{enumerate}
\end{Rmk}


\subsection{Ext for the rational \'etale and perfect \'etale sites}
\label{sec: comparison of Ext groups for the rational etale and perfect etale sites}

For any perfect affine $k$-scheme $X$, the localization $\Spec k^{\perf}_{\et} / X$
of the perfect \'etale site $\Spec k^{\perf}_{\et}$ at $X$
is also denoted by $X^{\perf}_{\et}$.

\begin{Prop} \label{prop: Ext for the rational etale and perfect etale sites}
	Let $A \in \Pro \Alg / k$, $B \in \Loc \Alg / k$ and $k' \in k^{\ind\rat}$.
	Then we have
		\[
				R \Hom_{k^{\ind\rat}_{\et} / k'}(A, B)
			=
				R \Hom_{k'^{\perf}_{\et}}(A, B).
		\]
	If $k'$ is a field, this is further isomorphic to
	$R \Hom_{k'^{\ind\rat}_{\et}}(A, B)$.
	If $A \in \Alg / k$ and $k' = k$, then the same equality is true
	with $k^{\ind\rat}$ replaced by $k^{\rat}$.
\end{Prop}

We prove this below.
Throughout this subsection,
let $k' \in k^{\ind\rat}$ and
let $h \colon \Spec k'^{\perf}_{\pro\fppf} \to \Spec k^{\ind\rat}_{\et} / k'$ be
the continuous map defined by the identity.
We first show that $h^{\ast}$ admits a left derived functor.
We know that $h_{\ast}$ sends acyclic sheaves to acyclic sheaves
(see the first paragraph of Section \ref{sec: The ind-rational etale site}).

\begin{Lem}
	Let $h \colon S' \to S$ be a continuous map of sites.
	Assume that $h_{\ast}$ sends acyclic sheaves to acyclic sheaves.
	Then the pullback functor $h^{\ast} \colon \Ab(S) \to \Ab(S')$
	admits a left derived functor $L h^{\ast} \colon D(S) \to D(S')$,
	which is left adjoint to $R h_{\ast} \colon D(S') \to D(S)$.
	We have $L_{n} h^{\ast} \Z[X] = 0$ for any object $X$ of $S$ and $n \ge 1$.
\end{Lem}

\begin{proof}
	By \cite[Theorem 14.4.5]{KS06},
	it is enough to show the existence of an $h^{\ast}$-projective full subcategory of $\Ab(S)$
	in the sense of \cite[Definition 13.3.4]{KS06}
	that contains sheaves of the form $\Z[X]$.
	Consider the following condition for a sheaf $P \in \Ab(S)$:
		\begin{equation} \label{eq: condition to derive the pullback}
				\Ext_{S}^{n}(P, I)
			=
				0
		\end{equation}
	for any acyclic sheaf $I \in \Ab(S)$ and $n \ge 1$.
	The sheaves satisfying this condition form a full subcategory of $\Ab(S)$,
	which contains sheaves of the form $\Z[X]$.
	We want to show that this full subcategory is $h^{\ast}$-projective.
	We check the conditions of the dual of \cite[Corollary 13.3.8]{KS06}
	(which is a standard criterion for the existence of derived functors).
	
	First, every object of $\Ab(S)$ is
	a quotient of a direct sum of objects of the form $P = \Z[X]$.
	Next, if $0 \to P \to Q \to R \to 0$ is an exact sequence in $\Ab(S)$
	with $Q$ and $R$ satisfying \eqref{eq: condition to derive the pullback},
	then $P$ also satisfies the same condition.
	Moreover, the homomorphism $\Hom_{S'}(h^{\ast} Q, I') \to \Hom_{S'}(h^{\ast} P, I')$
	for any acyclic sheaf $I' \in \Ab(S')$
	is identified with $\Hom_{S}(Q, h_{\ast} I') \to \Hom_{S}(P, h_{\ast} I')$,
	which is surjective by \eqref{eq: condition to derive the pullback}.
	Taking $I'$ to be an injective sheaf containing $h^{\ast} P$,
	we know that $0 \to h^{\ast} P \to h^{\ast} Q \to h^{\ast} R \to 0$ is exact.
	Hence we have verified the required conditions.
\end{proof}

The purpose of the previous two subsections
was to prove the following proposition.
This proposition is highly non-trivial
since $h^{\ast\set} \colon \Set(k^{\ind\rat}_{\et} / k') \to \Set(k'^{\perf}_{\pro\fppf})$
is not exact (Proposition \ref{prop: pullback is not exact}) and
we do not know whether
$h^{\ast} \colon \Ab(k^{\ind\rat}_{\et} / k') \to \Ab(k'^{\perf}_{\pro\fppf})$
is exact or not.
It is crucial to use Mac Lane's resolution here.

\begin{Prop} \label{lem: inj to Hom-acyclic}
	We have $L h^{\ast} A' = A'$ for any $A \in \Pro \Alg / k$,
	where $A' = A \times_{k} k'$.
\end{Prop}

\begin{proof}
	Let $M(A')$ be Mac Lane's resolution of $A'$ in $\Ab(k^{\ind\rat}_{\et} / k')$.
	By Proposition \ref{prop: generic resolution},
	we have $h^{\ast} M(A') = A'$.
	For each $n \ge 0$, the sheaf $h^{\ast} M_{n}(A')$ is a direct summand of
	a filtered union of a direct sum of sheaves of the form
	$\Z[\Spec k'']$ for $k'' \in k^{\ind\rat} / k'$.
	Hence
	$L h^{\ast} M(A') = h^{\ast} M(A') = A'$.
\end{proof}

\begin{Prop} \label{prop: RHom for profppf, perf etale and indrat etale}
	Let $f \colon \Spec k'^{\perf}_{\pro\fppf} \to \Spec k'^{\perf}_{\et}$ and
	$g \colon \Spec k'^{\perf}_{\et} \to \Spec k^{\ind\rat}_{\et} / k'$ be
	the continuous maps defined by the identities.
	Let $A \in \Pro \Alg / k$ and $B \in \Ab(k'^{\perf}_{\pro\fppf})$.
	Then we have
		\[
				R \Hom_{k'^{\perf}_{\pro\fppf}}(A, B)
			=
				R \Hom_{k'^{\perf}_{\et}}(A, R f_{\ast} B)
			=
				R \Hom_{k^{\ind\rat}_{\et} / k'}(A, g_{\ast} R f_{\ast} B).
		\]
	(Here we are using the same symbol $A$ to mean $A' = A \times_{k} k'$.)
\end{Prop}

\begin{proof}
	The first equality is clear
	since $f$ is a morphism of sites.
	For the second, note that $g_{\ast} R f_{\ast} = R h_{\ast}$
	since $g_{\ast}$ is exact.
	The previous lemma shows that
		\[
				R \Hom_{k'^{\perf}_{\pro\fppf}}(A', B)
			=
				R \Hom_{k'^{\perf}_{\pro\fppf}}(L h^{\ast} A', B)
			=
				R \Hom_{k^{\ind\rat}_{\et} / k'}(A', R h_{\ast} B).
		\]
\end{proof}

\begin{proof}[Proof of Proposition \ref{prop: Ext for the rational etale and perfect etale sites}]
	Let $f \colon \Spec k'^{\perf}_{\pro\fppf} \to \Spec k'^{\perf}_{\et}$
	be the morphism defined by the identity.
	Let $A \in \Pro \Alg / k$ and $B \in \Loc \Alg / k$.
	We have $R f_{\ast} B = B$
	by Corollary \ref{cor: profppf and etale cohomology of an algebraic group}.
	The proposition above then gives the first statement of the proposition.
	
	For the second statement,
	if $k'$ is a field, then the first statement with $k$ replaced by $k'$ implies that
		\[
				R \Hom_{k'^{\ind\rat}_{\et}}(A, B)
			=
				R \Hom_{k'^{\perf}_{\et}}(A, B).
		\]
	This gives the second statement.
	
	For the third,
	assume that $A \in \Alg / k$.
	It is enough to show that
		\[
				R \Hom_{k^{\ind\rat}_{\et}}(A, B)
			=
				R \Hom_{k^{\rat}_{\et}}(A, B).
		\]
	Let $\alpha \colon \Spec k^{\ind\rat}_{\et} \to \Spec k^{\rat}_{\et}$
	be the morphism of sites defined by the identity.
	Clearly $\alpha_{\ast}$ is exact and $\alpha_{\ast} \alpha^{\ast} = \id$.
	Note that $A$ is the disjoint union of its points $x \in A$ as sheaves of sets
	and these points are in $k^{\rat}$.
	Hence $\alpha^{\ast} A = A$ as sheaves of abelian groups.
	This implies the above equality.
\end{proof}


\subsection{Algebraic groups as sheaves on the perfect \'etale site}
\label{sec: Ext as sheaves and as algebraic groups}

The following proposition finishes the proof of Theorem \ref{thm: comparison of Ext, proalgebraic setting}
and hence Theorem \ref{thm: main theorem, comparison of Ext}.

\begin{Prop} \label{prop: Ext as sheaves and as algebraic groups}
	Let $A \in \Pro \Alg / k$, $B \in \Loc \Alg / k$ and
	$k' = \bigcup_{\nu} k'_{\nu} \in k^{\ind\rat}$ with $k'_{\nu} \in k^{\rat}$.
	Then for any $n \ge 0$, we have
		\[
				\Ext_{k'^{\perf}_{\et}}^{n}(A, B)
			=
				\dirlim_{\nu}
					\Ext_{(k'_{\nu})^{\perf}_{\et}}^{n}(A, B).
		\]
	If $B \in \Alg / k$, then
		\begin{gather*}
					\Ext_{k^{\perf}_{\et}}^{n}(A, B)
				=
					\Ext_{\Pro \Alg / k}^{n}(A, B),
			\\
					\Ext_{k^{\perf}_{\et}}^{n}(A, \Q)
				=
					\Hom_{k^{\perf}_{\et}}(A, \Z)
				=
					0,
			\\
					\Ext_{k^{\perf}_{\et}}^{n + 1}(A, \Z)
				=
					\Ext_{k^{\perf}_{\et}}^{n}(A, \Q / \Z)
				=
					\dirlim_{m}
						\Ext_{\Pro \Alg / k}^{n}(A, \Z / m \Z),
		\end{gather*}
\end{Prop}

This proposition is a consequence of results of Breen \cite{Bre70}, \cite{Bre81}.
We make this clear below.
We first need a lemma on limit arguments.

\begin{Lem} \label{lem: lim and Ext commutes}
	Let $A = \invlim_{\lambda} A_{\lambda} \in \Pro \Alg / k$ with $A_{\lambda} \in \Alg / k$.
	Let $\{R_{\nu}\}$ be a filtered direct system in $k^{\perf}$ with limit $R$.
	\begin{enumerate}
		\item \label{ass: Ext is continuous on A and k}
			If $B \in \Loc \Alg / k$,
			then we have
				\[
						\Ext_{R^{\perf}_{\et}}^{i}(A, B)
					=
						\dirlim_{\lambda, \nu}
						\Ext_{(R_{\nu})^{\perf}_{\et}}^{i}(A_{\lambda}, B)
				\]
			for all $i \ge 0$.
		\item \label{ass: Ext is continuous on B}
			The functor $\Ext_{R^{\perf}_{\et}}^{i}(A, \;\cdot\;)$ on $\Ab(R^{\perf}_{\et})$
			commutes with filtered direct limits.
	\end{enumerate}
\end{Lem}

We prove this lemma by reducing it to the corresponding facts for cohomology:
	\[
			H^{i}(A \times_{k} R, B)
		=
			\dirlim_{\lambda, \nu}
				H^{i}(A_{\lambda} \times_{k} R_{\nu}, B)
	\]
and the functor $H^{i}(A \times_{k} R, \;\cdot\;)$ on $\Ab(R^{\perf}_{\et})$
commutes with filtered direct limits.
The first fact is given in \cite[III, Lemma 1.16, Remark 1.17 (a)]{Mil80}
and second in \cite[III, Remark 3.6 (d)]{Mil80}.
It is important here that $A \times_{k} R$ is quasi-compact and $B$ is locally of finite presentation.
We do this reduction by describing Ext groups by cohomology groups
using the cubical construction and its homology.
(Mac Lane's resolution does not behave well with respect to limits
since the graded abelian group $\mathcal{B}$ in Section \ref{sec: review of Mac Lane's resolution}
is not finite rank in each degree.)

\begin{proof}
	It is enough to show the two statements
	in the case the system $\{R_{\nu}\}_{\nu}$ is constant
	(that is, $R_{\nu} = R$ for all $\nu$).
	Indeed, let $f \colon \Spec R^{\perf}_{\et} \to \Spec k^{\perf}_{\et}$
	and $f_{\nu} \colon \Spec (R_{\nu})^{\perf}_{\et} \to \Spec k^{\perf}_{\et}$
	be the natural morphisms.
	Then for $B \in \Loc \Alg / k$ and $j \ge 0$,
	the sheaf $R^{j} f_{\ast} B \in \Ab(k^{\perf}_{\et})$ is the sheafification of the presheaf
	$R' \mapsto H^{j}(R' \tensor_{k} R, B)$
	and the sheaf $\dirlim_{\nu} R^{j} f_{\nu \ast} B$ is the sheafification of the presheaf
	$R' \mapsto \dirlim_{\nu} H^{j}(R' \tensor_{k} R_{\nu}, B)$.
	These sheaves are isomorphic by \cite[III, Lemma 1.16, Remark 1.17 (a)]{Mil80}.
	Therefore we have
		\[
				\dirlim_{\nu}
				\Ext_{k^{\perf}_{\et}}^{i}(A, R^{j} f_{\nu \ast} B)
			=
				\Ext_{k^{\perf}_{\et}}^{i}(A, R^{j} f_{\ast} B)
		\]
	for any $i, j \ge 0$ if we use \eqref{ass: Ext is continuous on B}.
	Since the pushforward functors $f_{\nu \ast}$ and $f_{\ast}$ send injectives to injectives,
	we have Grothendieck spectral sequences
		\begin{gather*}
					E_{2}^{i j}
				=
					\dirlim_{\nu}
					\Ext_{k^{\perf}_{\et}}^{i}(A, R^{j} f_{\nu \ast} B)
				\Longrightarrow
					\dirlim_{\nu}
					\Ext_{(R_{\nu})^{\perf}_{\et}}^{i + j}(A, B),
			\\
					E_{2}^{i j}
				=
					\Ext_{k^{\perf}_{\et}}^{i}(A, R^{j} f_{\ast} B)
				\Longrightarrow
					\Ext_{R^{\perf}_{\et}}^{i + j}(A, B).
		\end{gather*}
	The isomorphisms between the $E_{2}$-terms induce isomorphisms between the $E_{\infty}$-terms
	since the spectral sequences come from the commutative diagram
		\[
			\begin{CD}
					\dirlim_{\nu}
					\Hom_{k^{\perf}_{\et}}(A, f_{\nu \ast} J)
				@=
					\dirlim_{\nu}
					\Hom_{(R_{\nu})^{\perf}_{\et}}(A, J)
				\\
				@VVV
				@VVV
				\\
					\Hom_{k^{\perf}_{\et}}(A, f_{\ast} J)
				@=
					\Hom_{R^{\perf}_{\et}}(A, J)
			\end{CD}
		\]
	on the level of complexes,
	where $J$ is an injective resolution of $B$ in $\Ab(k^{\perf}_{\et})$.
	(Here $J$ is pulled back to $R_{\nu}$ and $R$, which remains injective
	since the pullback functors $f_{\nu}^{\ast}$ and $f^{\ast}$ send injectives to injectives
	by \cite[III, Lemma 1.11]{Mil80}.)
	Thus
		\[
				\dirlim_{\nu}
				\Ext_{(R_{\nu})^{\perf}_{\et}}^{i}(A, B)
			=
				\Ext_{R^{\perf}_{\et}}^{i}(A, B).
		\]
	Therefore if we have shown
		\[
				\dirlim_{\lambda}
				\Ext_{(R_{\nu})^{\perf}_{\et}}^{i}(A_{\lambda}, B)
			=
				\Ext_{(R_{\nu})^{\perf}_{\et}}^{i}(A, B)
		\]
	for each fixed $\nu$, the general case of \eqref{ass: Ext is continuous on A and k} follows.
	
	\eqref{ass: Ext is continuous on A and k}
	when $R_{\nu} = R$ for all $\nu$:
	We need several reduction steps.
	First, it is enough to show the statement with
	the \'etale topology replaced by the pro-fppf topology:
		\[
				\dirlim_{\lambda}
				\Ext_{R^{\perf}_{\pro\fppf}}^{i}(A_{\lambda}, B)
			=
				\Ext_{R^{\perf}_{\pro\fppf}}^{i}(A, B)
		\]
	for $B \in \Loc \Alg / k$ and $i \ge 0$.
	Indeed, let $f \colon \Spec R^{\perf}_{\pro\fppf} \to \Spec R^{\perf}_{\et}$ be
	the morphism defined by the identity.
	Then we have $R f_{\ast} B = B$ by Corollary \ref{cor: profppf and etale cohomology of an algebraic group}.
	Hence
		\[
				R \Hom_{R^{\perf}_{\pro\fppf}}(A, B)
			=
				R \Hom_{R^{\perf}_{\et}}(A, R f_{\ast} B)
			=
				R \Hom_{R^{\perf}_{\et}}(A, B).
		\]
	The same is true with $A$ replaced by $A_{\lambda}$.
	Hence
		\[
				\Ext_{R^{\perf}_{\pro\fppf}}^{i}(A, B)
			=
				\Ext_{R^{\perf}_{\et}}^{i}(A, B),
			\quad
				\Ext_{R^{\perf}_{\pro\fppf}}^{i}(A_{\lambda}, B)
			=
				\Ext_{R^{\perf}_{\et}}^{i}(A_{\lambda}, B).
		\]
	
	We will prove a slightly more general statement:
		\[
				\dirlim_{\lambda}
				\Ext_{R^{\perf}_{\pro\fppf}}^{i}(A_{\lambda}, B)
			=
				\Ext_{R^{\perf}_{\pro\fppf}}^{i}(A, B)
		\]
	for any sheaf $B \in \Ab(R^{\perf}_{\pro\fppf})$
	that commutes with filtered direct limits as a functor on $R^{\perf}$.
	For this form of the statement, it is actually enough to treat the case $R = k$.
	Indeed, let $f \colon \Spec R^{\perf}_{\pro\fppf} \to \Spec k^{\perf}_{\pro\fppf}$ be the natural morphism.
	As before, we have Grothendieck spectral sequences
		\begin{gather*}
					E_{2}^{i j}
				=
					\dirlim_{\lambda}
					\Ext_{k^{\perf}_{\pro\fppf}}^{i}(A_{\lambda}, R^{j} f_{\ast} B)
				\Longrightarrow
					\dirlim_{\lambda}
					\Ext_{R^{\perf}_{\pro\fppf}}^{i + j}(A_{\lambda}, B),
			\\
					E_{2}^{i j}
				=
					\Ext_{k^{\perf}_{\et}}^{i}(A, R^{j} f_{\ast} B)
				\Longrightarrow
					\Ext_{R^{\perf}_{\et}}^{i + j}(A, B)
		\end{gather*}
	and a morphism between them compatible with the $E_{\infty}$-terms.
	By Proposition \ref{prop: pro-fppf and fppf cohomology},
	the sheaf $R^{j} f_{\ast} B \in \Ab(k^{\perf}_{\pro\fppf})$ is the pro-fppf sheafification of the presheaf
	$R' \mapsto H^{j}((R' \tensor_{k} R)_{\fppf}, B)$.
	The presheaf appearing here commutes with filtered direct limits as a functor on $k^{\perf}$
	by \cite[III, Lemma 1.16, Remark 1.17 (d)]{Mil80}.
	Hence the fppf sheafification of this presheaf is already a pro-fppf sheaf,
	and $R^{j} f_{\ast} B$ commutes with filtered direct limits as a functor on $k^{\perf}$.
	Hence we may replace $R$ by $k$ and $B$ by $R^{j} f_{\ast} B$.
	Thus we are reduced to showing that
		\[
				\dirlim_{\lambda}
				\Ext_{k^{\perf}_{\pro\fppf}}^{i}(A_{\lambda}, B)
			=
				\Ext_{k^{\perf}_{\pro\fppf}}^{i}(A, B)
		\]
	for any $B \in \Ab(k^{\perf}_{\pro\fppf})$
	that commutes with filtered direct limits as a functor on $k^{\perf}$.
	
	Note that the functor $\Pro \Alg / k \to \Ab(k^{\perf}_{\pro\fppf})$ is exact
	since a surjection in $\Pro \Alg / k$ is a faithfully flat morphism of profinite presentation
	and hence a surjection in $\Ab(k^{\perf}_{\pro\fppf})$.
	For each $\lambda$, let $A_{\lambda}', A_{\lambda}'', A_{\lambda}'''$ be
	the kernel, cokernel, image of $A \to A_{\lambda}$.
	If we know that
		\[
				\dirlim_{\lambda}
					\Ext_{k^{\perf}_{\pro\fppf}}^{i}(A_{\lambda}', B)
			=
				\dirlim_{\lambda}
					\Ext_{k^{\perf}_{\pro\fppf}}^{i}(A_{\lambda}'', B)
			=
				0
		\]
	for any $i \ge 0$,
	then
		\[
				\Ext_{k^{\perf}_{\pro\fppf}}^{i}(A, B)
			=
				\dirlim_{\lambda}
					\Ext_{k^{\perf}_{\pro\fppf}}^{i}(A_{\lambda}''', B)
			=
				\dirlim_{\lambda}
					\Ext_{k^{\perf}_{\pro\fppf}}^{i}(A_{\lambda}, B)
		\]
	for any $i \ge 0$.
	Therefore it is enough to show the following statement:
	if $\{ A_{\lambda} \}_{\lambda}$ is a filtered inverse system of proalgebraic groups over $k$
	with $\invlim A_{\lambda} = 0$ and
	$B \in \Ab(k^{\perf}_{\pro\fppf})$ commutes with filtered direct limits,
	then
		\[
				\dirlim_{\lambda}
					\Ext_{k^{\perf}_{\pro\fppf}}^{i}(A_{\lambda}, B)
			=
				0
		\]
	for any $i \ge 0$.
	
	We prove this statement.
	We denote by $K^{+}(k^{\perf}_{\pro\fppf})$ and $K^{+}(\Ab)$
	the homotopy categories of bounded below complexes
	in $\Ab(k^{\perf}_{\pro\fppf})$ and $\Ab$, respectively.
	For a filtered direct system $\{D_{\lambda}\}$
	of bounded above complexes in $\Ab(k^{\perf}_{\pro\fppf})$,
	consider the triangulated functor
		\[
				K^{+}(k^{\perf}_{\pro\fppf})
			\to
				K^{+}(\Ab),
			\quad
				C
			\mapsto
				\dirlim_{\lambda}
					\Hom_{k^{\perf}_{\pro\fppf}}(D_{\lambda}, C),
		\]
	where we denote by $\Hom_{k^{\perf}_{\pro\fppf}}(D_{\lambda}, C)$
	the total complex of the Hom double complex
	and take its term-wise direct limit in $\lambda$.
	This admits a right derived functor
		\[
				D^{+}(k^{\perf}_{\pro\fppf})
			\to
				D^{+}(\Ab),
			\quad
				C
			\mapsto
				\dirlim_{\lambda}
					R \Hom_{k^{\perf}_{\pro\fppf}}(D_{\lambda}, C)
		\]
	since $\Ab(k^{\perf}_{\pro\fppf})$ has enough injectives and
	by \cite[Proposition 13.2.3]{KS06}.
	For any $i$, its $i$-th cohomology is
	$\dirlim_{\lambda} H^{i} R \Hom_{k^{\perf}_{\pro\fppf}}(D_{\lambda}, C)$.
	For $n \ge 0$, let
		\begin{gather*}
					X^{n}
				=
					\dirlim_{\lambda}
						R \Hom_{k^{\perf}_{\pro\fppf}}(H_{n} Q(A_{\lambda}), B),
			\\
					Y^{n}
				=
					\dirlim_{\lambda}
						R \Hom_{k^{\perf}_{\pro\fppf}}(\tau_{\le -n} Q(A_{\lambda}), B),
		\end{gather*}
	where $\tau$ denotes truncation (in cohomological grading).
	We have a diagram
		\[
			\begin{CD}
					X^{0}
				@.
					X^{1}[-1]
				@.
					X^{2}[-2]
				@.
					X^{3}[-3]
				@.
					\cdots
				\\
				@VVV
				@VVV
				@VVV
				@VVV
				\\
					Y^{0}
				@>>>
					Y^{1}
				@>>>
					Y^{2}
				@>>>
					Y^{3}
				@>>>
					\cdots
			\end{CD}
		\]
	in $D^{+}(\Ab)$,
	where each L shape triangle $X^{n}[-n] \to Y^{n} \to Y^{n + 1}$ is distinguished.
	The homology groups $H_{n}(Q(A))$ are the stable homology of $A$
	(Remark \ref{rmk: stable homology}).
	We use the following fact (\cite[Theorem 3]{Bre70}):
	$H_{0}(Q(A)) = A$, and for each $n \ge 1$,
	the group $H_{n}(Q(A))$ is a finite direct sum of
	the kernel or the cokernel of multiplication by $l$ on $A$
	for some various primes $l \ge 2$.
	(This is a finite direct sum for each fixed $n$
	since by the definition of admissible sequences \cite[(1.18)]{Bre70},
	there are finitely many $l$-admissible sequences $(a_{1}, a_{2}, \dots)$ of degree $n$
	with $a_{1} \equiv 0 \mod 2 l - 2$ for any prime $l$
	and there are no such if $2 l - 2 > n$.)
	In particular, we have
		\[
				X^{0}
			=
				\dirlim_{\lambda}
					R \Hom_{k^{\perf}_{\pro\fppf}}(A_{\lambda}, B),
		\]
	of which we want to show the vanishing.
	
	We first show that $Y^{0} = 0$.
	We have a hyperext spectral sequence
		\[
				E_{1}^{i j}
			=
				\dirlim_{\lambda}
					\Ext_{k^{\perf}_{\pro\fppf}}^{j}(Q_{i}(A_{\lambda}), B)
			\Longrightarrow
				H^{i + j} Y^{0}.
		\]
	Recall from Section \ref{sec: review of Mac Lane's resolution} that for each $i$,
	the $i$-th term $Q_{i}(A_{\lambda})$ is
	a direct summand of the sheaf $\Z[A_{\lambda}^{2^{i}}]$.
	Therefore we can write the Ext groups in the above spectral sequences
	in terms of pro-fppf cohomology groups of schemes of the form
	$A_{\lambda}^{2^{i}}$.
	Since $\invlim A_{\lambda} = 0$, we have
		\[
				\dirlim_{\lambda}
					H^{j}(A_{\lambda}^{2^{i}}, B)
			=
				H^{j}(0, B)
		\]
	by the remark before the proof
	(more precisely, by the pro-fppf version of \cite[III, Remark 1.17 (d)]{Mil80}).
	Hence we have
		$
				E_{1}^{i j}
			=
				\Ext_{k^{\perf}_{\pro\fppf}}^{j}(Q_{i}(0), B)
		$.
	Note that the zero map $A \onto 0$ induces
	a morphism on their spectral sequences of the above type
	compatible with the filtrations on the $E_{\infty}$-terms,
	since it comes from the morphism
		\[
				\dirlim_{\lambda, \nu}
					\Hom_{k^{\perf}_{\pro\fppf}}(Q(A_{\lambda}), J)
			\to
				\Hom_{k^{\perf}_{\pro\fppf}}(Q(0), J)
		\]
	of double complexes,
	where $J$ is an injective resolution of $B$.
	Hence
		\[
				Y^{0}
			=
				R \Hom_{k^{\perf}_{\pro\fppf}}(Q(0), B)
			=
				0
		\]
	since $Q(0)$ has zero homology.
	
	Now we show that $X^{0} = 0$.
	We prove that $\tau_{\le n} X^{0} = 0$ by induction on $n \ge 0$.
	The base case $n = 0$ follows from $Y^{0} = 0$.
	Assume that $\tau_{\le n} X^{0} = 0$ (for all $\{A_{\lambda}\}$ satisfying the assumption).
	Then
		\[
				\tau_{\le n} \dirlim_{\lambda}
				R \Hom_{k^{\perf}_{\pro\fppf}}(A_{\lambda}^{\sharp}, B)
			=
				0,
		\]
	where $A_{\lambda}^{\sharp} \in \Pro \Alg / k$ is the kernel or cokernel of
	multiplication by any positive integer on $A_{\lambda}$.
	Hence
		\[
				\tau_{\le n + 1}(X^{1}[-1])
			=
				\tau_{\le n + 1}(X^{2}[-2])
			=
				\dots
			=
				0
		\]
	by the structure of the homology of $Q$.
	Note that $Y^{n}$ is concentrated in degrees $\ge n$.
	Hence by the above diagram and $Y^{0} = 0$, we have
		\[
				\tau_{\le n + 1} X^{0}
			=
				\tau_{\le n} Y^{1}
			=
				\tau_{\le n} Y^{2}
			=
				\dots
			=
				\tau_{\le n} Y^{n + 1}
			=
				0.
		\]
	
	\eqref{ass: Ext is continuous on B}
	We may assume that $R = k$.
	Indeed, let $f \colon \Spec R^{\perf}_{\et} \to \Spec k^{\perf}_{\et}$
	be the natural morphism.
	Then for any filtered direct system $\{C_{\mu}\}$ in $\Ab(R^{\perf}_{\et})$ with direct limit $C$, $j \ge 0$
	and $R' \in k^{\perf}$,
	we have
		\[
				\dirlim_{\mu}
				H^{j}(R' \tensor_{k} R, C_{\mu})
			=
				H^{j}(R' \tensor_{k} R, C)
		\]
	by \cite[III, Remark 3.6 (d)]{Mil80}.
	Hence $\dirlim_{\mu} R^{j} f_{\ast} C_{\mu} = R^{j} f_{\ast} C$.
	We have Grothendieck spectral sequences
		\begin{gather*}
					E_{2}^{i j}
				=
					\dirlim_{\mu}
					\Ext_{k^{\perf}_{\et}}^{i}(A, R^{j} f_{\ast} C_{\mu})
				\Longrightarrow
					\dirlim_{\mu}
					\Ext_{R^{\perf}_{\et}}^{i + j}(A, C_{\mu}),
			\\
					E_{2}^{i j}
				=
					\Ext_{k^{\perf}_{\et}}^{i}(A, R^{j} f_{\ast} C)
				\Longrightarrow
					\Ext_{R^{\perf}_{\et}}^{i + j}(A, C).
		\end{gather*}
	There is a morphism from the first spectral sequence to the second compatible with the $E_{\infty}$-terms
	by the following reason.
	By \cite[Corollary 9.6.6]{KS06},
	there exists a functorial choice of embeddings $D \into \Psi(D)$
	for arbitrary $D \in \Ab(R^{\perf}_{\et})$ with $\Psi(D)$ injective.
	Let $\Psi^{0}(D) = \Psi(D)$, $\Psi^{1}(D) = \Psi(\Psi(D) / D)$ and so on.
	Then we have a functorial injective resolution $\Psi^{\bullet}(D)$ of $D$.
	We have a commutative diagram
		\[
			\begin{CD}
					\dirlim_{\mu}
					\Hom_{k^{\perf}_{\et}}(A, f_{\ast} \Psi^{\bullet}(C_{\mu}))
				@=
					\dirlim_{\mu}
					\Hom_{R^{\perf}_{\et}}(A, \Psi^{\bullet}(C)),
				\\
				@VVV
				@VVV
				\\
					\Hom_{k^{\perf}_{\et}}(A, f_{\ast} \Psi^{\bullet}(C))
				@=
					\Hom_{R^{\perf}_{\et}}(A, \Psi^{\bullet}(C))
			\end{CD}
		\]
	of complexes in $\Ab$.
	This induces a desired morphism of spectral sequences.
	Replacing $R$ by $k$ and $C_{\mu}$ by $R^{j} f_{\ast} C_{\mu}$,
	we are reduced to the case $R = k$.
	
	Now assume $R = k$.
	Let $M$ be a directed set viewed as a filtered category.
	Consider the category $\Ab(k^{\perf}_{\et})(M)$ of functors
	from $M$ to $\Ab(k^{\perf}_{\et})$
	(that is, direct systems in $\Ab(k^{\perf}_{\et})$ indexed by $M$).
	This is an abelian category with enough injectives (\cite[III, Remark 3.6 (d)]{Mil80}).
	The additive bifunctor
		\[
					\Ab(k^{\perf}_{\et})^{\op}
				\times
					\Ab(k^{\perf}_{\et})(M)
			\to
				\Ab,
			\quad
				(D, \{C_{\mu}\})
			\mapsto
				\dirlim_{\mu} \Hom_{k^{\perf}_{\et}}(D, C_{\mu})
		\]
	($\op$ means the opposite category)
	extends to a triangulated bifunctor
		\[
					K^{-}(k^{\perf}_{\et})^{\op}
				\times
					K^{+} \bigl(
						\Ab(k^{\perf}_{\et})(M)
					\bigr)
			\to
				K^{+}(\Ab)
		\]
	by \cite[Proposition 11.6.4 (i)]{KS06}.
	Evaluating the left variable at any bounded above complex $D$ in $\Ab(k^{\perf}_{\et})$,
	we have a triangulated functor
		\[
				K^{+} \bigl(
					\Ab(k^{\perf}_{\et})(M)
				\bigr)
			\to
				K^{+}.
		\]
	This admits a right derived functor
		\[
				D^{+} \bigl(
					\Ab(k^{\perf}_{\et})(M)
				\bigr)
			\to
				D^{+}(\Ab),
			\quad
				\{C_{\mu}\}
			\mapsto
				\dirlim_{\mu}
					R \Hom_{k^{\perf}_{\et}}(D, C_{\mu})
		\]
	by \cite[Proposition 13.2.3]{KS06},
	whose $i$-th cohomology is
		$
			\dirlim_{\mu}
				H^{i} R \Hom_{k^{\perf}_{\et}}(D, C_{\mu})
		$.
	We have a morphism
		\[
				\dirlim_{\mu}
					R \Hom_{k^{\perf}_{\et}}(D, C_{\mu})
			\to
				R \Hom_{k^{\perf}_{\et}}(D, \dirlim_{\mu} C_{\mu})
		\]
	by universality.
	We consider the cubical construction $Q(A)$ for $A$ viewed as an object of $\DGAb(k^{\perf}_{\et})$,
	which is the \'etale (rather than pro-fppf) sheafification of the presheaf $R \mapsto Q(A(R))$.
	Let $\{C_{\mu}\} \in \Ab(k^{\perf}_{\et})(M)$ with $\dirlim_{\mu} C_{\mu} = 0$ and set
		\begin{gather*}
					X^{j}
				=
					\dirlim_{\mu}
						R \Hom_{k^{\perf}_{\et}}(H_{j} Q(A), C_{\mu}),
			\\
					Y^{j}
				=
					\dirlim_{\mu}
						R \Hom_{k^{\perf}_{\et}}(\tau_{\le -j} Q(A), C_{\mu}).
		\end{gather*}
	We know $Y^{0} = 0$ from
		\[
				\dirlim_{\mu}
					H^{j}(A^{2^{i}}, C_{\mu})
			=
				H^{j}(A^{2^{i}}, 0)
			=
				0
		\]
	by the same argument as the proof of the previous assertion.
	For a fixed $n \ge 0$, assume $\tau_{\le n} X^{0} = 0$ for any $A \in \Pro \Alg / k$.
	Then
		\[
				\tau_{\le n} \dirlim_{\mu}
				R \Hom_{k^{\perf}_{\et}}(A[l], C_{\mu})
			=
				0
		\]
	for any $l \ge 1$,
	where $A[l] \in \Pro \Alg / k$ is the $l$-torsion part.
	This and $\tau_{\le n} X^{0} = 0$ imply
		\[
				\tau_{\le n} \dirlim_{\mu}
				R \Hom_{k^{\perf}_{\et}}(A /_{\et} l, C_{\mu})
			=
				0,
		\]
	where $A /_{\et} l \in \Ab(k^{\perf}_{\et})$ is the cokernel of multiplication by $l$ on $A$
	in $\Ab(k^{\perf}_{\et})$
	(which might not be in $\Pro \Alg / k$
	since $\Pro \Alg / k\ \to \Ab(k^{\perf}_{\et})$ is only left exact).
	From these, we can deduce $\tau_{\le n + 1} X^{0} = 0$
	by the same argument as the proof of the previous assertion.
	Induction then finishes the proof.
\end{proof}

With this lemma,
the statements of the proposition are reduced to showing the following parts
	\begin{equation} \label{eq: reduced version of Ext as sheaves and as algebraic groups}
		\begin{gathered}
					\Ext_{k^{\perf}_{\et}}^{n}(A, B)
				=
					\Ext_{\Alg / k}^{n}(A, B),
			\\
					\Ext_{k^{\perf}_{\et}}^{n}(A, \Q)
				=
					0
		\end{gathered}
	\end{equation}
for $A, B \in \Alg / k$.
(Note that $\Hom_{k^{\perf}_{\et}}(A, \Q) = 0$ implies that its subgroup
$\Hom_{k^{\perf}_{\et}}(A, \Z)$ is zero.)

Next we will reduce these statements
to the case that $k$ is algebraically closed.
Let $A \in \Alg / k$ and $B \in \Loc \Alg / k$.
Assertion \eqref{ass: Ext is continuous on A and k} of the above lemma implies that
	\[
			\Ext_{\bar{k}^{\perf}_{\et}}^{j}(A, B)
		=
			\dirlim_{k' / k}
				\Ext_{k'^{\perf}_{\et}}^{j}(A, B),
	\]
where $k'$ runs through all finite Galois subextensions of $\bar{k} / k$.
Therefore the Hochschild-Serre spectral sequences for finite Galois extensions
give an isomorphism
	\begin{equation} \label{eq: Hochschild-Serre for perfect etale Rhom}
			R \Gamma \bigl(
				\Gal(\bar{k} / k),
				R \Hom_{\bar{k}^{\perf}_{\et}}(A, B)
			\bigr)
		=
			R \Hom_{k^{\perf}_{\et}}(A, B)
	\end{equation}
in $D(\Ab)$.
Hence the second line of \eqref{eq: reduced version of Ext as sheaves and as algebraic groups}
is reduced to the case $\bar{k} = k$.

For the first line, assume $B \in \Alg / k$.
Then \cite[\S 1, Proposition]{Mil70} (or its proof) shows that
the functor $\Hom_{\Pro \Alg / \bar{k}}(\;\cdot\;, B)$ takes projectives in $\Pro \Alg / k$
to $\Gamma(\Gal(\bar{k} / k), \;\cdot\;)$-acyclics.
Hence we have an isomorphism
	\begin{equation} \label{eq: Hochschild-Serre for proalgebraic Rhom}
			R \Gamma \bigl(
				\Gal(\bar{k} / k),
				R \Hom_{\Pro \Alg / \bar{k}}(A, B)
			\bigr)
		=
			R \Hom_{\Pro \Alg / k}(A, B)
	\end{equation}
in $D(\Ab)$.
The functor $\Pro \Alg / k \to \Ab(k^{\perf}_{\pro\fppf})$ is exact as we saw in the proof of the previous lemma.
This induces morphisms
	\begin{gather*}
				R \Hom_{\Pro \Alg / k}(A, B)
			\to
				R \Hom_{k^{\perf}_{\pro\fppf}}(A, B),
		\\
				R \Hom_{\Pro \Alg / \bar{k}}(A, B)
			\to
				R \Hom_{\bar{k}^{\perf}_{\pro\fppf}}(A, B)
	\end{gather*}
in $D(\Ab)$, $D(\Gal(\bar{k} / k))$ ($=$ the derived category of discrete $\Gal(\bar{k} / k)$-modules),
respectively.
The right-hand sides are isomorphic to
$R \Hom_{k^{\perf}_{\et}}(A, B)$,
$R \Hom_{\bar{k}^{\perf}_{\et}}(A, B)$, respectively,
by Proposition \ref{prop: RHom for profppf, perf etale and indrat etale}
and Corollary \ref{cor: profppf and etale cohomology of an algebraic group}.
Hence we have morphisms
	\begin{gather*}
				R \Hom_{\Pro \Alg / k}(A, B)
			\to
				R \Hom_{k^{\perf}_{\et}}(A, B),
		\\
				R \Hom_{\Pro \Alg / \bar{k}}(A, B)
			\to
				R \Hom_{\bar{k}^{\perf}_{\et}}(A, B)
	\end{gather*}
in $D(\Ab)$, $D(\Gal(\bar{k} / k))$,
respectively.
Applying $R \Gamma(\Gal(\bar{k} / k), \;\cdot\;)$ to the second morphism results the first
by \eqref{eq: Hochschild-Serre for perfect etale Rhom}
and \eqref{eq: Hochschild-Serre for proalgebraic Rhom}.
Therefore if we show that
	\[
			R \Hom_{\Pro \Alg / \bar{k}}(A, B)
		=
			R \Hom_{\bar{k}^{\perf}_{\et}}(A, B),
	\]
then
	\[
			R \Hom_{\Pro \Alg / k}(A, B)
		=
			R \Hom_{k^{\perf}_{\et}}(A, B).
	\]
Noting that $\Ext_{\Pro \Alg / k}^{n}(A, B) = \Ext_{\Alg / k}^{n}(A, B)$,
we are reduced to the case $\bar{k} = k$.

We show \eqref{eq: reduced version of Ext as sheaves and as algebraic groups}
in the case $\bar{k} = k$.

\begin{Lem} \label{lem: higher Ext are torsion}
	Let $A, B \in \Alg / k$ be connected.
	Then the group $\Ext_{k^{\perf}_{\et}}^{i}(A, B)$ consists of
	$p$-power-torsion elements for any $i \ge 2$.
	In particular, we have
	$\Ext_{k^{\perf}_{\et}}^{i}(\Gm, \Gm) = 0$ for all $i \ge 2$.
\end{Lem}

\begin{proof}
	Let $A_{0}, B_{0}$ be commutative affine algebraic groups over $k$
	whose perfections are $A, B$, respectively.
	Let $k^{\sch}$ be the category of affine $k$-schemes
	and $\Spec k^{\sch}_{\et}$ the \'etale site on it.
	By \cite[Lemmas 1.1-1.3]{Bre81} and their proofs,
	we have
		\[
				\Ext_{k^{\perf}_{\et}}^{i}(A, B)
			=
				\Ext_{k^{\sch}_{\et}}^{i}(A_{0}, \dirlim_{n} B_{0}^{(-n)}),
		\]
	where $B_{0} \to B_{0}^{(-1)} \to B_{0}^{(-2)} \to \cdots$ are the Frobenius morphisms over $k$.
	Assertion \eqref{ass: Ext is continuous on B} of Lemma \ref{lem: lim and Ext commutes}
	is also true for $\Spec k^{\sch}_{\et}$
	by exactly the same proof.
	Hence we have
		\[
				\Ext_{k^{\perf}_{\et}}^{i}(A, B)
			=
				\dirlim_{n}
					\Ext_{k^{\sch}_{\et}}^{i}(A_{0}, B_{0}^{(-n)}).
		\]
	We want to show that
	$\Ext_{k^{\sch}_{\et}}^{i}(A_{0}, B_{0}^{(-n)})$ consists of $p$-power-torsion elements for $i \ge 2$.
	It is enough to treat the case $n = 0$
	since $B_{0}^{(-n)}$ is also a commutative affine algebraic group over $k$.
	Since the fppf cohomology with coefficients in a smooth group scheme agrees with the \'etale cohomology,
	the Grothendieck spectral sequence shows that
	$\Ext_{k^{\sch}_{\et}}^{i}(A_{0}, B_{0})$ is equal to $\Ext_{k_{\fppf}}^{i}(A_{0}, B_{0})$.
	By d\'evissage,
	we may assume that $A_{0}$ and $B_{0}$ are $\Ga$ or $\Gm$.
	Since $\Ga$ in characteristic $p > 0$ is killed by $p$,
	it is enough to treat the case $A_{0} = B_{0} = \Gm$.
	By \cite[\S 10, 6th paragraph (p.43)]{Bre70},
	we know that $\Ext_{k_{\fppf}}^{i}(\Gm, \Gm)$ consists of $p$-power-torsion elements for $i \ge 2$.
	(As the proof of loc.\ cit.\ shows,
	the meaning of an abelian group being ``at most $p$-torsion'' in the terminology of \cite{Bre70} is
	that it consists of $p$-power-torsion elements.)
	This proves the desired property of $\Ext_{k^{\perf}_{\et}}^{i}(A, B)$.
	On the other hand,
	the $p$-th power map on the group $\Gm$ over $k^{\perf}_{\et}$ is an isomorphism.
	Hence $\Ext_{k^{\perf}_{\et}}^{i}(\Gm, \Gm) = 0$ for all $i \ge 2$.
\end{proof}

\begin{proof}[Proof of Proposition \ref{prop: Ext as sheaves and as algebraic groups}]
	Let $A, B \in \Alg / k$.
	We first prove
		\[
				\Ext_{k^{\perf}_{\et}}^{n}(A, B)
			=
				\Ext_{\Alg / k}^{n}(A, B)
		\]
	for $A, B \in \Alg / k$ and $n \ge 0$.
	The case $n = 0, 1$ is classical
	(apply the direct limit in Frobenii to \cite[Corollary 17.5]{Oor66} and
	argue as in the proof of the previous lemma).
	We may assume that each of $A$ and $B$ is $\Ga$, $\Gm$, or finite.
	The case $A$ is finite is easy.
	The case $B$ is finite is reduced to the case $B$ is connected affine
	by embedding a finite $B$ into a connected affine group.
	There are no higher extensions between $\Ga$ (killed by $p$)
	and $\Gm$ (having invertible $p$-th power map).
	For the case $A = B = \Gm$ and $n \ge 2$, the right-hand side is zero
	since the category of quasi-algebraic groups of multiplicative type is equivalent to
	the category of finitely generated $\Z[1 / p]$-modules \cite[\S 7.2, Proposition 1, 2]{Ser60}.
	The left-hand side is also zero by Lemma \ref{lem: higher Ext are torsion}.
	For the case $A = B = \Ga$ and $n \ge 2$, the right-hand side is zero
	by \cite[\S 8.6, Corollaries 4 and 5 to Proposition 6]{Ser60}.
	The left-hand side is also zero by \cite[Corollary 1.7]{Bre81}.
	
	Next we prove
		$
				\Ext_{k^{\perf}_{\et}}^{n}(A, \Q)
			=
				0
		$
	for $n \ge 0$.
	This is trivial if $A$ is finite.
	Hence we may assume that $A$ is connected.
	Let $M(A)$ be Mac Lane's resolution of $A$ in $\Ab(k^{\perf}_{\et})$.
	Consider the hyperext spectral sequence
		\[
				E_{1}^{i j}
			=
				\Ext_{k^{\perf}_{\et}}^{j}(M_{i}(A), \Q)
			\Longrightarrow
				\Ext_{k^{\perf}_{\et}}^{i + j}(A, \Q).
		\]
	Each term of $M(A)$ is a direct summand of a direct sum of sheaves of the form $\Z[A^{n}]$
	for various $n$.
	Hence $\Ext_{k^{\perf}_{\et}}^{j}(M_{i}(A), \Q)$
	is a direct factor of a direct product of abelian groups of the form
	$H^{j}(A^{n}_{\et}, \Q)$.
	If $j \ge 1$, this is zero since
	the \'etale cohomology of a normal scheme with coefficients in $\Q$ vanishes in non-zero degree
	by \cite[\S 2.1]{Den88}.
	If $j = 0$, this is $\Q = H^{0}(0, \Q)$ since $A$ is assumed to be connected.
	Therefore
		\[
				\Ext_{k^{\perf}_{\et}}^{n}(A, \Q)
			=
				\Ext_{k^{\perf}_{\et}}^{n}(0, \Q)
			=
				0.
		\]
\end{proof}

\begin{Rmk} \label{rmk: generalizing to abelian varieties}
	Results of this section can be extended for abelian varieties instead of affine proalgebraic groups.
	This can be done as follows.
	Let $\Alg' / k$ be the category of (not necessarily affine) quasi-algebraic groups over $k$.
	Consider its procategory $\Pro \Alg' / k$
	(which is Serre's category of proalgebraic groups over $k$ \cite[\S 2.6, Proposition 12]{Ser60}).
	Let $\Pro' \Alg' / k$ be the full subcategory of $\Pro \Alg' / k$
	consisting of extensions of perfections of abelian varieties
	by affine proalgebraic groups.
	This category contains projective envelopes of perfections of abelian varieties
	by \cite[\S 9.2, Proposition 4]{Ser60}.
	It follows that any object $G \in \Pro' \Alg' / k$ has a projective resolution
	$G_{\bullet}$ in $\Pro \Alg' / k$ such that $G_{n} \in \Pro' \Alg' / k$
	(hence projective also in $\Pro' \Alg' / k$) for any $n$.
	We use $\Pro' \Alg' / k$ instead of the category $\Pro \Alg / k$ of affine proalgebraic groups.
	
	Let $k^{\perf \prime}$ be the category of quasi-compact quasi-separated perfect schemes over $k$.
	We use $k^{\perf \prime}$ instead of the category $k^{\perf}$ of perfect $k$-algebras.
	The category $k^{\perf \prime}$ contains the perfection $A$ of an abelian variety over $k$,
	its generic point $\xi_{A}$ and the inclusion morphism $\xi_{A} \into A$.
	A morphism $Y \to X$ in $k^{\perf \prime}$ is said to be \emph{flat of finite presentation} (in the perfect sense)
	if it can be written as the perfection of a $k$-scheme morphism $Y_{0} \to X$
	flat of finite presentation in the usual sense.
	A morphism $Y \to X$ in $k^{\perf \prime}$ is said to be \emph{flat of profinite presentation}
	if it can be written as the inverse limit of a filtered inverse system $\{Y_{\lambda} \to X\}$
	of morphisms in $k^{\perf \prime}$ flat of finite presentation
	such that the transition morphisms $Y_{\mu} \to Y_{\lambda}$ are affine.
	A finite family $\{Y_{i} \to X\}$ of morphisms in $k^{\perf \prime}$ is called
	an \emph{fppf} (resp.\ \emph{pro-fppf}) \emph{covering}
	if each $Y_{i} \to X$ is a flat morphism of finite (resp.\ profinite) presentation
	and $\bigsqcup Y_{i} \to X$ is surjective.
	This defines an \emph{fppf} (resp.\ \emph{pro-fppf}) \emph{site}
	$\Spec k^{\perf \prime}_{\fppf}$ (resp.\ $\Spec k^{\perf \prime}_{\pro\fppf}$) on the category $k^{\perf \prime}$.
	The \'etale site $\Spec k^{\perf \prime}_{\et}$ on $k^{\perf \prime}$
	can be defined in the usual way,
	noting that an \'etale scheme over a perfect $k$-scheme is again perfect.
	The identity functor defines a morphism of sites
	$\Spec k^{\perf \prime}_{\et} \to \Spec k^{\perf}_{\et}$,
	which induces an equivalence on the topoi.
	
	Then all the definitions, statements and proofs in this section before Lemma \ref{lem: higher Ext are torsion}
	can be generalized with $\Pro \Alg / k$ replaced by $\Pro' \Alg' / k$
	and $k^{\perf}$ replaced by $k^{\perf \prime}$.
	For Lemma \ref{lem: higher Ext are torsion},
	we need to prove the following generalization:
	the group $\Ext_{k^{\perf \prime}_{\et}}^{i}(A, B)$ consists of
	$p$-power-torsion elements for any $i \ge 2$
	and any connected $A, B \in \Alg' / k$.
	To prove this, by the same argument as the original proof,
	it is enough to show that
	$\Ext_{k_{\fppf}}^{i}(A_{0}, B_{0})$ consists of $p$-power-torsion elements for any $i \ge 2$,
	where $A_{0}$ and $B_{0}$ are abelian varieties or $\Gm$.
	This follows from the same result \cite[\S 10, 6th paragraph (p.43)]{Bre70} of Breen.
	
	For the paragraph after Lemma \ref{lem: higher Ext are torsion},
	we need to prove that
		\[
				\Ext_{k^{\perf \prime}_{\et}}^{n}(A, B)
			=
				\Ext_{\Alg' / k}^{n}(A, B)
		\]
	for $A, B \in \Alg' / k$ and $n \ge 0$.
	The case $n = 0, 1$ is again classical
	(this time using the not-necessarily-affine version \cite[Remark after Corollary 17.5]{Oor66}).
	We have $\Ext_{\Alg' / k}^{n}(A, B) = 0$ for $n \ge 3$ by \cite[\S 10.1, Theorem 1]{Ser60}.
	As noted after the proof of the cited theorem,
	the group $\Ext_{\Alg' / k}^{2}(A, B)$ is zero if $A$ and $B$ are elementary
	in the sense of \cite[\S 3.2, Definition 1]{Ser60}
	except the case $A = \Ga$ and $B = \Z / p \Z$,
	where $\Ext_{\Alg' / k}^{2}(\Ga, \Z / p \Z)$ is killed by $p$.
	By d\'evissage, this implies that $\Ext_{\Alg' / k}^{2}(A, B)$ is torsion for any $A, B \in \Alg' / k$.
	By the same argument as the proof of the original Lemma \ref{lem: higher Ext are torsion},
	we may assume that $A$ and $B$ are $\Ga$, $\Gm$ or perfections of abelian varieties.
	First, let $A$ be connected affine and $B$ the perfection of a semi-abelian variety.
	Then for any $n, m \ge 1$, the Kummer sequence $0 \to B[m] \to B \to B \to 0$ induces
	exact sequences
		\begin{gather*}
					0
				\to
					\Ext_{\Alg' / k}^{n - 1}(A, B) / m
				\to
					\Ext_{\Alg' / k}^{n}(A, B[m])
				\to
					\Ext_{\Alg' / k}^{n}(A, B)[m]
				\to
					0,
			\\
					0
				\to
					\Ext_{k^{\perf \prime}_{\et}}^{n - 1}(A, B) / m
				\to
					\Ext_{k^{\perf \prime}_{\et}}^{n}(A, B[m])
				\to
					\Ext_{k^{\perf \prime}_{\et}}^{n}(A, B)[m]
				\to
					0,
		\end{gather*}
	where $/ m$ denotes the cokernel of multiplication by $m$.
	We have a natural morphism from the first sequence to the second.
	As we saw, the morphism on the first terms of these sequences is an isomorphism if $n \le 2$.
	The group $B[m]$ is finite and hence affine.
	Hence the affine case implies that the morphism on the second terms is an isomorphism for any $n$.
	The group $\Ext_{k^{\perf \prime}_{\et}}^{n}(A, B)$ is torsion for $n \ge 2$
	by Lemma \ref{lem: higher Ext are torsion} generalized right above.
	Hence we can inductively prove that
	the morphism on the third terms of these sequences is an isomorphism for any $n$.
	Second, let $A$ be the perfection of a semi-abelian variety and $B$ connected affine.
	Then the same argument using the exact sequences
		\begin{gather*}
					0
				\to
					\Ext_{\Alg' / k}^{n - 1}(A, B) / m
				\to
					\Ext_{\Alg' / k}^{n - 1}(A[m], B)
				\to
					\Ext_{\Alg' / k}^{n}(A, B)[m]
				\to
					0,
			\\
					0
				\to
					\Ext_{k^{\perf \prime}_{\et}}^{n - 1}(A, B) / m
				\to
					\Ext_{k^{\perf \prime}_{\et}}^{n - 1}(A[m], B)
				\to
					\Ext_{k^{\perf \prime}_{\et}}^{n}(A, B)[m]
				\to
					0
		\end{gather*}
	gives the desired result in this case.
	Finally, let $A$ and $B$ be perfections of semi-abelian varieties.
	Then using either one of the above two sets of sequences reduces the statement
	to the case where $A$ or $B$ is affine.
	This finishes the case of arbitrary $A, B \in \Alg' / k$.
	The final paragraph of the proof of Proposition \ref{prop: Ext as sheaves and as algebraic groups}
	needs no change.
\end{Rmk}


\section{Comparison with Jannsen-Rovinsky's dominant topology}
\label{sec: comparison with Jannsen-Rovinsky's dominant topology}

Let $\DM_{k}$ be the category of the perfections of smooth $k$-morphisms
of smooth affine $k$-schemes with the topology
where a cover is a dominant morphism.
In \cite{JR10}, Jannsen and Rovinsky defined and used the site $\DM_{k}$
(without taking perfections or affineness).
Its topos is equivalent to the category of sets with
smooth actions by the field automorphism group of a universal domain over $k$,
which Rovinsky introduced in \cite{Rov05}.
Their purpose was to study motives.
In this section, we explain a relation between the site $\DM_{k}$
and our sites $\Spec k^{\rat}_{\et}$, $\Spec k^{\perf}_{\pro\fppf}$.
We hope our sites and Theorem \ref{thm: main theorem, comparison of Ext}
are useful for the study of motives and the site $\DM_{k}$.

Note that the scheme-theoretic fiber product $Y \times_{X} Z$ for $X, Y, Z \in \DM_{k}$
(which is always in $\DM_{k}$)
gives the fiber product in the category $\DM_{k}$ when $Y \to X$ or $Z \to X$ is \'etale,
but not always.
Morphisms in $\DM_{k}$ are restricted to be smooth.
A little more precise definition of the topology of $\DM_{k}$ is that
a sieve on an object $X \in \DM_{k}$ is a covering
if it contains a (finite) family of morphisms $X_{i} \to X$
such that $\bigsqcup X_{i} \to X$ is dominant.
A presheaf $F$ on $\DM_{k}$ is a sheaf if and only if
it sends disjoint unions to direct products and
the sequence $F(X) \to F(Y) \rightrightarrows F(Y \times_{X} Y)$ is exact
for any dominant morphism $Y \to X$ in $\DM_{k}$.

\subsection{The Zariski and \'etale cases}
\label{sec: The Zariski and etale cases}

Let $\DM_{k, \et}$ (resp.\ $\DM_{k, \zar}$) be the same category as $\DM_{k}$
with the topology where a cover is a dominant \'etale morphism (resp.\ a dominant open immersion).
Let $\Spec k^{\rat}_{\zar}$ be the category $k^{\rat}$ with the topology
where a cover is the identity map
(so a presheaf $F$ is a sheaf if and only if
$F(k_{1} \times \dots \times k_{n}) = F(k_{1}) \times \dots \times F(k_{n})$
for fields $k_{1}, \dots k_{n} \in k^{\rat}$).
Note that for a perfect scheme $X$ essentially of finite type over $k$,
its generic point $\xi_{X}$ of $X$ (Definition \ref{def: points in profinite setting}) is
the $\Spec$ of a rational $k$-algebra.

\begin{Prop}
	The functor taking an object of $\DM_{k}$ to its generic point
	defines morphisms of sites
		\[
				f \colon \Spec k^{\rat}_{\zar} \to \DM_{k, \zar},
			\quad
				g \colon \Spec k^{\rat}_{\et} \to \DM_{k, \et},
		\]
	which induce equivalences on the associated topoi.
\end{Prop}

\begin{proof}
	First we treat $f$.
	A presheaf $F$ on $\DM_{k, \zar}$ is a sheaf
	if and only if $F(X \sqcup Y) = F(X) \times F(Y)$ for $X, Y \in \DM_{k, \zar}$
	and $F(X) \isomto F(Y)$ for a dominant open immersion $Y \to X$ in $\DM_{k, \zar}$.
	The pullback $f^{\ast} F$ for $F \in \Set(\DM_{k, \zar})$
	is given by $(f^{\ast} F)(x) = F(X)$ for any $x \in k^{\rat}$,
	where $X \in \DM_{k, \zar}$ is any object with $\xi_{X} = x$.
	Obviously $f_{\ast} f^{\ast} = \id$ and $f^{\ast} f_{\ast} = \id$.
	Hence $f$ induces an equivalence on the topoi.
	
	Hence for $g$, we only need to show that
	$f^{\ast}$ maps the subcategory $\Set(\DM_{k, \et})$ to $\Set(k^{\rat}_{\et})$.
	Let $F \in \Set(\DM_{k, \et})$.
	Let $y \to x$ be a faithfully flat \'etale morphism in $k^{\rat}$.
	We want to show that the sequence
		\[
				f^{\ast} F(x)
			\to
				f^{\ast} F(y)
			\rightrightarrows
				f^{\ast} F(y \times_{x} y)
		\]
	is exact.
	Take a dominant \'etale morphism $Y \to X$ in $\DM_{k, \et}$
	whose associated morphism $\xi_{Y} \to \xi_{X}$ gives $y \to x$.
	Then $\xi_{Y \times_{X} Y} = y \times_{x} y$.
	Since $F \in \Set(\DM_{k, \et})$,
	the sequence
		\[
				F(X)
			\to
				F(Y)
			\rightrightarrows
				F(Y \times_{X} Y)
		\]
	is exact.
	This sequence is identical to the above sequence.
\end{proof}

\begin{Rmk}
	The composite
	$\Spec k^{\perf}_{\et} \to \Spec k^{\rat}_{\et} \to \DM_{k, \et}$
	is defined by the functor sending $X \in \DM_{k}$ to its generic point $\xi_{X}$.
	This functor is different from the identity functor $X \mapsto X$.
	The identity functor does not define a continuous map
	$\Spec k^{\perf}_{\et} \to \DM_{k, \et}$.
\end{Rmk}


\subsection{The flat case}
\label{sec: the flat case}

We define a site $\Spec \tilde{k}^{\rat}_{\fl}$ and compare it with $\DM_{k}$.
Let $\tilde{k}^{\rat}$ be the full subcategory of the category of perfect $k$-schemes
consisting of (not necessarily finite) disjoint unions of
the $\Spec$'s of finitely generated perfect fields over $k$.
The fiber product of any objects $y, z$ over any object $x$ in $\tilde{k}^{\rat}$ exists.
It is given by the disjoint union of the points
of the underlying set of the usual fiber product $y \times_{x} z$ as a scheme.
We denote it by $y \tilde{\times}_{x} z$.
We endow the category $\tilde{k}^{\rat}$ with the topology
where a covering is a faithfully flat morphism.
The resulting site is denoted by $\Spec \tilde{k}^{\rat}_{\fl}$.
We call this the \emph{rational flat site} of $k$.
We denote $y \times_{x}^{\gen} z := \xi_{y \times_{x} z}$.

\begin{Prop} \label{prop: sheaf condition in the rational flat site}
	In order for a presheaf $F$ on $\Spec \tilde{k}^{\rat}_{\fl}$ to be a sheaf,
	it is necessary and sufficient that the following two conditions be satisfied:
	\begin{enumerate}
		\item
			For any fields $k_{\lambda} \in k^{\rat}$ and $x_{\lambda} = \Spec k_{\lambda}$,
			we have $F(\bigsqcup x_{\lambda}) = \prod F(x_{\lambda})$.
		\item
			For any field extension $k'' / k'$ in $k^{\rat}$ and $x = \Spec k'$, $y = \Spec k''$,
			the sequence
				\[
						F(x)
					\to
						F(y)
					\overset{p_{1}}{\underset{p_{2}}{\rightrightarrows}}
						F(y \times_{x}^{\gen} y)
				\]
			is exact,
			where $p_{1}$ and $p_{2}$ are given by the first and the second projections
			$y \times_{x}^{\gen} y \rightrightarrows y$.
	\end{enumerate}
\end{Prop}

\begin{proof}
	Clearly these conditions are sufficient.
	For necessity, let $F \in \Set(\tilde{k}^{\rat}_{\fl})$.
	The first condition is clear.
	For the second, let $s \in F(y)$ satisfy $p_{1}(s) = p_{2}(s)$ in $F(y \times_{x}^{\gen} y)$.
	We need to show that for any connected component $z \subset y \tilde{\times}_{x} y$
	(or in other words, a point of $y \times_{x} y$),
	we have $p_{1}(s) = p_{2}(s)$ in $F(z)$.
	Since $F \in \Set(\tilde{k}^{\rat}_{\fl})$,
	the projection $z \times_{x}^{\gen} y \to z$ induces
	an injection $F(z) \into F(z \times_{x}^{\gen} y)$.
	Hence it is enough to show that
	the elements $p_{1}(s)$ and $p_{2}(s)$ map to the same element of $F(z \times_{x}^{\gen} y)$.
	The third projection $p_{3} \colon z \times_{x}^{\gen} y \to y$ sends $s \in F(y)$
	to another element $p_{3}(s) \in F(z \times_{x}^{\gen} y)$.
	We show that $p_{1}(s) = p_{3}(s)$ and $p_{2}(s) = p_{3}(s)$ in $F(z \times_{x}^{\gen} y)$.
	The morphism $(p_{1}, p_{3}) \colon z \times_{x} y \to y \times_{x} y$ is flat.
	Hence this induces a morphism
	$(p_{1}, p_{3}) \colon z \times_{x}^{\gen} y \to y \times_{x}^{\gen} y$.
	Since $p_{1}(s) = p_{2}(s)$ in $F(y \times_{x}^{\gen} y)$ by assumption,
	we have $p_{1}(s) = p_{3}(s)$ in $F(z \times_{x}^{\gen} y)$.
	Similarly we have $p_{2}(s) = p_{3}(s)$ in $F(z \times_{x}^{\gen} y)$.
	Hence $p_{1}(s) = p_{2}(s)$ in $F(z \times_{x}^{\gen} y)$.
\end{proof}

\begin{Prop}
	The functor taking an object of $\DM_{k}$ to its generic point
	defines a morphism of sites
		\[
			h \colon \Spec \tilde{k}^{\rat}_{\fl} \to \DM_{k},
		\]
	which induces an equivalence on the associated topoi.
\end{Prop}

\begin{proof}
	It is enough to show that $f_{\ast}$ (= $h_{\ast}$) and $f^{\ast}$
	used for the Zariski sites send sheaves to sheaves.
	If $F \in \Set(\tilde{k}^{\rat}_{\fl})$
	and $Y \to X$ dominant in $\DM_{k}$,
	then $\xi_{Y} \to \xi_{X}$ is faithfully flat in $k^{\rat}$.
	Therefore Proposition \ref{prop: sheaf condition in the rational flat site} implies that
	the sequence
	$F(\xi_{X}) \to F(\xi_{Y}) \rightrightarrows F(\xi_{Y \times_{X} Y})$
	is exact.
	Hence $f_{\ast} F \in \Set(\DM_{k})$.
	Conversely, let $F \in \Set(\DM_{k})$,
	$k'' / k'$ a field extension in $k^{\rat}$,
	$x = \Spec k'$ and $y = \Spec k''$.
	Take a dominant morphism $Y \to X$ in $\DM_{k}$
	whose associated morphism $\xi_{Y} \to \xi_{X}$ gives $y \to x$.
	Then $\xi_{Y \times_{X} Y} = y \times_{x}^{\gen} y$.
	The sheaf condition for $F \in \Set(\DM_{k})$ says that
	the sequence $F(X) \to F(Y) \rightrightarrows F(Y \times_{X} X)$ is exact.
	This sequence is identical to the sequence
	$(f^{\ast} F)(x) \to (f^{\ast} F)(y) \rightrightarrows (f^{\ast} F)(y \times_{x}^{\gen} y)$.
	Hence $f^{\ast} F \in \Set(\tilde{k}^{\rat}_{\fl})$
	by Proposition \ref{prop: sheaf condition in the rational flat site}.
\end{proof}


\subsection{The dominant topology and the pro-fppf topology}
\label{sec: The dominant topology and the pro-fppf topology}

Next we relate $\Spec \tilde{k}^{\rat}_{\fl}$ and $\DM_{k}$
to a variant of $\Spec k^{\perf}_{\pro\fppf}$.
Let $k^{\perf\fin}$ be the full subcategory of $k^{\perf}$
consisting of the perfections of $k$-algebras essentially of finite type.
Let $\Spec k^{\perf\fin}_{\pro\fppf}$ be the site
obtained by restricting $\Spec k^{\perf}_{\pro\fppf}$ to $k^{\perf\fin}$.
(In the category $k^{\perf\fin}$, every flat morphism is of profinite presentation,
and the pro-fppf and fpqc topologies coincide,
but we do not need this fact.)
We need a generic variant of the \v{C}ech complex.

\begin{Prop} \label{prop: acyclicity of the generic Cech complex}
	Let $k'' / k'$ be a field extension in $k^{\rat}$ and $x = \Spec k'$, $y = \Spec k''$.
	Let $\check{C}^{\gen}(y / x)$ be the complex
		\begin{equation} \label{eq: generic Cech complex}
				\cdots
			\to
				\Z[y \times_{x}^{\gen} y \times_{x}^{\gen} y]
			\to
				\Z[y \times_{x}^{\gen} y]
			\to
				\Z[y]
		\end{equation}
	($\Z[y]$ placed in degree $0$)
	in $\Ab(\tilde{k}^{\perf\fin}_{\pro\fppf})$ with differential given by
	the usual formula for \v{C}ech cohomology.
	Give it the augmentation $\check{C}^{\gen}(y / x) \to \Z[x]$
	and denote the resulting complex by $\check{C}_{0}^{\gen}(y / x)$
	(so $\Z[x]$ is placed in degree $1$).
	Then $\check{C}_{0}^{\gen}(y / x)$ is acyclic.
\end{Prop}

We need a lemma.
To simplify the notation,
we write $y_{x}^{n} := y \times_{x} \dots \times_{x} y$ ($n$-fold product) and
$y_{x}^{n, \gen} := y \times_{x}^{\gen} \dots \times_{x}^{\gen} y$.
Note that $y_{x}^{0} = y_{x}^{0, \gen} = x$.

\begin{Lem}
	Let $y / x$ as above.
	Let $X \in k^{\perf\fin}$, $Y = y \times_{k} X$ and $n \ge 0$.
	\begin{enumerate}
		\item \label{ass: generifying Cech cells}
			Let $a \in y_{x}^{n, \gen}(X)$ and
			$U a = (\id, a) \in (y \times_{x} y_{x}^{n, \gen})(Y)$ the natural extension of $a$.
			Then there exists $Z \in k^{\perf}$ and a morphism $Z \to Y$
			satisfying the two conditions of Lemma \ref{lem: making fpqc covers}
			such that the natural image $U a \in (y \times_{x} y_{x}^{n, \gen})(Z)$ is contained in the subset
			$y_{x}^{n + 1, \gen}(Z)$.
		\item \label{ass: generifying Cech cycles}
			Let $t \in \Z[y_{x}^{n, \gen}(X)]$ and
			$U t \in \Z[(y \times_{x} y_{x}^{n, \gen})(Y)]$ the natural extension of $t$.
			Then there exists $Z \in k^{\perf}$ and a morphism $Z \to Y$
			with the composite $Z \to Y \to X$ faithfully flat of profinite presentation
			such that the natural image $U t \in \Z[(y \times_{x} y_{x}^{n, \gen})(Z)]$ is contained in the subset
			$\Z[y_{x}^{n + 1, \gen}(Z)]$.
	\end{enumerate}
\end{Lem}

\begin{proof}
	\eqref{ass: generifying Cech cells}
	Define $Z$ by the following cartesian diagram:
		\[
			\begin{CD}
					Z
				@>>>
					Y
				@>> \text{proj} >
					X
				\\
				@VVV
				@V U a VV
				@V a VV
				\\
					y_{x}^{n + 1, \gen}
				@> \text{incl} >>
					y \times_{x} y_{x}^{n, \gen}
				@> \text{proj} >>
					y_{x}^{n, \gen}
			\end{CD}
		\]
	Then we can check the requirements for $Z$ in the same way
	as the proof of Lemma \ref{lem: generic splitting}.
	
	\eqref{ass: generifying Cech cycles}
	We can deduce this from \eqref{ass: generifying Cech cells} in the same way we deduced
	Lemma \ref{lem: generic homotopy} from Lemma \ref{lem: generic splitting}.
\end{proof}

\begin{proof}[Proof of Proposition \ref{prop: acyclicity of the generic Cech complex}]
	We use the same notation as the lemma.
	Let $n \ge 0$ and $a \in y_{x}^{n, \gen}(X)$
	(which in particular gives an element of $x(X)$, or a morphism $X \to x$).
	Write $a = (a_{0}, \dots, a_{n - 1}) \in y_{x}^{n}(X)$ with
	$a_{i} \in y(X)$.
	Define $U a = (c, a_{0}, \dots, a_{n - 1}) \in y_{x}^{n + 1}(Y)$,
	where $c$ corresponds to the identity map $y \to y$.
	Then we have $\boundary U a + U \boundary a = a$ in $\Z[y_{x}^{n, \gen}(Z)]$,
	where $\boundary$ is the boundary map for the \v{C}ech complex.
	By linearity, we have
	$\boundary U t + U \boundary t = t$ in $\Z[y_{x}^{n, \gen}(Z)]$
	for any $t \in \Z[y_{x}^{n, \gen}(X)]$.
	Therefore if $t \in \Z[y_{x}^{n, \gen}(X)]$ is a cycle,
	it is a boundary of a chain in $\Z[y_{x}^{n + 1, \gen}(Z)]$.
	Since $Z$ is faithfully flat of profinite presentation over $X$,
	this proves the result.
\end{proof}

\begin{Prop} \BetweenThmAndList
	\begin{enumerate}
		\item \label{ass: comparison of profppf and ratfl: continuous}
			The identity functor defines a continuous map
				\[
					f \colon \Spec k^{\perf\fin}_{\pro\fppf} \to \Spec \tilde{k}^{\rat}_{\fl}.
				\]
		\item \label{ass: comparison of profppf and ratfl: exact}
			The functor $f_{\ast}$ is exact.
		\item \label{ass: comparison of profppf and ratfl: identity}
			We have $f_{\ast} f^{\ast} = \id$.
		\item \label{ass: comparison of profppf and ratfl: fully faithful}
			The functor $f^{\ast} \colon \Set(\tilde{k}^{\rat}_{\fl}) \to \Set(k^{\perf\fin}_{\pro\fppf})$
			is fully faithful.
		\item \label{ass: comparison of profppf and ratfl: inj to flabby}
			The functor $f_{\ast} \colon \Ab(k^{\perf\fin}_{\pro\fppf}) \to \Ab(\tilde{k}^{\rat}_{\fl})$
			sends acyclic sheaves (see the first paragraph of Section \ref{sec: The ind-rational etale site})
			to acyclic sheaves.
		\item \label{ass: comparison of profppf and ratfl: higher push is zero}
			We have
				\[
						R \Gamma(\tilde{k}'^{\rat}_{\fl}, f_{\ast} A)
					=
						R \Gamma(X, f_{\ast} A)
					=
						R \Gamma(k'^{\perf\fin}_{\pro\fppf}, A)
				\]
			for any $A \in \Ab(k^{\perf\fin}_{\pro\fppf})$,
			$k' \in \tilde{k}^{\rat}_{\fl}$ and
			$X \in \DM_{k}$ with $\xi_{X} = k'$.
		\item \label{ass: comparison of profppf and ratfl: rational flat and etale cohomology}
			Let $g \colon \Spec \tilde{k}^{\rat}_{\fl} \to \Spec k^{\rat}_{\et}$
			be the morphism defined by the identity
			and $B \in \Alg / k$.
			Then we have $R g_{\ast} B = B$.
		\item \label{ass: comparison of profppf and ratfl: Theorem B for the rational flat site}
			Theorem \ref{thm: main theorem, comparison of Ext} also holds
			when replacing $\Spec k^{\rat}_{\et}$ with
			$\Spec \tilde{k}^{\rat}_{\fl}$ or $\DM_{k}$.
	\end{enumerate}
\end{Prop}

Actually we need to extend $k^{\perf}$ to allow infinite disjoint unions.
Alternatively, we could define the site $\Spec \tilde{k}^{\rat}_{\fl}$
to be the category of finitely generated perfect fields over $k$
endowed with the topology whose covering sieves are all non-empty sieves.
These modifications do not change the associated topoi,
so we ignore this issue.

\begin{proof}
	\eqref{ass: comparison of profppf and ratfl: continuous}
	We need to show that $f_{\ast}$ sends sheaves to sheaves.
	Let $F \in \Set(k^{\perf\fin}_{\pro\fppf})$.
	Let $k'' / k'$ be a field extension in $k^{\rat}$
	and $x = \Spec k'$, $y = \Spec k''$.
	By Proposition \ref{prop: sheaf condition in the rational flat site},
	it is enough to show that the sequence
	$F(x) \to F(y) \rightrightarrows F(y \times_{x}^{\gen} y)$
	is exact.
	Let $s \in F(y)$ be such that $p_{1}(s) = p_{2}(s)$ in $F(y \times_{x}^{\gen} y)$,
	where $p_{i}$ is the $i$-th projection.
	To show $s \in F(x)$,
	it is enough to show that $p_{1}(s) = p_{2}(s)$ in $F(y \times_{x} y)$
	since $F \in \Set(k^{\perf\fin}_{\pro\fppf})$.
	Let $X = y \times_{x} y$, $Y = y \times_{x} X = y \times_{x} y \times_{x} y$,
	$Z_{1} = (y \times_{x}^{\gen} y) \times_{x} y \to Y$.
	Let $Z_{2} \to Y$ be obtained from $Z_{1} \to Y$
	by flipping the last two factors $y \times_{x} y$ of $Y$.
	Let $Z = Z_{1} \times_{Y} Z_{2}$.
	Let $p_{0}, p_{1}, p_{2} \colon Y = y \times_{x} y \times_{x} y \to y$ be the projections.
	In $F(Z_{1})$, we have $p_{0}(s) = p_{1}(s)$.
	In $F(Z_{2})$, we have $p_{0}(s) = p_{2}(s)$.
	Hence in $F(Z)$, we have $p_{1}(s) = p_{0}(s) = p_{2}(s)$.
	Hence it is enough to show that $Z / X$ is faithfully flat of profinite presentation.
	We check the two conditions of Lemma \ref{lem: making fpqc covers}
	for $Z_{1}$ and $Z_{2}$.
	It is enough to check them for $Z_{1}$ only.
	The flatness of $Z_{1} \to Y$ is obvious.
	Take a point $x_{0}$ of $X = y \times_{x} y$.
	Then we have a commutative diagram
		\[
			\begin{CD}
					y \times_{x}^{\gen} x_{0}
				@>>>
					y \times_{x} x_{0} = Y_{x_{0}}
				\\
				@VVV
				@VVV
				\\
					(y \times_{x}^{\gen} y) \times_{x} y = Z_{1}
				@>>>
					y \times_{x} y \times_{x} y = Y.
			\end{CD}
		\]
	Hence the fiber $(Z_{1})_{x_{0}}$ contains the generic point of $Y_{x_{0}}$.
	Hence the morphism $(Z_{1})_{x_{0}} \to Y_{x_{0}}$ is dominant.
	This completes the proof of \eqref{ass: comparison of profppf and ratfl: continuous}.
	
	\eqref{ass: comparison of profppf and ratfl: exact}
	It is enough to show that for any field $k' \in k^{\rat}$
	and a covering $X \to x$ in $\Spec k^{\perf\fin}_{\pro\fppf}$ with $x = \Spec k'$,
	there exists a covering $y \to x$ in $\Spec \tilde{k}^{\rat}_{\fl}$
	and an $x$-morphism $y \to X$.
	We can take such a $y$ to be any point of $X \ne \emptyset$.
	
	\eqref{ass: comparison of profppf and ratfl: identity}
	The argument done right above shows that
	$f_{\ast}$ commutes with sheafification.
	Let $f^{-1}$ be the pullback for presheaves of sets.
	Clearly $f_{\ast} f^{-1} = \id$.
	Combining these two facts, we have $f_{\ast} f^{\ast} = f_{\ast} f^{-1} = \id$.
	
	\eqref{ass: comparison of profppf and ratfl: fully faithful}
	This follows from \eqref{ass: comparison of profppf and ratfl: identity}.
	
	\eqref{ass: comparison of profppf and ratfl: inj to flabby}
	First note that
	the complex of the form \eqref{eq: generic Cech complex}
	in Proposition \ref{prop: acyclicity of the generic Cech complex}
	remains isomorphic to $\Z[x]$
	even if it is considered in $\Ab(\tilde{k}^{\rat}_{\fl})$
	by \eqref{ass: comparison of profppf and ratfl: exact} and
	\eqref{ass: comparison of profppf and ratfl: identity}.
	We denote this complex in $\Ab(\tilde{k}^{\rat}_{\fl})$
	by the same symbol $\check{C}^{\gen}(y / x)$.
	
	Let $x = \Spec k'$ and $I \in \Ab(k^{\perf\fin}_{\pro\fppf})$ acyclic.
	We want to show that $H^{j}(x^{\rat}_{\fl}, f_{\ast} I) = 0$ for $j \ge 1$.
	What we just saw above gives a spectral sequence
		\[
				E_{2}^{i j}
			=
				\dirlim_{y / x}
				H^{i} H^{j} \bigl(
					(y_{x}^{\gen, \,\cdot\,})^{\rat}_{\fl}, f_{\ast} I
				\bigl)
			\Longrightarrow
				H^{i + j}(x^{\rat}_{\fl}, f_{\ast} I)
		\]
	where $y = \Spec k''$ and $k'' \in k^{\rat}$ runs over all field extensions of $k'$,
	and the $H^{i}$ above is  the $i$-th cohomology of the complex
	$\Ext_{\tilde{k}^{\rat}_{\fl}}^{j}(\check{C}^{\gen}(y / x), f_{\ast} I)$
	of term-wise Ext groups.
	We have $E_{2}^{0 j} = 0$ for $j > 0$ since it is
	the $\dirlim_{y / x}$ of the kernel of
		\[
				H^{j}(y^{\rat}_{\fl}, f_{\ast} I)
			\to
				H^{j}((y \times_{x}^{\gen} y)^{\rat}_{\fl}, f_{\ast} I).
		\]
	For $E_{2}^{i 0}$ for $i > 0$, note that
	it is the $i$-th cohomology of the complex
		\[
				\Hom_{\tilde{k}^{\rat}}(\check{C}^{\gen}(y / x), f_{\ast} I)
			=
				\Hom_{k^{\perf\fin}}(\check{C}^{\gen}(y / x), I).
		\]
	Since $I$ is an acyclic sheaf, we have
		\[
				\Ext_{k^{\perf\fin}_{\pro\fppf}}^{i}(\Z[z], I)
			=
				H^{i}(z^{\perf\fin}_{\pro\fppf}, I)
			=
				0
		\]
	for any $z \in k^{\rat}$ and $i > 0$.
	Hence the isomorphism $\check{C}^{\gen}(y / x) = \Z[x]$
	(Proposition \ref{prop: acyclicity of the generic Cech complex}) implies 
	the acyclicity of $\Hom_{k^{\perf\fin}}(\check{C}^{\gen}(y / x), I)$ in positive degrees,
	so $E_{2}^{i 0} = 0$ for $i > 0$.
	Hence the above spectral sequence implies the result by induction on $j$.
	
	\eqref{ass: comparison of profppf and ratfl: higher push is zero}
	The property \eqref{ass: comparison of profppf and ratfl: inj to flabby}
	is precisely what we need to apply the Grothendieck spectral sequence
	for $R \Gamma$ (that is, the Leray spectral sequence).
	With \eqref{ass: comparison of profppf and ratfl: exact},
	the result follows.
	
	\eqref{ass: comparison of profppf and ratfl: rational flat and etale cohomology}
	This follows from \eqref{ass: comparison of profppf and ratfl: higher push is zero}
	and Corollary \ref{cor: profppf and etale cohomology of an algebraic group}.
	
	\eqref{ass: comparison of profppf and ratfl: Theorem B for the rational flat site}
	This follows from \eqref{ass: comparison of profppf and ratfl: rational flat and etale cohomology}.
\end{proof}


\end{document}